\documentclass[a4paper,11pt]{article}
\usepackage{cite}
\usepackage{listings}
\usepackage{latexsym,amssymb,enumerate,amsmath,epsfig,amsthm,dsfont,bm}
\usepackage{mathrsfs}
\usepackage[margin=1in]{geometry}
\usepackage{setspace,color}
\usepackage{tikz}
\usepackage{floatrow}
\usepackage{graphicx,subfigure}
\usepackage[ruled]{algorithm2e}
\usepackage{epstopdf}
\usepackage[multidot]{grffile}
\usepackage{comment}
\usepackage{multirow}
\usepackage[colorlinks, linkcolor=blue, anchorcolor=blue, citecolor=blue]{hyperref}
\usepackage{lineno}

\newcommand{\x}{\mathbf{x}}

\newcommand{\R}{\mathbb{R}}

\newcommand{\bfpsi}{\bm{\psi}}

\newcommand{\tr}[1]{\textcolor{blue}{#1}}

\makeatletter
\newcommand\figcaption{\def\@captype{figure}\caption}
\newcommand\tabcaption{\def\@captype{table}\caption}
\makeatother

\theoremstyle{plain}
\newtheorem{Th}{Theorem}[section]

\newtheorem{Prop}[Th]{Proposition}

\theoremstyle{definition}

\newtheorem{Rem}[Th]{Remark}
\newtheorem{?}[Th]{Problem}

\newcommand{\bx}{\mathbf{x}}
\newcommand{\bb}{\mathbf{b}}
\newcommand{\bp}{\mathbf{p}}
\newcommand{\bq}{\mathbf{q}}
\newcommand{\by}{\mathbf{y}}

\DeclareMathOperator*{\argmin}{arg\,min}

\begin{document}
	\title{Numerical Identification of Nonlocal Potential in  
		Aggregation}
	\author{Yuchen He\thanks{Institute of Natural Sciences, Shanghai Jiao Tong University, Shanghai, China. Email: yuchenroy@sjtu.edu.cn}
		, Sung Ha Kang\thanks{School of Mathematics, Georgia Institute of Technology, Atlanta, Georgia, GA 30332-0160, USA. Email: kang@math.gatech.edu}
		, Wenjing Liao\thanks{School of Mathematics, Georgia Institute of Technology, Atlanta, Georgia, GA 30332-0160, USA. Email: wliao60@gatech.edu}
		, Hao Liu\thanks{Hong Kong Baptist University, Hong Kong SAR. Email: haoliu@hkbu.edu.hk} 
		, Yingjie Liu\thanks{School of Mathematics, Georgia Institute of Technology, Atlanta, Georgia, GA 30332-0160, USA. Email: yingjie@math.gatech.edu}}
	\date{}
	\maketitle
	\begin{abstract}
		Aggregation equations are broadly used to model population dynamics with nonlocal interactions, characterized by a potential in the equation. This paper considers the inverse problem of identifying the potential from a single noisy spatial-temporal process. The identification is challenging in the presence of noise due to the instability of numerical differentiation.
		We propose a robust model-based technique to identify the potential by minimizing a regularized data fidelity term, and regularization is taken as the total variation and the squared Laplacian.  A split Bregman method is used to solve the regularized optimization problem. Our method is robust to noise by utilizing a Successively Denoised Differentiation technique. We consider additional constraints such as compact support and symmetry constraints to enhance the performance further.   We also apply this method to identify time-varying potentials and identify the interaction kernel in an agent-based system.   Various numerical examples in one and two dimensions are included to verify the effectiveness and robustness of the proposed method.
	\end{abstract}

\section{Introduction}

Nonlocal Partial Differential Equations (PDE) are often used to model dynamics with nonlocal interactions. They have wide applications in neuronal networks \cite{bressloff2001traveling}, biological aggregation \cite{parrish1997animal} and material science \cite{bates2002spectral}.  In neuronal networks, nonlocal PDEs are used to describe the dynamics of excitatory neurons' local activities in the cortex, where the nonlocal term models the connection strength between neurons \cite{bressloff2001traveling}. In biological aggregation, the population density of fish schools can be modeled by a nonlocal PDE \cite{parrish1997animal}, where the nonlocal term describes the long-range attraction and short-range repulsion.

In this paper, we consider the aggregation equation
\begin{align}
	u_t+\nabla\cdot\left(u \, \mathbf{p}\right)=0, \text{ with  } \mathbf{p}=-\nabla (\phi* u)
	\label{eq.aggregation}
\end{align}
where $\phi$ is a potential (also known as the kernel), $\phi* u$ denotes the convolution of $\phi$ and $u$.  This equation has broad applications in physics and biology. In granular materials, (\ref{eq.aggregation}) is used to characterize the dynamics of kinetic models \cite{caglioti2002homogeneous}. In biology, the evolution of swarming can be described by (\ref{eq.aggregation})  in which the potential $\phi$ represents the long-range attraction, and short-range repulsion between individuals \cite{topaz2004swarming}. In particular, the authors in \cite{morale2005interacting}  show that starting from an Eulerian description of an attraction-repulsion dynamical system, 
as the number of individuals goes to infinity, the dynamical system converges to (\ref{eq.aggregation}) which describes the evolution of the mean-field spatial density of the population. In bacterial chemotaxis, the convolution $\phi* u$ represents the concentration of chemoattractant which is emitted by bacteria and used to interact with other individuals \cite{keller1970initiation}. A popular model in the kinetic aspect for this dynamics is the Othmer--Dunbar--Alt system whose hydrodynamic limit is (\ref{eq.aggregation}) \cite{dolak2005kinetic}. Other applications can be found in particle assembly \cite{holm2006formation}, opinion dynamics \cite{motsch2014heterophilious} and pattern formation\cite{balague2013dimensionality}.

Although (\ref{eq.aggregation}) has been successfully applied to model dynamics in different fields, its solution may blow up in the evolution process. It has been shown that, even with a smooth initial condition, when the potential has a Lipschitz point at the origin, a weak solution of (\ref{eq.aggregation}) may always concentrate and become a Dirac function in a finite time, which is known as the finite-time blow-up solution \cite{bertozzi2009blow}.  Here, the potential having a Lipschitz point means that the  potential is Lipschitz but has a singular point.
This finite-time blow-up behavior of solutions brings difficulties in solving (\ref{eq.aggregation}) numerically, especially near the blow-up time.  In \cite{huang2010self}, the authors use a characteristic method to solve an equivalent coupled ODE system with potential $\phi=|x|$ in various dimensions. The particle method is studied in \cite{carrillo2011global} which enables one to track the behavior of solutions after the blow-up time.

In literature, most existing works focus on the mathematical theories on the existence and regularity of the solution, or the numerical solvers of (\ref{eq.aggregation}) with a given potential. The inverse problem of identifying the potential from a given solution has not been widely studied in comparison with the forward problem.  The identification of the potential from the steady-state solution is considered in \cite{fetecau2011swarm}, where finding the underlying potential amounts to solving a time-independent nonlocal PDE. In \cite{you2020data,you2020data1}, the authors consider learning the potential in a non-local linear PDE from high-fidelity data. The potential is represented as a linear combination of Bernstein polynomials, and the polynomial coefficients are recovered from an optimization problem solved by the Adam optimizer and L-BFGS.
In \cite{bongini2017inferring,lu2019nonparametric}, a variational method is introduced to estimate the kernel from the trajectory data in a dynamical system of agents, and a statistical theoretical guarantee
is established in \cite{lu2021learning}.
The inverse problem of parameter estimation in aggregation-diffusion equations is considered in \cite{huang2019learning}, where the diffusion parameter estimation 
is studied subject to the Newtonian aggregation and Brownian diffusion. In \cite{huang2019learning}, the potential is known, and only the diffusion parameter is to be estimated.  


In this paper, we study the inverse problem of potential identification in aggregation equations. Given a noisy data set governed by an aggregation equation, we aim to numerically identify the underlying potential $\phi$. In comparison with the aforementioned works \cite{fetecau2011swarm,you2020data,you2020data1,bongini2017inferring,lu2019nonparametric,huang2019learning}, we utilize a small amount of noisy data from a single realization of the PDE.  This work is motivated by a series of works, such as \cite{brunton2016discovering,schaeffer2017learning,rudy2017data,he2020robust, kang2019ident} where the objective is to identify a parametric PDE (or dynamical system) from a single set of time-dependent noisy data.   The PDE identification with noisy data is particularly challenging due to the instability of numerical differentiation.
While the identification methods in \cite{he2020robust, kang2019ident} can handle a wide range of PDEs and a considerable amount of noise, the extension to non-local PDEs is not trivial.
A non-local PDE such as (\ref{eq.aggregation}) requires a different identification approach.

We propose identifying the potential by minimizing a functional regularized by a total variation term and a Laplacian term. 
A split Bregman method is used to solve the optimization problem efficiently.
We utilize a Successively Denoised Differentiation technique \cite{he2020robust} to stabilize numerical differentiation so that the proposed method is robust to noise.  We consider additional constraints such as the compact support and symmetry constraints to enhance the performance further.  The proposed method can be extended to identifying time-varying potentials from agent-based data.
The agent-based data are simulated according to certain interaction rules \cite{reynolds1987flocks}, instead of solving the aggregation equation.
Our method can identify a potential with which the solution of (\ref{eq.aggregation}) approximates the evolution of the agent density.

This paper is organized as follows: We present our identification method in Section \ref{sec.model}. The numerical scheme and discretization details are given in Section \ref{sec_method}. Some techniques to improve the robustness are presented in Section \ref{sec.stability}. Our numerical experiments are shown in Section \ref{sec_num}. We discuss extensions to the estimation of time-varying potentials and potentials from agent-based data in Section \ref{sec_time} and Section \ref{sec.agentbased}, respectively.
We conclude this paper in Section \ref{sec.conclusion}.

\paragraph{Notation:} In this paper, we use regular lowercase letters to denote scalars and bold lowercase letter to denote vectors. Uppercase letters are used to denote operators and matrices. We use $B(\mathbf{a},r)$ to denote the Euclidean ball centered at $\mathbf{a}$ with radius $r$.

\section{The proposed method: identification of nonlocal potential}\label{sec.model}

In this section, we describe our method to identify the potential  from a single set of noisy data.
%
We assume the continuous PDE solution as our measured data, and propose our method with the continuous data.  The discretization setting is discussed in Section \ref{sec.dis}--\ref{sec.discretize.Bregman}.

Let $u:[0,T]\times\mathbb{R}^d\to\mathbb{R}$ be a solution of (\ref{eq.aggregation}) in which $\phi$ is an unknown potential with a compact support  in $\mathbb{R}^d$. For any $t\in[0,T]$, assume $u(t,\bx)$ has a compact support in $\mathbb{R}^d$. Denote our spatial computational domain by $\Omega$ which contains the support of $\phi$ and $u(t,\cdot)$ for any $t\in[0,T]$. 
Given a set of noisy discretized data of $u$, we aim to identify the spatially dependent potential $\phi$.  We further consider  time and spatially dependent potentials in Section \ref{sec_time}.  
The equation (\ref{eq.aggregation}) is linear in $\phi$, which can be written as
\begin{align}
	u_t=\nabla\cdot\left(u(\nabla (\phi*u))\right)\equiv L_u\phi\;,\label{eq_operator}
\end{align}
where  the linear operator $L_u:\phi\mapsto\nabla\cdot\left(u(\nabla (\phi*u))\right)$ depends on the solution $u$.

To estimate $\phi$, we propose to minimize the following functional:
\begin{align}
	\frac{1}{2}\int_0^T \int_{\Omega} (u_t-L_u\phi)^2 d\bx dt+  \alpha \int_{\Omega} |\nabla\phi| d\bx + \frac{\beta}{2} \int_{\Omega} |\nabla^2\phi|^2 d\bx,
	\label{eq.min.reg}
\end{align}
where $d\bx=dx_1\cdots dx_d$, $\nabla^2 \phi$ denotes the Laplacian of $\phi$, and $\alpha,\beta\geq 0$ are two weight parameters. Here $|\cdot|$ denotes the isotropic $L^1$ norm:
$$
|\nabla \phi|= \sqrt{\phi_{x_1}^2+\cdots +\phi_{x_d}^2} \mbox{ with } \phi_{x_j}=\partial \phi/\partial x_j.
$$
The minimizer of (\ref{eq.min.reg}) is the identified potential by our method. In (\ref{eq.min.reg}), the first term is a fidelity term representing the residual.
The second term gives a Total Variation (TV) regularization, which is popular in image processing \cite{rudin1992nonlinear,strong2003edge}. It is well known that this term helps to remove oscillations and keeps sharp changes in the gradient. In aggregation equations, many potential functions have singularities \cite{huang2010self}. The TV term helps to keep such features while the noise is suppressed. However, the TV term itself may produce undesired staircase effects \cite{maso2009higher,papafitsoros2014combined}. The third term in (\ref{eq.min.reg}) is the square of the Laplacian of $\phi$, which helps to ameliorate the staircase phenomenon. 
A model similar to (\ref{eq.min.reg}) is explored for image segmentation in \cite{cai2013two,li2020three}, where great performance has been demonstrated.
The effects of these two and other regularization terms are explored and compared extensively in Section \ref{sec_num}. 

The two regularization terms $|\nabla \phi|$ and $|\nabla^2\phi|^2$ in (\ref{eq.min.reg}) have their physical meanings: a \textit{bounded power} and a \textit{finite flux}, respectively. Since $\phi$ is the interaction potential, its gradient $\nabla \phi$ gives the associated force field whose value at $\bx$ specifies the force from the individual at $\bx$ to the origin.  The \textit{power}~\cite[Chapter 6]{halliday2013fundamentals} of moving an individual at $\bx$  towards $\by\neq \bx$ with speed $v>0$ is defined as $p_v(\bx,\by) := v\nabla \phi(\by-\bx)\cdot\frac{(\by-\bx)}{|\by-\bx|}$. For any bounded Borel set $\Gamma$, $\overline{p}_v(\bx) =  \frac{1}{|\Gamma|}\int_{\Gamma} p_v(\bx,\by)\,d\by$ measures the average power of moving an individual away from $\bx$ to any location within $\Gamma$, where  $|\Gamma|$ denotes the Lebesgue measure of $\Gamma$. The condition $\int_\Gamma |\nabla\phi(\by-\bx)|d\by<\infty$ thus implies that  $\overline{p}_v(\bx)$ is bounded for any finite moving speed $v$. Consider the total force received by the individual at $\bx$ from the neighbors on a sphere $\partial B(\bx,r)$ for some small $r>0$, i.e., $\int_{\partial B(\bx,r)}\nabla \phi(\by-\bx)\,ds(\by)$, where $ds$ denotes the differential of surface area. By the divergence theorem, $\int_{\partial B(\bx,r)}\nabla \phi(\by-\bx)\,ds(\by) = \int_{B(\bx,r)}\nabla^2\phi(\by-\bx)\,d\by$, which is bounded by $C|\nabla^2\phi|^2$ for some constant $C$. As a result, the second regularization term gives a finite flux.

The functional in \eqref{eq.min.reg} is well-defined for appropriate function spaces. We first introduce some related notations. Suppose $\Omega\subset\mathbb{R}^d$  is a bounded, open, and connected subset of $\mathbb{R}^d$ with Lipschitz boundary. Let $H^k(\Omega)$ be the Sobolev space of order $k$.  We use conventional notations $H_0^k(\Omega)$ for the $H^k$-closure  of smooth functions vanishing at $\partial \Omega$, and $\bigotimes^d H^k(\Omega)$ is the product space such that every $F \in \bigotimes^dH^k(\Omega)$ has the form $F=(f_1,\dots,f_d)$, $f_i\in H^k(\Omega)$ for $i=1,\dots,d$.  Denote $H^1(0,T; H^k(\Omega))$ for $k\geq 0$ as the space of functions $f:[0,T]\times \Omega\to\mathbb{R}$ that $f(\cdot,\bx)\in H^1([0,T])$ for any fixed $\bx\in\Omega$, and $f(t,\cdot)\in H^k(\Omega)$ for any fixed $t\in[0,T]$. For any $u\in  H^1(0,T; H^k(\Omega))$, we use $u_t\in L^2([0,T]\times\Omega)$ to denote its weak time derivative.  We set $H_0^1(\Omega)$ as the domain of the linear operator $L_u$ in~\eqref{eq_operator} where the spatial gradient $\nabla$ and divergence $\nabla\cdot$ are defined in the weak sense.

We take $\phi\in H_0^1(\Omega)$ and $u\in H^1([0,T]; H_0^k(\Omega))$. Assume $k>(d+2)/4$. For any $t\in[0,T]$,  we have $u*\nabla\phi(t,\cdot)\in \bigotimes^d H^k(\mathbb{R}^d)$ and $u(u*\nabla\phi)(t,\cdot)\in \bigotimes^d H^1(\mathbb{R}^d)$ \cite{behzadan2015multiplication}.  Therefore, the range of the operator $L_u$ is  contained in $ L^2([0,T]\times\mathbb{R}^d)$. The fidelity  term as well as the TV regularization term in~\eqref{eq.min.reg} is then well-defined. Furthermore, we assume that  $\nabla\phi \in \bigotimes^d H^1(\Omega\setminus \bar{B}(\mathbf{0},\varepsilon))$ for any closed ball $\bar{B}(\mathbf{0},\varepsilon)$ centered at the origin with radius $\varepsilon>0$.  As the second order weak derivatives of $\phi$ at $0$ may not exist, e.g., Morse potential~\cite{d2006self}, the second regularization term is understood as $\lim_{\varepsilon\to 0^+}\left(\int_{\Omega\setminus \bar{B}(\mathbf{0},\varepsilon)}(\nabla\cdot\nabla\phi)^2\,dx\right)$. The set
\begin{align}
	P = \left\{\phi \in H^1_0(\Omega)\;:\;\nabla\phi\in \bigcap_{\varepsilon>0}\left(\bigotimes^d H^1(\Omega\setminus \bar{B}_{\varepsilon})\right)\;,\;\;\lim_{\varepsilon\to 0^+}\left(\int_{\Omega\setminus \bar{B}(\mathbf{0},\varepsilon)}(\nabla^2\phi)^2\,dx\right)<\infty\right\}\label{eq_space}
\end{align}
forms a reflexive Banach space with the norm $\|\phi\|_P:= \|\phi\|_{H_0^1(\Omega)}+\lim_{\varepsilon\to 0^+}\|\nabla\cdot\nabla\phi\|_{L^2(\Omega\setminus\bar{B}(\mathbf{0},\varepsilon))}$ (Proposition~\ref{prop_Banach}). By the direct method  \cite[Section 3.2]{scherzer2009variational}, we conclude

\begin{Th}\label{th_ex}[Existence and Uniqueness] Let  $T>0$ and $\Omega\subset\mathbb{R}^d$ be an  open, bounded domain with Lipschitz boundary. Assume $k>(d+2)/4$.  For any $u\in H^1([0,T];H_0^k(\Omega))$,   the functional~\eqref{eq.min.reg} admits a unique minimizer in the reflexive Banach space defined in~\eqref{eq_space} in the supplementary material.
\end{Th}

We consider time-dependent data of the aggregation equation from a single initial condition.  Instead of utilizing many realizations of the PDE from multiple initial conditions, we adopt this setting for practical considerations: (i) The dynamics of different populations may follow different potentials, and it is better not to combine data sets from different populations.  (ii) It is challenging to conduct different experiments on the same group of wild animals that it is more practical to consider a single realization of the PDE.   Since our data set is from a single initial condition that may contain noise, we tackle these difficulties by imposing additional constraints, such as regularity via successively denoised differentiation, (adaptive) compact support constraint, and symmetric constraint, as discussed in the following sections.
\begin{Rem}
	The problem setting of this paper is related to but different from many inverse problems. For example, in transmission travel-time tomography  \cite{leung2006adjoint, sei1995convergent,taillandier2009first,leung2021level}, the objective is to recover a velocity function in the Eikonal equation from the first-arrival travel-time measurements on the final time data.  One can only access the PDE solution on the time (or spatial) boundary. These inverse problems are typically formularized as a PDE constrained optimization problem. The time (or spatial) boundary sets are collected using multiple source locations for a robust recovery. In this paper, we consider a single process (with a single initial condition) of time-dependent data collected at every spatial and temporal grid point of the domain.
\end{Rem}	

\section{The proposed numerical scheme}\label{sec_method}

The split Bregman method~\cite{goldstein2009split} is a popular iterative algorithm, which has been successfully applied in image processing~\cite{papafitsoros2014combined} with mixed regularization terms. 
In this paper, we use the split Bregman method to design an iterative numerical scheme to minimize (\ref{eq.min.reg}).

We first introduce a vector-valued variable $\bfpsi$ such that (\ref{eq.min.reg}) is equivalent to the following constrained minimization problem
\begin{align}
	\begin{cases}
		{\displaystyle  \min_{\phi,\bfpsi} \left[\frac{1}{2}\int_0^T\int_{\Omega} |u_t-L_u\phi|^2d\bx dt+ \alpha \int_{\Omega}|\bfpsi| d\bx +\frac{\beta}{2}\int_{\Omega} |\nabla^2\phi|^2d\bx\right]},\\
		\bfpsi=\nabla\phi.
	\end{cases}
	\label{eq.min.constraint}
\end{align}
By introducing an additional penalty to quantify the mismatch between $\nabla\phi$ and $\bfpsi$, \eqref{eq.min.constraint} can be approximated by the following unconstrained problem
\begin{align}
	\min_{\phi,\bfpsi} \left[\frac{1}{2}\int_0^T\int_{\Omega} |u_t-L_u\phi|^2d\bx dt+ \alpha \int_{\Omega}|\bfpsi| d\bx +\frac{\beta}{2}\int_{\Omega} |\nabla^2\phi|^2d\bx+ \frac{\lambda}{2} \int_{\Omega} |\nabla\phi-\bfpsi|^2d\bx\right],
	\label{eq.min.unconstraint}
\end{align}
where $\lambda>0$ is a weight parameter. 

We introduce an auxiliary variable $\bb$ in the same space as $\bfpsi$.
We solve (\ref{eq.min.unconstraint}) using Bregman iterations.
We set
$ (\phi^0,\bfpsi^0,\bb^0)=(\phi_0,\bfpsi_0,\bb_0)$ as the initial, and
update $(\phi^k,\bfpsi^k,\bb^k)$ to $(\phi^{k+1},\bfpsi^{k+1},\bb^{k+1})$ as follows:
\begin{align}
	(\phi^{k+1},\bfpsi^{k+1})&=\argmin_{\phi,\bfpsi}\bigg[\frac{1}{2}\int_0^T \int_{\Omega}|u_t-L_u\phi|^2d\bx dt+\alpha \int_{\Omega}|\bfpsi| d\bx\nonumber \\
	&\hspace{2cm}
	+ \frac{\beta}{2}\int_{\Omega} |\nabla^2\phi|^2d\bx + \frac{\lambda}{2}\int_{\Omega} |\bb^k+ \nabla\phi-\bfpsi|^2 d\bx\bigg], \label{eq.Bregman0.1}\\
	\bb^{k+1} &=  \bb^k+\nabla \phi^{k+1}-\bfpsi^{k+1}. \label{eq.Bregman0.2}
\end{align}
It is difficult to solve (\ref{eq.Bregman0.1}) directly. In this paper we adopt the operator-splitting method \cite{glowinski2019finite,liu2020color,he2020curvature,liu2021operator}. We refer the readers to \cite{glowinski2017splitting} for a detailed discussion on the operator-splitting method.  We update $\phi^{k+1},\bfpsi^{k+1}$ as
\begin{align}
	&\phi^{k+1}=\argmin_{\phi}\Bigg[\frac{1}{2}\int_0^T \int_{\Omega} |u_t-L_u\phi|^2d\bx dt 
	+ \frac{\beta}{2}\int_{\Omega} |\nabla^2\phi|^2d\bx \nonumber\\
	&\hspace{3cm} + \frac{\lambda}{2}\int_{\Omega} |\bb^k + \nabla\phi-\bfpsi^k|^2 d\bx\Bigg], \label{eq.split.1}\\
	&\bfpsi^{k+1}=\argmin_{\bfpsi} \left[\alpha \int_{\Omega}|\bfpsi| d\bx+ \frac{\lambda}{2}\int_{\Omega} |\bb^k+ \nabla\phi^{k+1}-\bfpsi|^2 d\bx\right]. \label{eq.split.2}
\end{align}
The explicit formulas for $\phi^{k+1}$ and $\bfpsi^{k+1}$ are derived as follows: We denote the adjoint operator of $L_u$ by $L^*_u$. According to the Euler-Lagrange equation of (\ref{eq.split.1}), we obtain the following optimality condition for $\phi^{k+1}$:
\begin{align}
	\int_0^T \left(L^*_uL_u\phi^{k+1} -L^*_uu_t \right)dt+\beta \nabla^4\phi^{k+1} +\lambda (\nabla\cdot (\bfpsi^k-\bb^k)-\nabla^2\phi^{k+1})=0,
	\label{eq.Bregman.1}
\end{align}
which is linear in $\phi^{k+1}$ and therefore can be easily solved. Here $\nabla^4=\nabla^2\circ \nabla^2$.
For (\ref{eq.split.2}), we have the closed form solution $\bfpsi$ using the shrinkage operator \cite{donoho1995noising}
\begin{align}
	\bfpsi^{k+1}=\max\left( 0, 1-\frac{\alpha}{\lambda |\bp|}\right)\bq \quad \mbox{with} \quad  \bq=\bb^k+ \nabla\phi^{k+1}.
	\label{eq.Bregman.2}
\end{align}
The above procedure is repeated until 	
\begin{align}
	\|\phi^{k+1}-\phi^k\|_{\infty}<\varepsilon\label{eq_terminate}
\end{align}
for some small $\varepsilon>0$.
This iterative algorithm is summarized in Algorithm \ref{alg1}.
\begin{algorithm}[t]
	\SetKwInOut{KwIni}{Initialization}
	\KwIn{$\phi^0,\bfpsi^0,\bb^0$, parameters $\alpha,\beta,\lambda,\varepsilon$.}
	\While{(\ref{eq_terminate}) is not satisfied}{
		\textbf{Step 1}: Update $\phi^{k+1}$ by solving (\ref{eq.Bregman.1}).\\
		\textbf{Step 2}: Update $\bfpsi^{k+1}$ according to (\ref{eq.Bregman.2}). \\
		\textbf{Step 3}: Update $\bb^{k+1}$ according to (\ref{eq.Bregman0.2}).
	}
	\KwOut{Identified potential $\phi^{k}$.}
	\caption{Potential identification scheme}
	\label{alg1}
\end{algorithm}

\begin{Rem}
	Our algorithm can be applied to identify the potential in a large class of nonlocal PDEs. As long as the PDE is linear in the potential, one can always formularize the problem as minimizing a functional in the form of (\ref{eq.min.reg}) and apply Algorithm \ref{alg1}.
\end{Rem} 


\subsection{Numerical discretization}\label{sec.dis}

We present the discretized algorithm for $d=1$. The algorithm in high dimensions can be derived analogously.
We choose $L$ large enough such that $u$ and $\phi$ are supported in $\Omega=[-L,L]$. The spatial domain $[-L,L]$ is discretized by equidistant nodes $\{x_i\}_{i=-M}^M$ with step size $\Delta x=L/M$. The time domain $[0,T]$ is discretized by equidistant nodes $\{t^n\}_{n=0}^N$ with step size $\Delta t=T/N$. Denote $u_i^n=u(t^n, x_i), u^n=u(t^n,x)$ and $\phi_i=\phi(x_i)$. The given data set is
\begin{align}
	\mathcal{U}=\{U_{i}^n: i=-M,...,M, \ n=0,...,N\}, \text{ where each datum is }	U_i^n=u_i^n+\varepsilon_i^n,
	\label{eq.data.noise}
\end{align}
with $\varepsilon_i^n$ being some random noise with mean 0. We denote the set of given data at time $t=t^n$ by $U^n=\{U_i^n: i=-M,...,M\}$.

For any function $v(t,x)$, we define the forward (+) and backward (-) approximation of $\partial v/\partial x$ as
\begin{align*}
	(D_x^+ v)_i^n=
	\begin{cases}
		-v_i^n/\Delta x, &\mbox{if } i=M,\\
		(v_{i+1}^n-v_{i}^n)/ \Delta x, &\mbox{if } i<M,
	\end{cases}
	\ 
	(D_x^- v)_i^n=
	\begin{cases}
		v_i^n/\Delta x, &\mbox{if } i=-M,\\
		(v_i^n-v_{i-1}^n)/ \Delta x, &\mbox{if } i>-M.
	\end{cases}
\end{align*}
For the simplicity of notation, we omit the parenthesis and denote $(D_x^+ v)_i^n$ and $(D_x^- v)_i^n$ by $D_x^+ v_i^n$ and $D_x^- v_i^n$, respectively.
The central difference approximation of $\partial u$ is then denoted as $D_x=\frac{1}{2}(D_x^+ +D_x^-)$. We approximate the Laplacian $\nabla^2 $ by $D_x^-D_x^+ $, which recovers the central difference approximation. The time derivative is approximated by the forward Euler scheme
$$\left(\frac{\partial v}{\partial t}\right)_i^n\approx D_tv_i^n=(v_i^{n+1}-v_i^n)/\Delta t,$$
where $D_t$ represents the forward time difference operator.

Let $p,q$ be two functions supported on $[-L,L]$. The discretized convolution $p*q$ is computed as
\begin{align*}
	(p*q)_i=\sum_{j=-M}^{M} p_jq_{i-j},\ j=-M,...,M,
\end{align*}
where $q_{j-i}=0$ for $|j-i|>M$ are used.

In the aggregation equation, $u$ usually represents the population density, which follows the conservation law. To keep the conservation property, we use the finite volume method to approximate $L_{u^n}\phi$:
\begin{equation} \label{eq:L}
	L_{u^n}\phi=\nabla\cdot F \quad \mbox{with} \quad F=u^n(\nabla u^n*\phi).
\end{equation}
In the finite volume method, $F$ is known as flux.
We denote the value of $F$ at $x_i$ by $F_i$ and let $F_{i+1/2}=(F_i+F_{i+1})/2$. A conservative way to approximate $\nabla \cdot F_i$ is
\begin{align}
	\nabla \cdot F_i=\frac{F_{i+1/2}-F_{i-1/2}}{\Delta x}.
	\label{eq.flux}
\end{align}
From the given data set $\mathcal{U}$, the discrete analogue of $F_i$ is computed as
$$
F_{i}=U_{i}^n(\nabla (U^n*\phi)_{i})=U_i^n ((D_xU^n)*\phi)_{i}=U_i^n \sum_{j=-M}^{M} D_xU^n_{j}\phi_{i-j}.
$$
Substituting $F_i$ into (\ref{eq:L}) gives rise to the discrete analogue of $L_{u^n}\phi$:
\begin{align}
	(L_{U^n}\phi)_i=\nabla \cdot F_i&= \frac{1}{2\Delta x} \left(U_{i+1}^n \sum_{j=-M}^{M}  D_x U^n_{j}\phi_{i+1-j} - U_{i-1}^n \sum_{j=-M}^{M} D_x U^n_{j}\phi_{i-1-j}\right)\nonumber\\
	&=\frac{1}{2\Delta x} \left( \sum_{j=-M+1}^{M+1} U_{i+1}^n D_x U^n_{j-1}\phi_{i-j} - \sum_{j=-M-1}^{M-1} U_{i-1}^n D_xU^n_{j+1}\phi_{i-j}\right)\nonumber\\
	&=\frac{1}{2\Delta x} \sum_{j=-M}^{M} \left[U_{i+1}^n D_xU^n_{j-1}-U_{i-1}^n D_xU^n_{j+1}\right]\phi_{i-j},
	\label{eq.dis.F}
\end{align}
where we use $U_j^n=0$ for $|j|\geq M$. 

\subsection{Details on the algorithm and denoising}
\label{sec.discretize.Bregman}
We next present details to solve each discretized subproblem in Algorithm \ref{alg1} when $d=1$. Formulas in higher dimensions can be derived similarly. When $d=1$, we use $\psi$ and $b$ to represent $\bfpsi,\bb$. We first derive an explicit formula of $\phi^{k+1}$. Since $L_{U^n}\phi$ is linear in $\phi$, we can find a set of matrices $\{A^n\}_{n=0}^{N}$ such that $L_{U^n}\phi=A^n\phi$. Such matrices can be easily constructed according to (\ref{eq.dis.F}). Then $L_{U^n}=A^n$ and $L^*_{U^n}=(A^n)^{\top}$. Since the forward Euler method $D_t$ is used to compute $D_tU^n$, we only have $D_tU^n$ for $n=0,...,N-1$. Therefore, the first equation (\ref{eq.Bregman.1}) can be discretized as
\begin{align*}
	&\sum_{n=0}^{N-1} \left[(A^n)^{\top}A^n\phi^{k+1}-(A^n)^{\top}D_tU^n\right]\Delta t+ \beta (D_x^-D_x^+)^2\phi^{k+1}\\
	&\quad +\lambda\left( D_x^-(\psi^k-b^k) -D_x^-D_x^+\phi^{k+1}\right)=0.
\end{align*}
Solving for $  \phi^{k+1}$, we obtain
\begin{align}
	\phi^{k+1}=&\left[ \sum_{n=0}^{N-1}\Delta t(A^n)^{\top}A^n+\beta (D_x^-D_x^+)^2-\lambda D_x^-D_x^+\right]^{-1}\nonumber\\
	&\quad\times \left(\sum_{n=0}^{N-1}\Delta t(A^n)^{\top}D_tU^n-\lambda D_x^-(\psi^k-b^k)\right).
	\label{eq.dis.1}
\end{align}

To update $\psi^{k+1}$ and $b^{k+1}$, we first compute $p=b^k+D_x^+ \phi^{k+1}$, and then $\psi^{k+1}$ and $ b^{k+1} $ are  updated as
\begin{align}
	\psi^{k+1}=\max\left( 0, 1-\frac{\alpha}{\lambda |p|}\right)p, \;\;
	\text{ and } \;\;  b^{k+1} =  b^k+D_x^+ \phi^{k+1}-\psi^{k+1}.
	\label{eq.dis.2}
\end{align}
For the initial condition, a simple choice is
$ \phi^0=\psi^0=b^0=0 .$
Another choice is to set $\phi^0$ as the solution to
\begin{align*}
	\sum_{n=0}^{N-1} \Delta t\left[(A^n)^{\top}A^n\phi^0-(A^n)^{\top}D_tU^n\right]- \alpha D_x^-D_x^+\phi^{0}=0,
\end{align*}
which is the Euler-Lagrange equation of the discrete analogue of
$$
\frac{1}{2}\int_0^T \int_{-L}^L (u_t-L_u\phi)^2dx dt+\frac{\alpha}{2}\int_{-L}^L |\nabla \phi|^2dx.
$$
This choice provides a better initial guess of $\phi$. We then let $ \psi^0=D_x^+ \phi^0$ and set $ b^0=0$.

Identifying the underlying potential is challenging with noisy data since noise is amplified in numerical differentiation.
To stabilize the numerical differentiation, we apply the Successively Denoised Differentiation (SDD) proposed in \cite{he2020robust}.   We describe the case of $d=1$ here. Formulas in higher-dimensional cases can be derived in the same way. For the given data set $\mathcal{U}$, the Moving Least Square (MLS) method \cite{lancaster1981surfaces} can be used to denoise the data along the $x$-direction (denoted by $S_x$) or $t$-direction (denoted by $S_t$) respectively,
\begin{align*}
	&S_x \left[U_i^n\right] = p_i^n(x_i),&
	&\text{where  }\;\;p^n_i=\argmin_{p\in P_2}\sum_{j=-M}^M(p(x_j)-U_j^n)^2	\exp\left(-\frac{(x_j-x_i)^2}{h_x^2}\right)\;,&\\
	&S_t \left[U_i^n\right] = p_i^n(t^n),&
	&\text{where  }\;\;p_i^n=\argmin_{p\in P_2}\sum_{j=0}^N (p(t^j)-U_i^j)^2	\exp\left(-\frac{(t^j-t^n)^2}{h_t^2}\right)\;.&
\end{align*}
Here  $h_x$, $h_t>0$ are width parameters, and $P_2$ denotes the set of polynomials of degree no more than $2$.   SDD computes the partial derivatives of the given data set by applying MLS to  denoise the data first and then applying MLS again after each finite difference approximation to denoise each derivative:
\begin{align}\label{eq_sdd}
	&\left(\frac{\partial u}{\partial x}\right)_i^n \approx  S_xD_xS_xU_{i}^n\;,\;\left(\frac{\partial u}{\partial t}\right)_i^n \approx  S_tD_tS_xU_i^{n}.
\end{align}
In the model (\ref{eq.aggregation}), the computation of $L_u$ requires the partial derivatives of $u$. To keep the linearity and the conservative property of the discretization of $L_u$ in (\ref{eq:L}), we only apply SDD to $D_x^+ U_j^n$ and $D_x^- U_j^n$ in (\ref{eq.dis.F}).  

\section{Adaptive support and symmetry  constraint} \label{sec.stability}

When we identify the underlying potential from noisy data, a mismatch between the true support $\widetilde{\Omega}$ of the potential function and the computational domain $\Omega$ may lead to unsatisfactory results. We propose an adaptive support scheme that learns the support during the potential identification process. We also consider a symmetry constraint on the potential to improve the performance further.

\subsection{Adaptive support identification scheme}

When the computational domain $\Omega$ contains the true support of the potential function $\widetilde{\Omega}$,  the identified potential on $\Omega\setminus\widetilde{\Omega}$ often have oscillations when the data is noisy.   If we know that $\widetilde{\Omega}\subseteq B(0,r)\subset \Omega$, we can suppress the oscillation of $\phi$ outside $\widetilde{\Omega}$ by enforcing that
\begin{align}
	\int_{\Omega\setminus B(0,r)}\phi^2\,d\x~\label{eq_L2}
\end{align}
is small. When (\ref{eq_L2}) incorporated into (\ref{eq.min.reg}), it serves as a penalty on the region outside the estimated support $B(0,r)$. Our goal is to automatically identify $r$ such that the boundary $\partial B(0,r)$ stays close to $\partial \widetilde{\Omega}$.

To learn the the optimal $r$, we start from a small $r^0$, and update $r$ in each iteration such that $r^{k+1}\geq r^k$, where $r^k$ is the estimated $r$ in the $k$-th iteration. Specifically,
%
we incorporate the new regularization~\eqref{eq_L2} to~\eqref{eq.split.1} as follows
\begin{align}
	&\phi^{k+1}=\argmin_{\phi}\bigg[\frac{1}{2}\int_0^T \int_{\Omega} |u_t-L_u\phi|^2d\bx dt+ \frac{\beta}{2}\int_{\Omega} |\nabla^2\phi|^2d\bx\nonumber \\
	&\hspace{5cm} + \frac{\lambda}{2}\int_{\Omega} |\bb^k+ \nabla\phi-\bfpsi^k|^2 d\bx+\frac{\gamma}{2}\int_{\Omega\setminus B(0,r^k)}\phi^2\,d\x\bigg] \label{new_eq.split.1}
\end{align}
for some fixed weight parameter $\gamma>0$. Then the Euler-Lagrange equation  becomes
\begin{align*}
	&\int_0^T \left(L^*_uL_u\phi^{k+1} -L^*_uu_t \right)dt+\beta \nabla^4\phi^{k+1} 
	+\lambda \left(\nabla\cdot (\bfpsi^k-\bb^k)-\nabla^2\phi^{k+1}\right) \\
	&\quad+\mathds{1}_{\Omega\setminus B(0,r^k)}\phi=0,
\end{align*}
where $\mathds{1}_{\Omega\setminus B(0,r^k)}(x)$ is the indicator function of ${\Omega\setminus B(0,r^k)}$ which is 1 if $x\in {\Omega\setminus B(0,r^k)}$ and is 0 otherwise.
When $d=1$, the updating formula of $\phi^{k+1}$ is
\begin{align}
	\phi^{k+1}=&\Bigg[ \sum_{n=0}^{N-1}(A^n)^{\top}A^n+\beta (D_x^-D_x^+)^2-\lambda D_x^-D_x^++\mathbf{1}_{\Omega\backslash B(0,r^k)}\Bigg]^{-1} \nonumber\\ &\quad\times \left(\sum_{n=0}^{N-1}(A^n)^{\top}D_tU^n-\lambda D_x^-(\psi^k-b^k)\right),\label{eq.phi.support}
\end{align}
where the matrix $A^n$ is defined in Section \ref{sec.discretize.Bregman}, , $\mathbf{1}_{\Omega\setminus B(0,r^k)}$ is a diagonal matrix whose $(i,i)$-th element is 1 if $x_i\in\Omega\setminus B(0,r^k)$ and  0 otherwise.

As for updating the radius from $r^k$ to $r^{k+1}$, we propose
\begin{align*}
	r^{k+1} &= \argmin_{r } \left[\frac{1}{2}|r-r^{k}|^2+\frac{\gamma}{2}\int_{\Omega\setminus B(0,r)}\left(\phi^{k+1}\right)^2\,dx\right],
\end{align*}
which finds the optimal $r$ near $r^k$ such that the update from $\phi^{k}$ to $\phi^{k+1}$ on $B(0,r)$  is small.
We update $r^{k+1}$ by a one-step of fixed-point method:
\begin{align}
	r^{k+1}&=r^k - \frac{\gamma}{2}\frac{\partial}{\partial r}\left(\int_{\Omega\setminus B(0,r)}(\phi^{k+1})^2\,dx\right)\Bigg|_{r=r^k}.\label{eq_r}
\end{align}
In (\ref{eq_r}), since $r$ controls the integrating domain, the partial derivative is always less than or equal to 0. Therefore we always have $r^{k+1}\geq r^k$  and \eqref{eq_r} produces an non-decreasing sequence of radii.  When $d=1$, this reduces to
\begin{align}
	r^{k+1}&=r^k + \frac{\gamma}{2}\left(\left(\phi^{k+1}(-r^k)\right)^2+\left(\phi^{k+1}(r^k)\right)^2\right)\;.
	\label{eq.r.update}
\end{align}

Our new adaptive support algorithm is summarized in Algorithm \ref{alg2}. Note that Algorithm \ref{alg1} is a special case of Algorithm \ref{alg2} with $\gamma=0$.
\begin{algorithm}[t]
	\SetKwInOut{KwIni}{Initialization}
	\KwIn{$\phi^0,\psi^0,b^0,r^0$, parameters $\alpha,\beta,\lambda,\gamma,\varepsilon$.}
	\While{(\ref{eq_terminate}) is not satisfied}{
		\textbf{Step 1}: Update $\phi^{k+1}$ according to (\ref{eq.phi.support}).\\
		\textbf{Step 2}: Update $\psi^{k+1}$ according to (\ref{eq.dis.2}) with $p=b^k+ D_x^+\phi^{k+1}$. \\
		\textbf{Step 3}: Update $b^{k+1}$ according to (\ref{eq.dis.2}).\\
		\textbf{Step 4}: Update $r^{k+1}$ according to (\ref{eq.r.update}).
	}
	\KwOut{Identified potential $\phi^{k}$.}
	\caption{An adaptive support scheme}
	\label{alg2}
\end{algorithm}

\subsection{Symmetric potential scheme}\label{sec_symm}
In many applications \cite{carrillo2019aggregation}, the potential $\phi$ is a radially symmetric function, in the form of $f(|x|)$ for some function $f:[0,\infty)\rightarrow \R$. In this case, we aim to find the values of $\phi$ along the radial direction. When $d=1$, the discretized potential satisfies  $\phi_{-i}=\phi_i$ for $i=0,...,M$.
When the potential is known to be symmetric, we can modify
Algorithm \ref{alg2} to enforce the symmetry constraint. The major modifications are about the discretization of $L_u\phi$ and how to handle the boundary condition. 

\noindent\emph{Discretization of $L_u\phi$}. After taking the symmetry of $\phi$ into account, we compute $F_i$ as
$$
F_{i}=U_i^n \sum_{j=-M}^{M} D_xU^n_{j}\phi_{i-j}=U_i^n \left(\sum_{j=-M}^{i} D_xU^n_{j}\phi_{i-j}+\sum_{j=i+1}^{M} D_xU^n_{j}\phi_{j-i}\right).
$$
Correspondingly,
\begin{align*}
	\nabla \cdot F_i&= \frac{1}{\Delta x} \sum_{j=-M+1}^{M-1} \left[U_{i+1}^n D_xU^n_{j-1}-U_{i-1}^n D_xU^n_{j+1}\right]\phi_{i-j}\\
	&= \frac{1}{\Delta x} \sum_{j=-M+1}^{i} \left[U_{i+1}^n D_xU^n_{j-1}-U_{i-1}^n D_xU^n_{j+1}\right]\phi_{i-j}
	\\
	&\quad +  \frac{1}{\Delta x} \sum_{j=i+1}^{M-1} \left[U_{i+1}^n D_xU^n_{j-1}-U_{i-1}^n D_xU^n_{j+1}\right]\phi_{j-i}.
\end{align*}

\noindent\emph{Natural boundary condition.} For the boundary condition of $\phi$, we have $\phi_M=0$ at $x=L$. When deriving the first variation of~\eqref{eq.split.1}, the boundary terms arising from the integration by parts are
\begin{align}
	\beta\phi_{xx}\eta_{x}|_{0}^L-\beta\phi_{xxx}\eta|_{0}^L+\lambda(\phi_x+b^k-\psi^k)\eta|_{0}^L,\label{eq_boundaryterm}
\end{align}
where $\eta$ denotes a test function (see Appendix~\ref{app_der} for details). After evaluating the functions at the boundary points with $\eta_x(0)=\eta_x(L)=\eta(L)=0$, the optimality condition gives rise to the constraint
\begin{align}
	-\beta\phi_{xxx}(0)+\lambda(\phi_x(0)+b^k(0)-\psi^k(0))=0.
\end{align}
After discretization, we obtain
\begin{align*}
	\beta\frac{\phi_{2}-2\phi_{1}+2\phi_{-1}-\phi_{-2}}{2\Delta x^3}+\lambda\left(\frac{\phi_{0}-\phi_{-1}}{\Delta x}+b^k_0-\psi^k_0\right)=0.
\end{align*}
Setting $\phi_2=\phi_{-2}$ and $\phi_{1}=\phi_{-1}$ gives
\begin{align}
	\phi_0=\phi_{1}+(\psi_0^k-b_0^k)\Delta x.
\end{align}

This scheme can be easily extended to high dimensions using polar coordinates. In Section~\ref{sec_num}, we discuss the effects of imposing symmetry to the potential recovery.

\section{Numerical experiments}\label{sec_num}

In this section, we demonstrate the effectiveness and robustness of our proposed method through systematic experiments.
We denote the exact solution to (\ref{eq.aggregation}) with the underlying potential $\phi^*$ by $u^*$. Our noisy data is generated by  adding i.i.d. Gaussian noise to the discretized samples of $u^*$. The noise has $0$ mean and the standard deviation is $\sigma$.
We say the noise is $\rho\%$ if
\begin{align*}
	\sigma=\frac{\rho}{100}\left[\sum_{n=1}^N\sum_{i=-M}^M (u_i^n)^2\Delta x\Delta t\right]^{1/2}.
\end{align*}	
We use $\widetilde{u}$ to denote the denoised data.  We denote the identified potential from the noisy data set by $\widehat{\phi}$, with which the simulated solution of (\ref{eq.aggregation}) is denoted by $\widehat{u}$. When computing $\widehat{u}$, we use $\widetilde{u}^0$ (the denoised initial data) as the initial condition.

We qualify the identified potential by the following relative errors
\begin{align}
	&e_{\phi}=\frac{\|\widehat{\phi}-\phi^*\|_1}{\|\phi^*\|_1}\times 100\%,\label{eq_phi_error}\\
	&e^*(t)= \frac{\|\widehat{u}(t,\cdot)-u^*(t,\cdot)\|_1}{\|u^*(t,\cdot)\|_1}\times 100\%\;,\;\;\;t\in[0,T],\label{eq_R_star}\\
	&\widetilde{e}(t)= \frac{\|\widehat{u}(t,\cdot)-\widetilde{u}(t,\cdot)\|_1}{\|\widetilde{u}(t,\cdot)\|_1}\times 100\%\;,\;\;\;t\in[0,T],\label{eq_R_tild}
\end{align}
where $\|v\|_1=\sum_{i,n} |v_i^n|$ is the $L^1$-norm of $v$ over the sampling grid. The error $e_{\phi}$ in \eqref{eq_phi_error} compares the identified potential with the exact potential. The error $e^*$ in \eqref{eq_R_star} compares the simulated data with the exact solution. The error $\widetilde{e}$ in ~\eqref{eq_R_tild} measures the difference between the simulated data and the denoised data. The first two errors require the exact potential or the exact data, while the third one only uses the denoised data. Hence,~\eqref{eq_R_tild} is more practical and can be used when the exact potential or data are not given.

When $\gamma=0$, Algorithm \ref{alg2} does not learn the support of the potential and is reduced to Algorithm \ref{alg1}. Therefore Algorithm \ref{alg1} is a special case of Algorithm \ref{alg2} and we use Algorithm \ref{alg2} for all experiments.
There are four parameters in Algorithm \ref{alg2}: $\alpha,\beta,\gamma$ and $\lambda$, where $\alpha$ and $\beta$ controls the smoothness of the identified potential, $\gamma$ controls the size of the support, and $\lambda$ is the weight of the penalty on the mismatch between $\nabla \phi$ and $\bfpsi$. The optimal choice of these parameters is problem-dependent. 
Here we give a guideline for the choice of these parameters based on their effects on the identified potential. If one assumes the potential contains singularities corresponding to non-collision conditions, one should use a large $\alpha$ and a small $\beta$. If the potential is assumed to be very smooth, then a larger $\beta$ should be used. For the parameter $\gamma$, a larger support of the potential implies a longer range of nonlocal interactions. If one assumes the behavior of each individual is affected by others in a large neighborhood of it, then one should use a small $\gamma$, i.e., a small penalty on the support. Otherwise, a large $\gamma$ should be used. For $\lambda$, larger $\lambda$ makes the functional (\ref{eq.min.unconstraint}) a better approximation of the original function (\ref{eq.min.reg}).
In our algorithm, we use $h=0.04$ in SDD.
Without specification, $\lambda=0.05$ and $r^0=L/100$ are used in our one-dimensional experiments.

\subsection{One-dimensional potential identification}

For all of the one-dimensional examples, we generate clean data by numerically solving \eqref{eq.aggregation} with the initial condition
\cite{carrillo2019aggregation}
\begin{align}
	u^*(0,x) = {0.15^{1/3}}M_0\max\left\{C_0 -\frac{|x|^2}{12\times 0.15^{2/3}}\;,0\right\}\;\label{Yao_initial}
\end{align}
on the computational domain $-1\leq x\leq 1$, $0\leq t\leq T$ for some maximal time $T$. Let $\Delta x = \Delta t = 0.01$. We set $M_0=0.6$, and $C_0$ is a normalization factor such that $\int_{\mathbb{R}}u^*(0,x)\,dx = M_0$.


\subsubsection{Truncated repulsive-attractive power law}

\begin{figure}[t!]
	\begin{tabular}{c@{\hspace{2pt}}c@{\hspace{2pt}}c}
		(a)&(b)&(c)\\
		\includegraphics[trim={0 0 1.6cm 0},clip,width=0.3\textwidth]{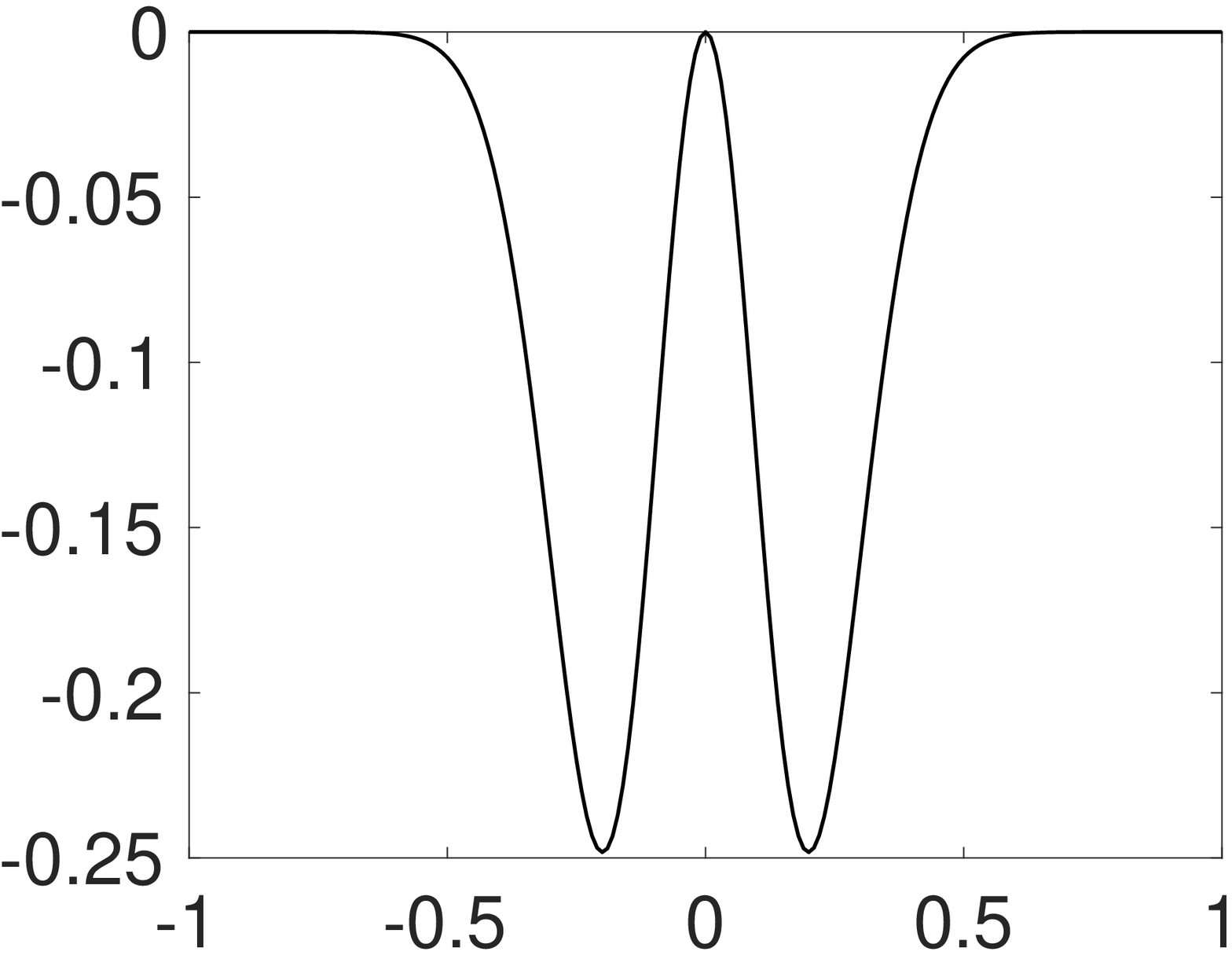}&
		\includegraphics[trim={0 0 1.6cm 0},clip,width=0.3\textwidth]{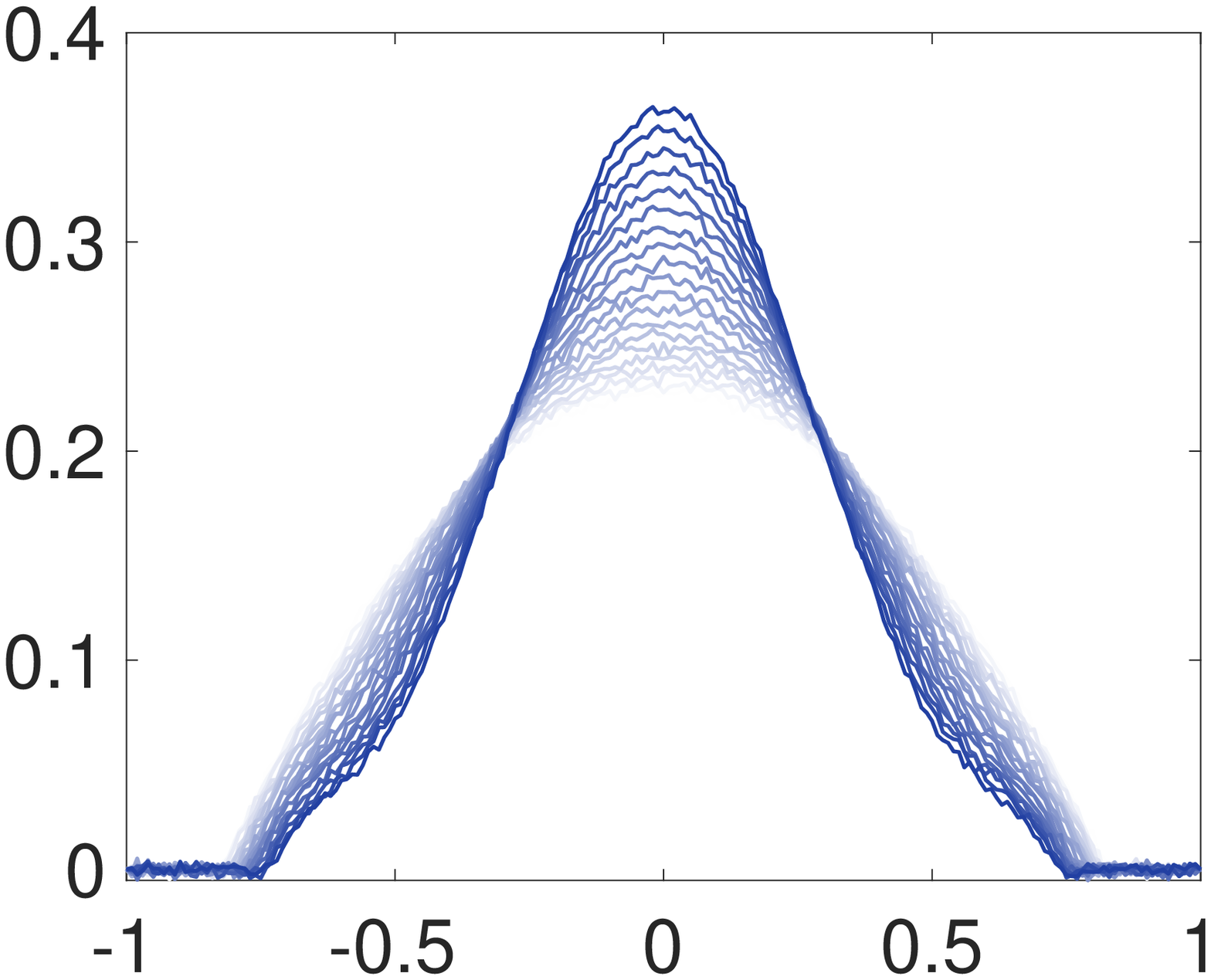}&
		\includegraphics[trim={0 0 1.6cm 0},clip,width=0.3\textwidth]{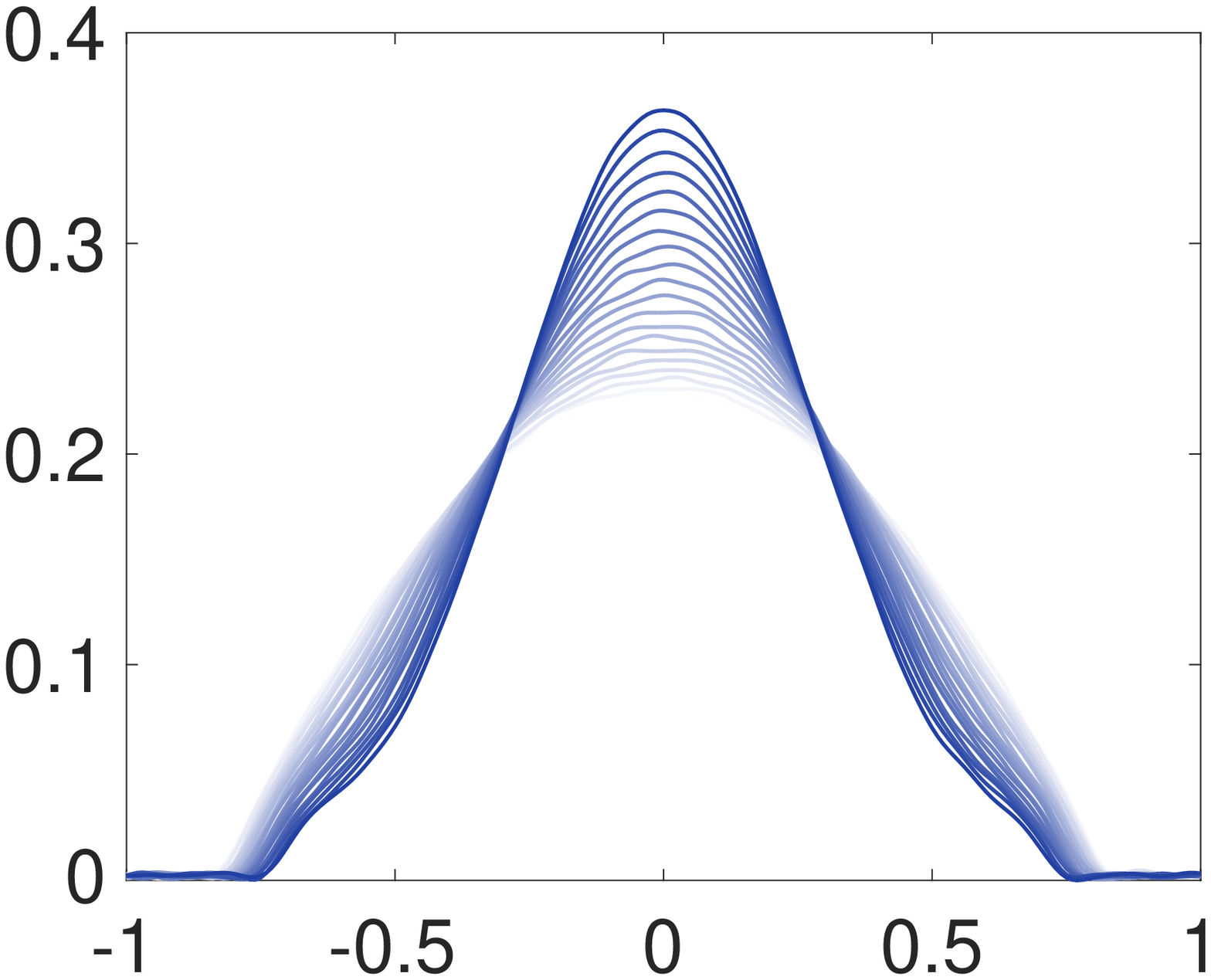}\\
		(d)&(e)&(f)\\
		\includegraphics[trim={0 0 1.6cm 0},clip,width=0.3\textwidth]{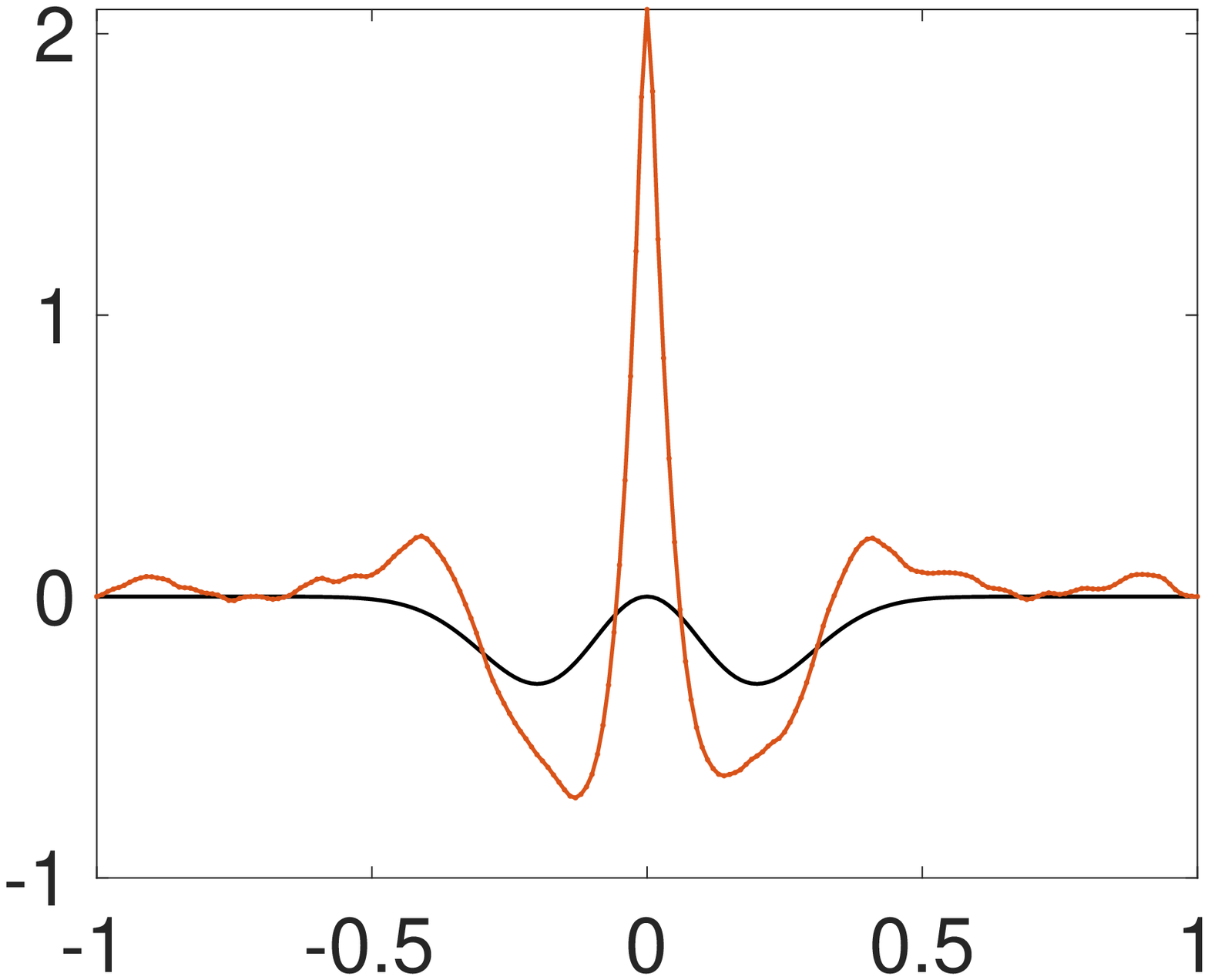}&
		\includegraphics[trim={0 0 1.6cm 0},clip,width=0.3\textwidth]{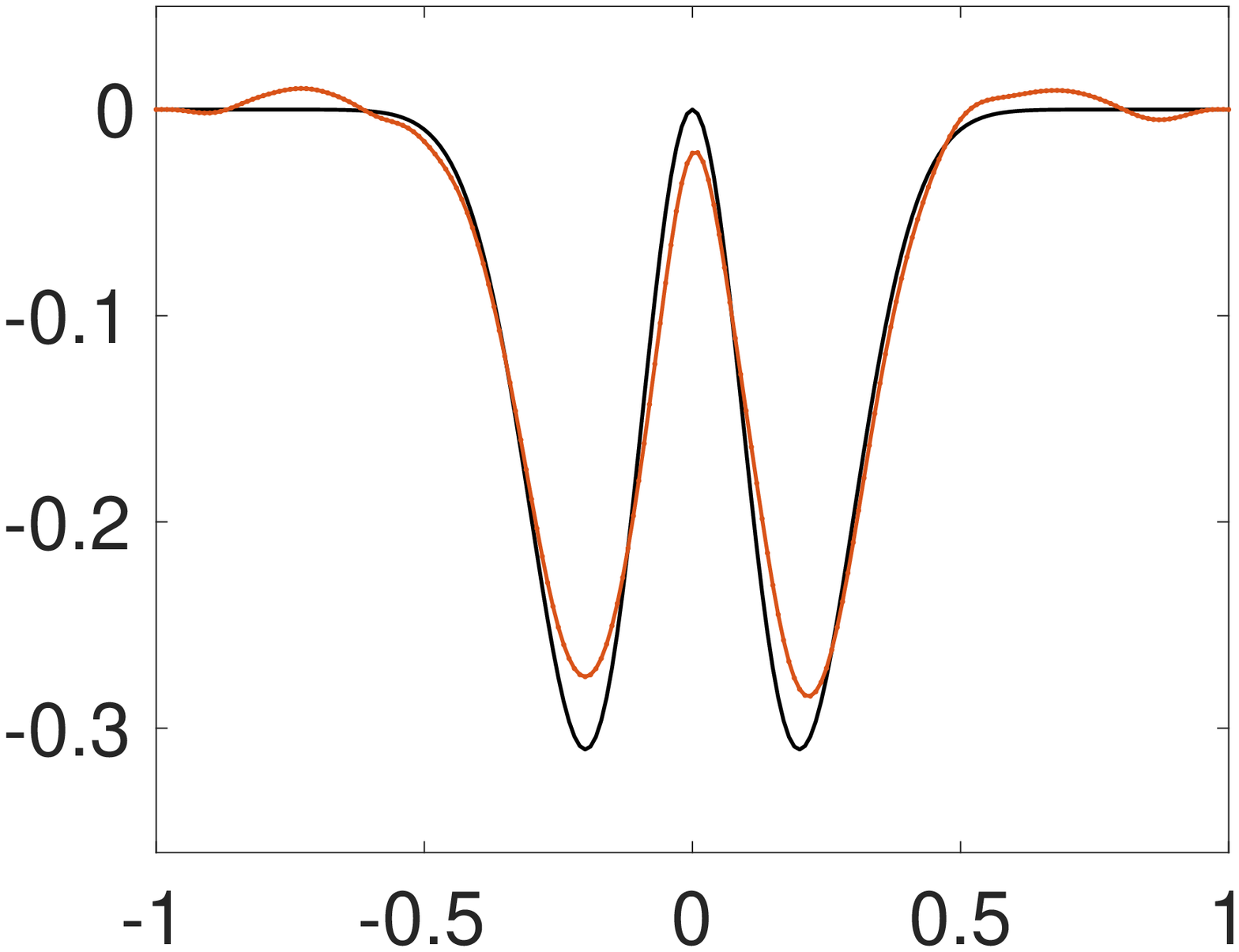}&
		\includegraphics[trim={0 0 1.6cm 0},clip,width=0.3\textwidth]{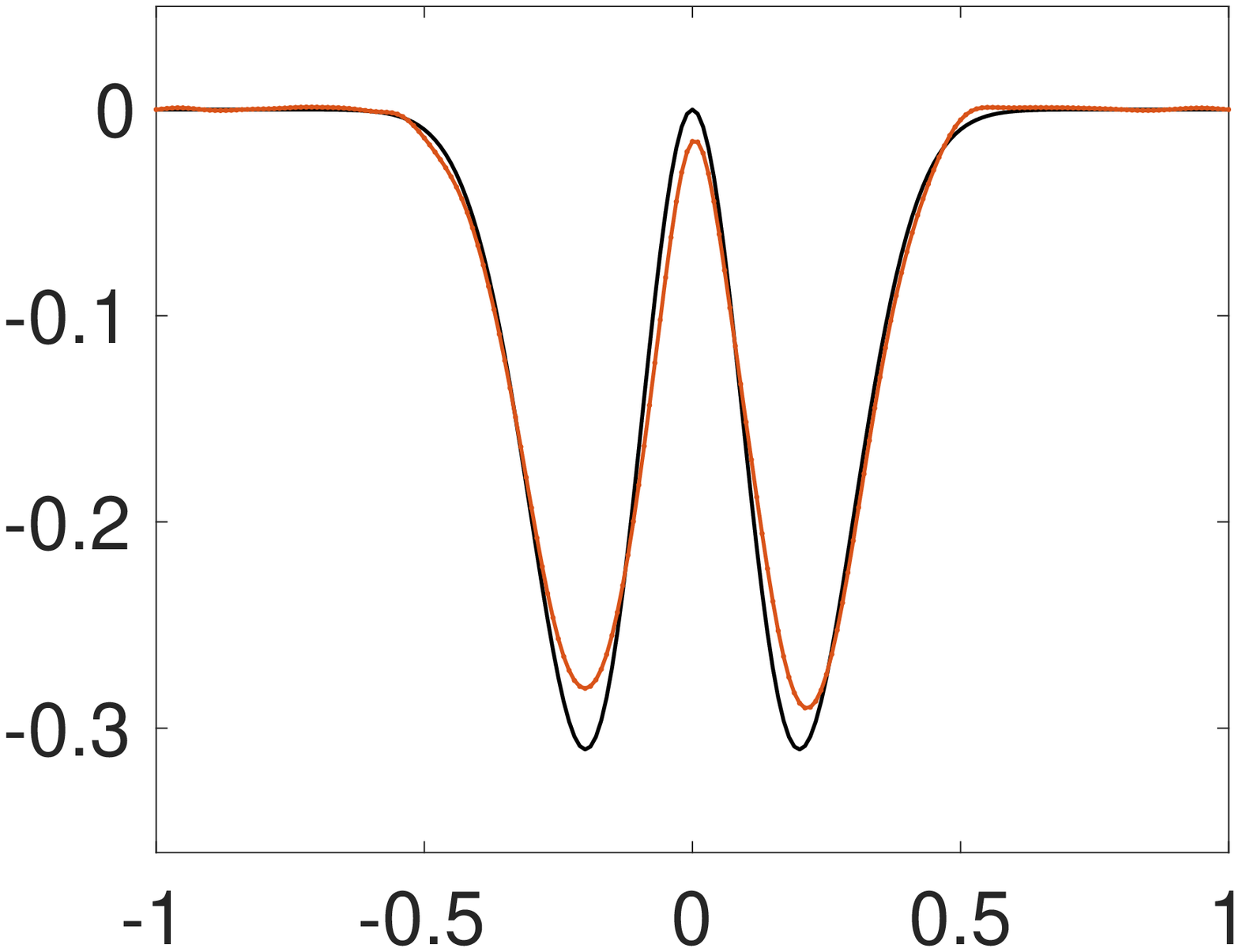}\\
		(g)&(h)&(i)\\
		\includegraphics[trim={0 0 1.6cm 0},clip,width=0.3\textwidth]{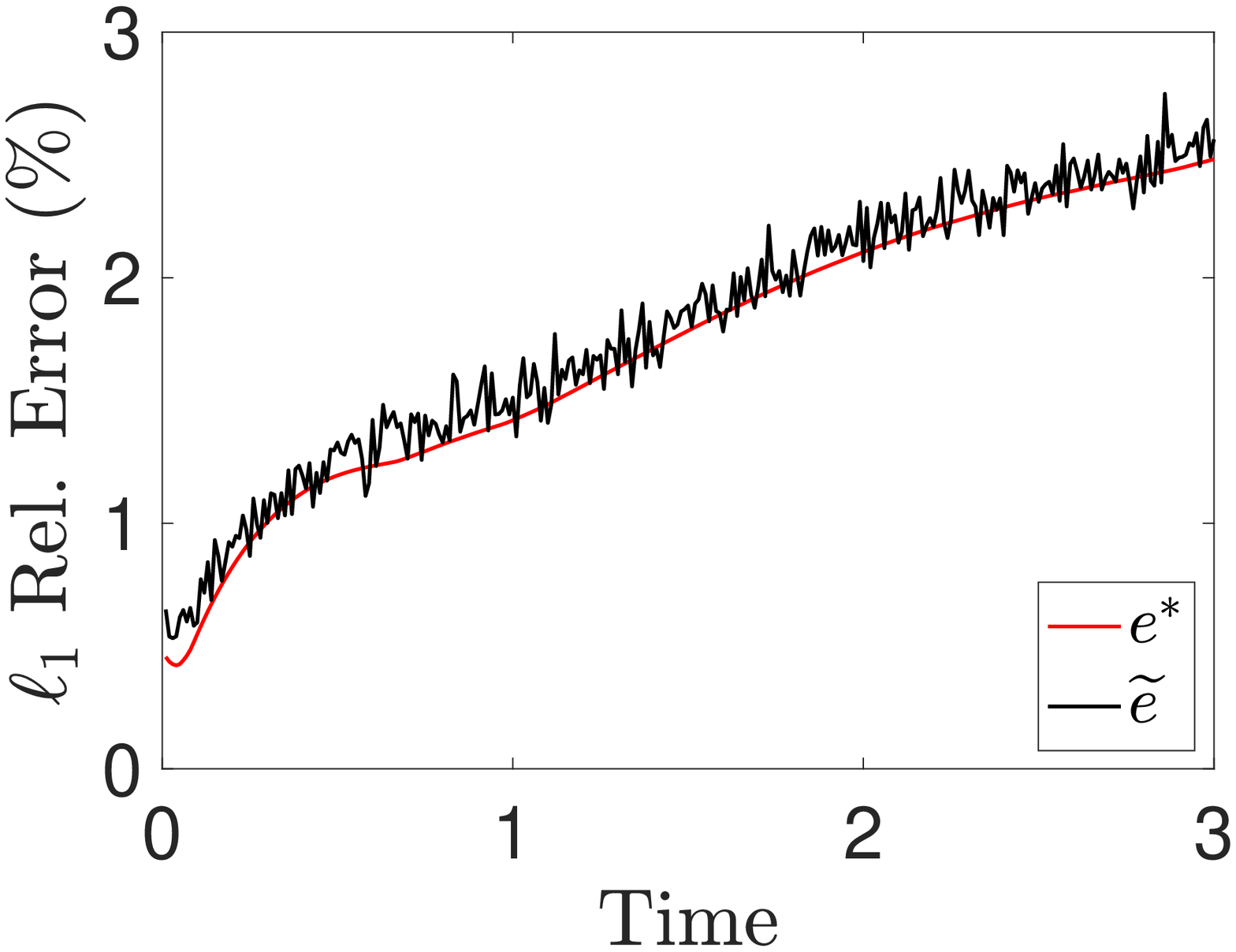}&
		\includegraphics[trim={0 0 1.6cm 0},clip,width=0.3\textwidth]{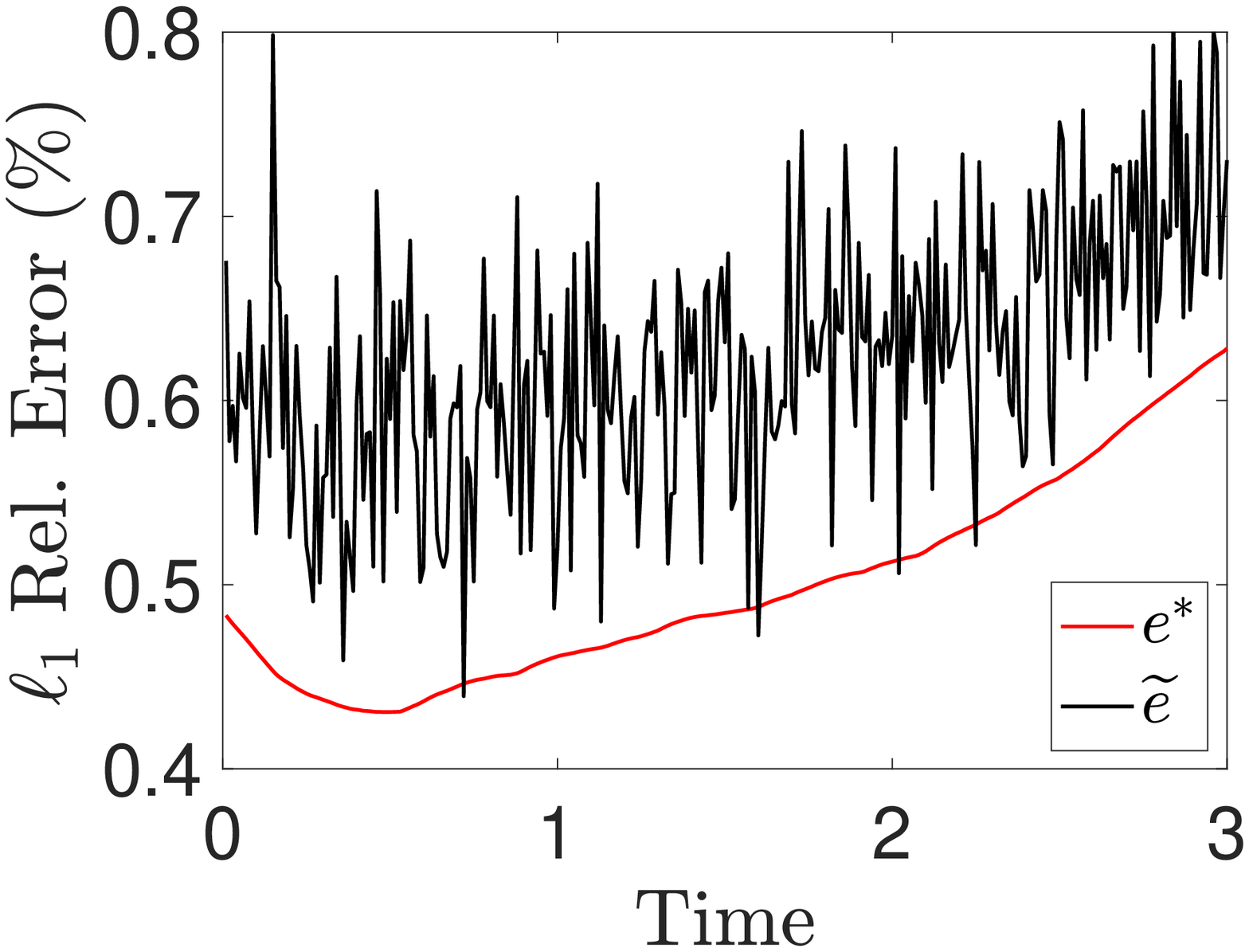}&
		\includegraphics[trim={0 0 1.6cm 0},clip,width=0.3\textwidth]{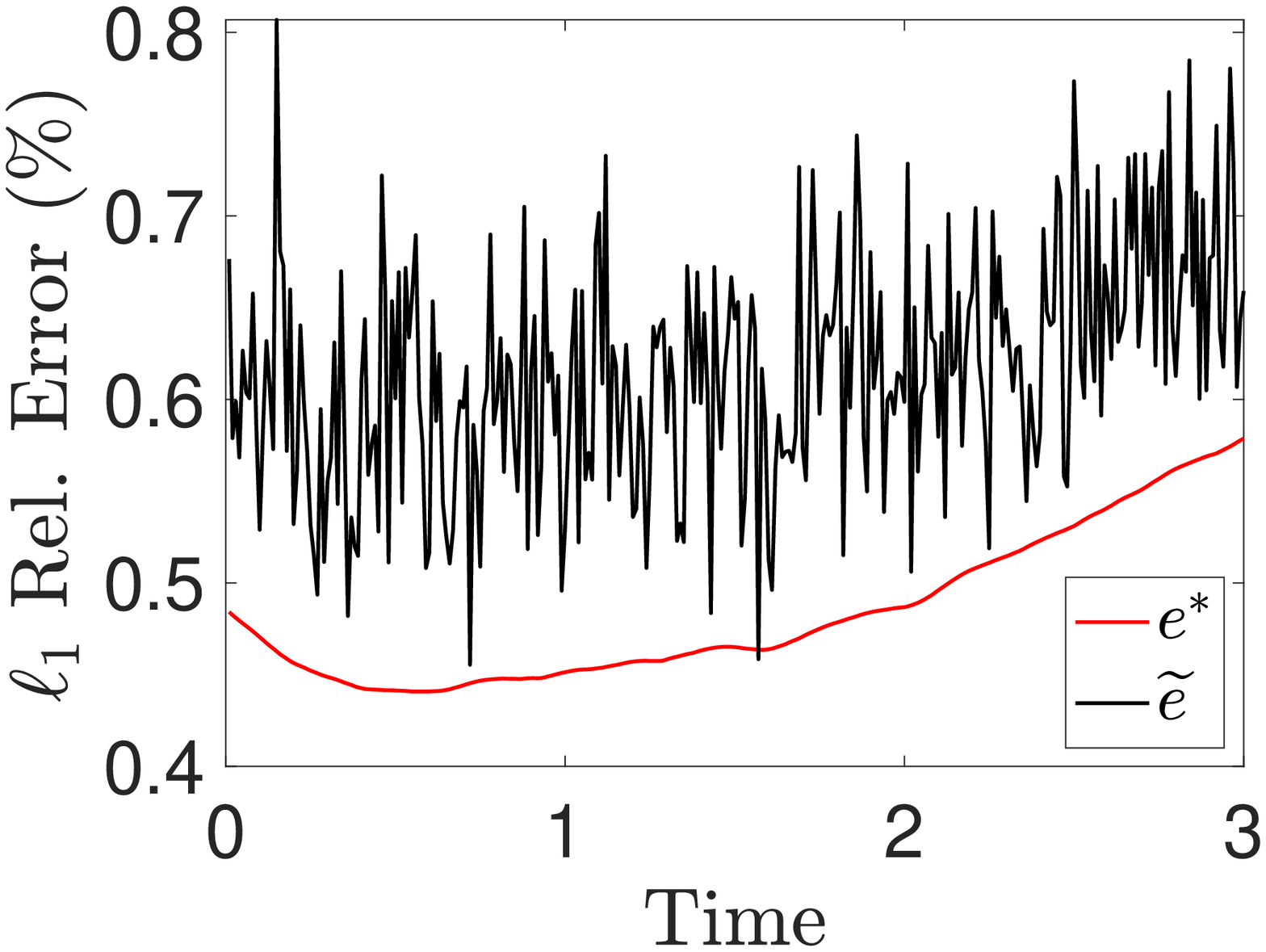}
	\end{tabular}
	\caption{1D example of the repulsive-attractive power law $\phi_{RA}$ in (\ref{eq_RA}). (a) The graph of $\phi_{RA}$.
		(b) The given noisy data with $1\%$ additive Gaussian noise. Each curve represents the data at a specific time.  (c) Denoised data $\widetilde{u}$ by SDD.   
		(d)  Exact $\phi_{RA}$ (black) and the identified $\widehat{\phi}$ (red) from the noisy data in (b) without any regularization, i.e., $\alpha=\beta=\gamma=0$. The error is $e_\phi=233.27\%$. (e) Exact $\phi_{RA}$ (black) and the identified $\widehat{\phi}$ (red) from the denoised data in (c) without adaptive support. The error is $e_\phi=13.31\%$.  (f) Exact $\phi_{RA}$ (black) and the identified $\widehat{\phi}$ (red) from the denoised data in (b) with the adaptive support scheme where $\gamma=10$. The error is $e_\phi=10.01\%$.  (g), (h), (i) show the errors $e^*(t)$ (red) and $\widetilde{e}(t)$ (black) for the identified potential in (d), (e), (f) as a function of $t$, respectively. 
	}
	\label{exp1}
\end{figure}

Figure~\ref{exp1} shows the identification result of the truncated repulsive-attractive power law~\cite{carrillo2019aggregation}
\begin{align}
	\phi_{RA}(x;\theta_1,\theta_2,m_0) = m_0\left(\frac{|x|^{\theta_1}}{\theta_1} - \frac{|x|^{\theta_2}}{\theta_2}\right)\exp(-\frac{|x|^2}{4\tau^{2}})/\sqrt{4\pi\tau^2}\;\label{eq_RA}
\end{align}
with $\theta_1=5,\theta_2=2,\tau=0.1$, and $m_0=15$. The graph of $\phi_{RA}$ is shown in Figure \ref{exp1}~(a), whose variation is concentrated near the origin.
The clean data is generated by solving (\ref{eq.aggregation}) with potential (\ref{eq_RA}) and $\Delta x=0.01, \Delta t=0.01, T=3$. We add  $1\%$ Gaussian noise to obtain the noisy data in Figure \ref{exp1}~(b). The denoised data by SDD are shown in (c).  Figure \ref{exp1}~(d) shows the identified potential (red) from the noisy data without regularization, i.e., $\alpha=\beta=\gamma=0$. Such an identification has a large error with $e_\phi=233.27\%$. 
The unstable recovery in (d) results from the noise amplification in numerical differentiation. The corresponding errors $e^*(t)$ and $\tilde{e}(t)$ are displayed in (g), which show the mismatch between the exact (or denoised) data and the simulated solution based on the identified potential from noisy data.
Denoising and regularization are important for a stable identification. In (e) and (f), we employ SDD and utilize regularization by setting $\alpha=1\times10^{-5}$ and $\beta=1\times10^{-7}$. The result in (e) does not use the adaptive support scheme, i.e., $\gamma=0$, while in (f), we set \tr{ $\gamma=10$}. Regularization significantly improves the identification result. For the identified kernel in (e) and (f), the error $e_{\phi}$ are $13.31\%$ and $10.01\%$ respectively. The errors $e^*(t)$ and $\widetilde{e}(t)$ as a function of $t$ are shown in (h) and (i), respectively.
The identified potential in (e) has oscillations near the boundary without the adaptive support scheme. These oscillations are eliminated in (f) by adopting our adaptive support scheme.

\subsubsection{Truncated Morse potential and Topaz potential}	
Figure~\ref{exp2} shows the identification results of  the
truncated Morse potential~\cite{d2006self}
\begin{align}
	\phi_{\text{Morse}}(x) =\left( -C_A\exp(-|x|/\ell_A)+C_R\exp(-|x|/\ell_R)\right)\exp(-\frac{|x|^2}{4\tau^{2}})/\sqrt{4\pi\tau^2}\;\label{eq_morse}
\end{align}
with $C_A= 0.5,\ell_A = 0.5, C_R = 0.2, \ell_2=0.4$ and
the truncated Topaz potential~\cite{topaz2006nonlocal}
\begin{align}
	\phi_{\text{Topaz}}(x) =\left(1+|x|\right)^{-a}\exp(-\frac{|x|^2}{4\tau^{2}})/\sqrt{4\pi\tau^2}\;\label{eq_topaz}
\end{align}
with $a=-0.1$ and $\tau=0.1$. The graph of these two potentials are shown in Figure~\ref{exp2} (a) and (c) in black, respectively. In both examples, $1\%$ Gaussian noise is added to the clean data generated with $\Delta x = 0.1$ and $\Delta t =0.1$.  We use $T=1$ for the Morse potential, and $T=0.6$ for the Topaz potential. The parameters are set as $\alpha=1\times10^{-5}, \beta=1\times 10^{-7}, \gamma = 10$. We present the identification results for the Morse potential in Figure~\ref{exp2} (a) and (b), and the results for the Topaz potential in Figure~\ref{exp2} (c) and (d). The identified potential in (a) has error $e_\phi=12.73\%$, the error in (c) is $e_\phi=7.59\%$. 
\begin{figure}[h!]
	\begin{tabular}{cccc}
		(a)&(b)&(c)&(d)\\
		\includegraphics[trim={0.5cm 0 1.8cm 0},clip,width=0.21\textwidth]{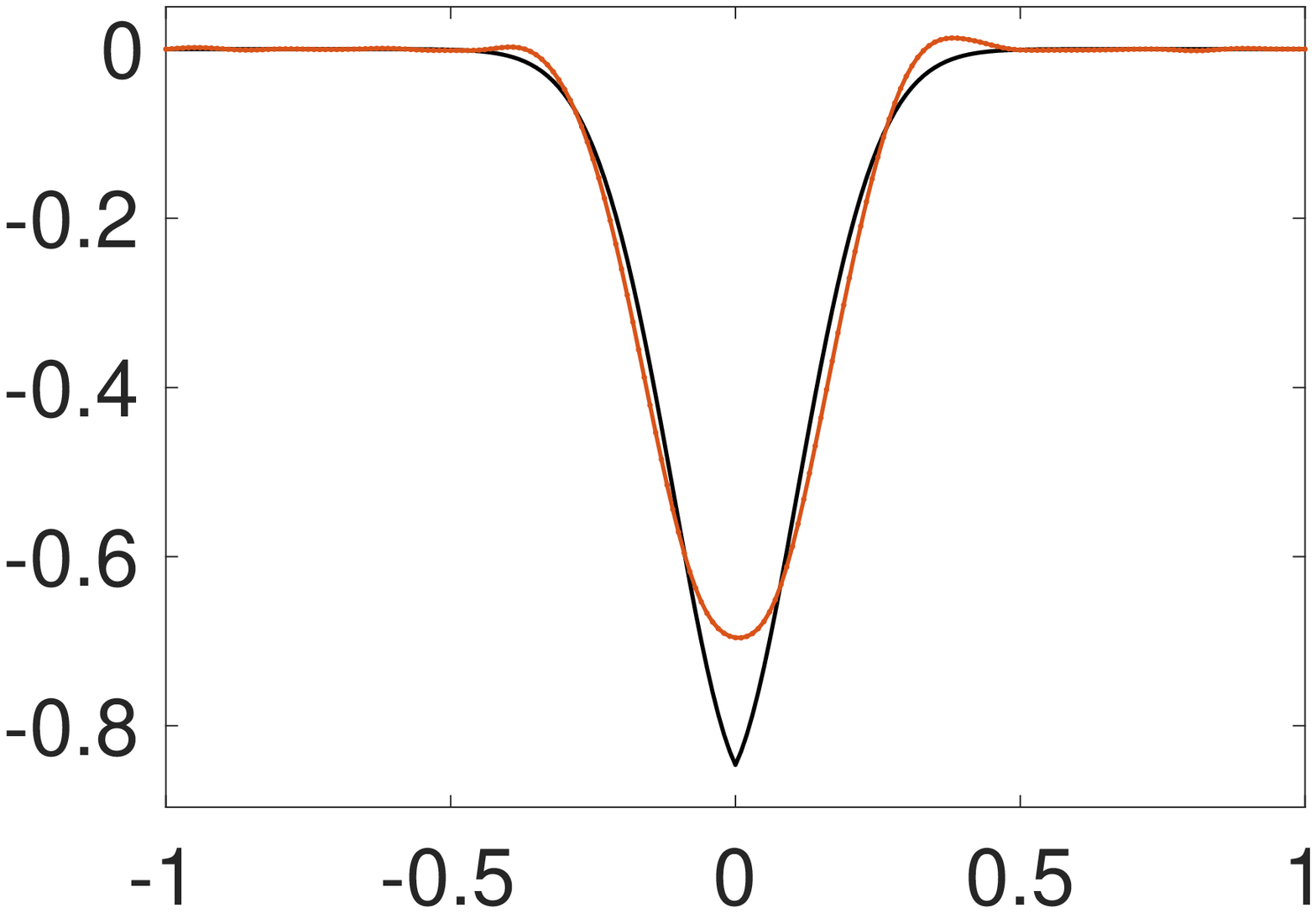}\hspace{0.3cm}
		&
		\includegraphics[trim={0.1cm 0 1.8cm 0},clip,width=0.21\textwidth]{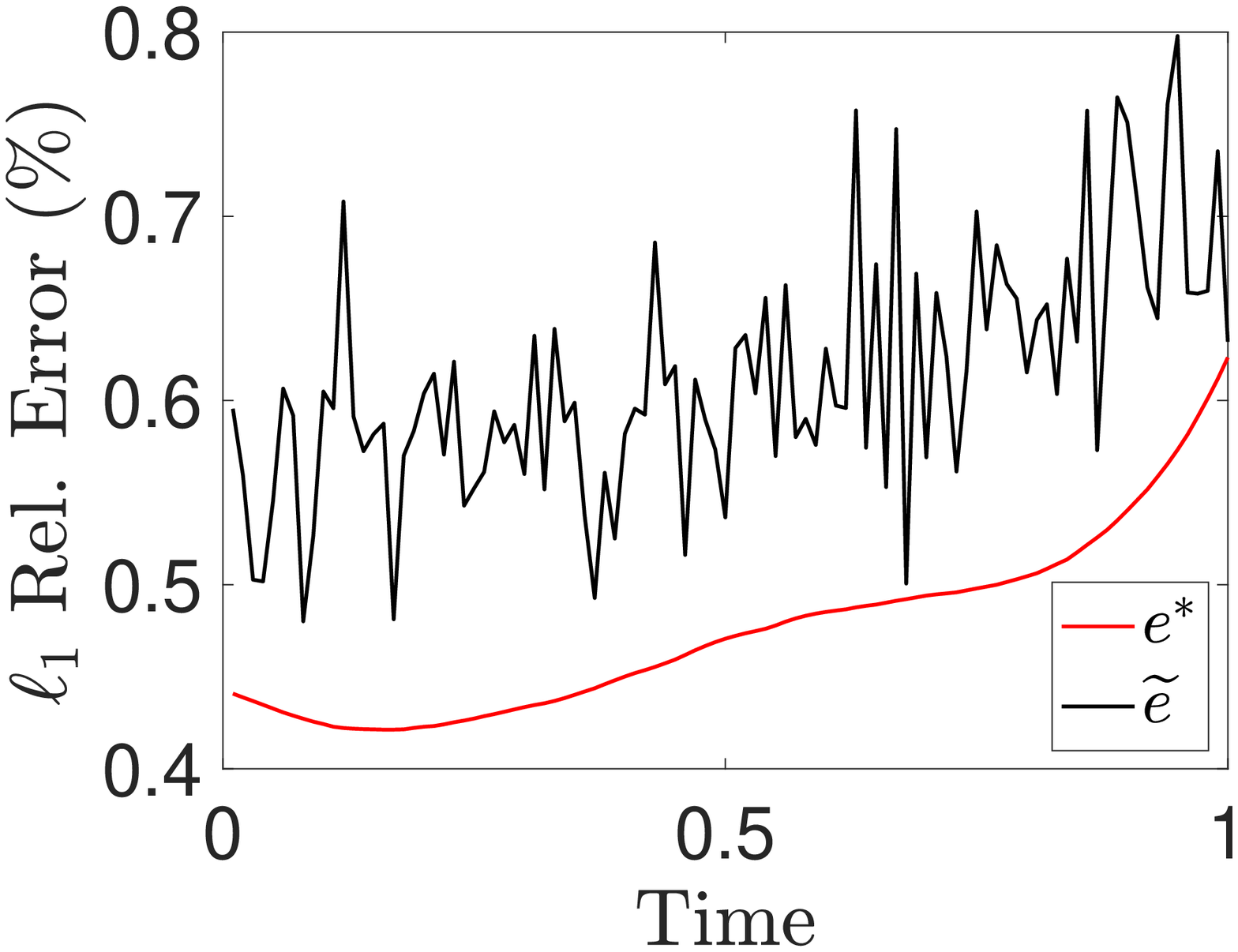}\hspace{0.3cm} &
		\includegraphics[trim={1.7cm 0 1.8cm 0},clip,width=0.21\textwidth]{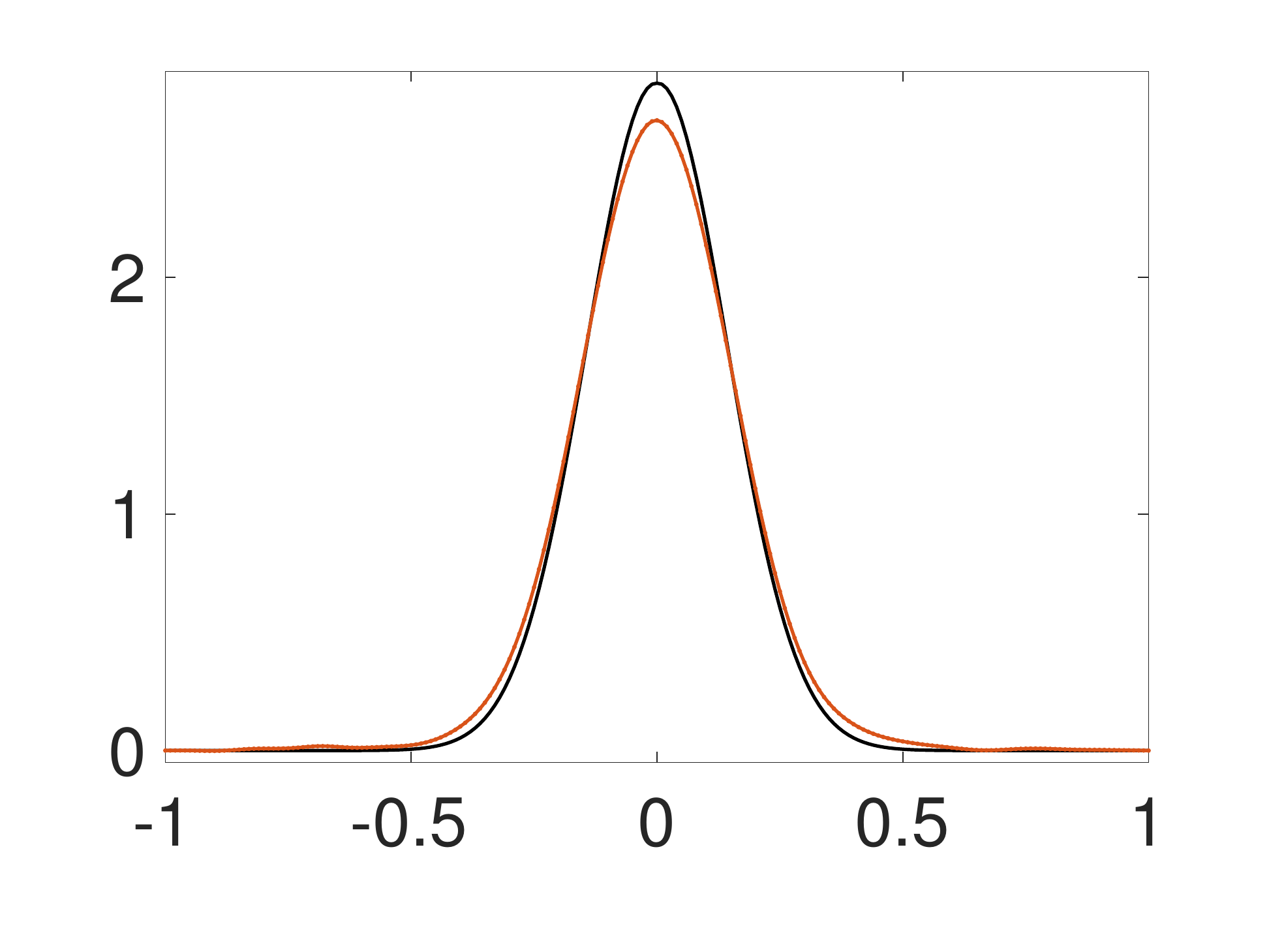}\hspace{0.3cm} &
		\includegraphics[trim={0.1cm 0 1.2cm 0},clip,width=0.21\textwidth]{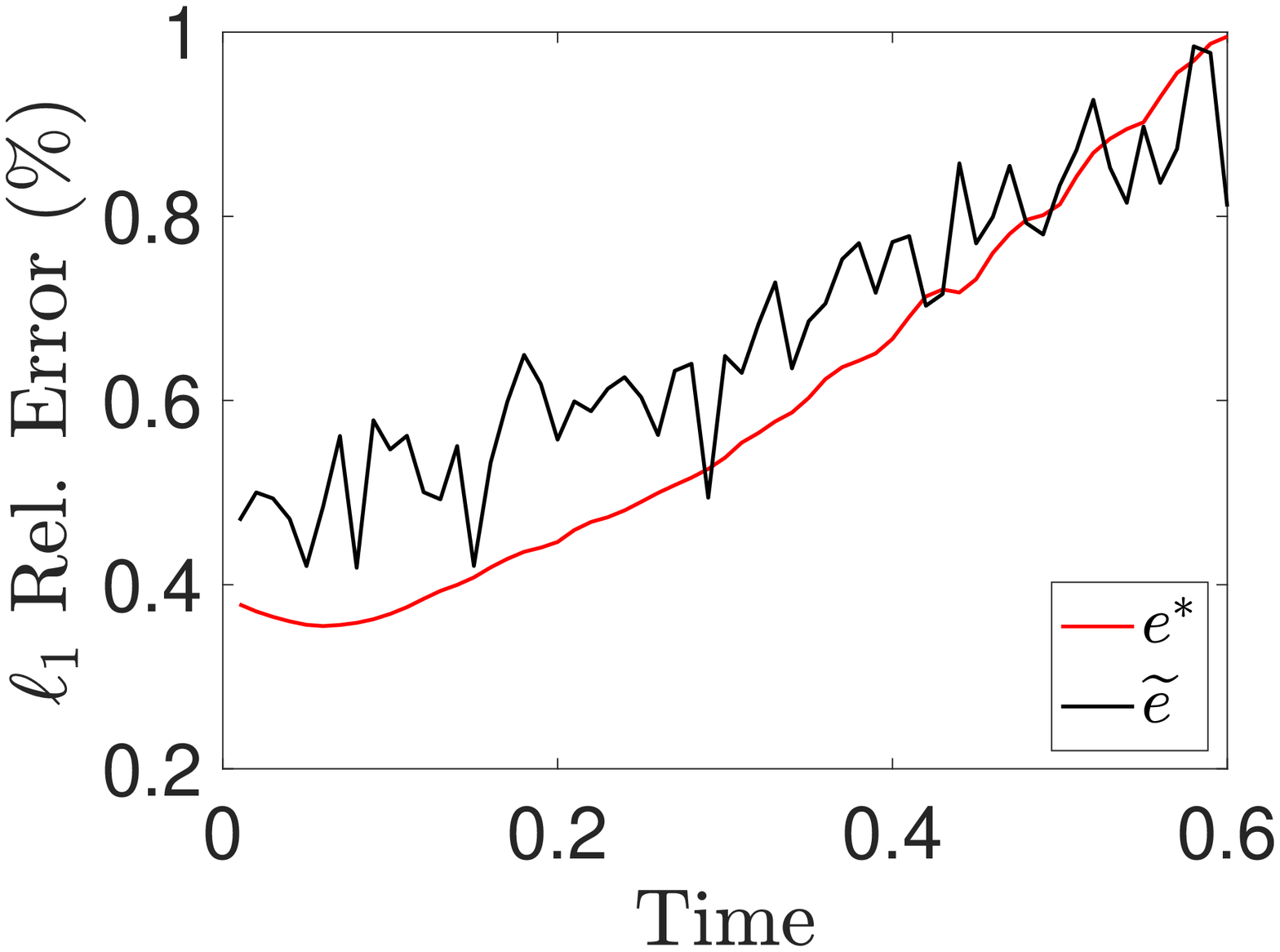}
	\end{tabular}
	\caption{1D example of the Morse potential $\phi_{\text{Morse}}$ in~\eqref{eq_morse} and the Topaz potential $\phi_{\text{Topaz}}$ in~\eqref{eq_topaz} . (a) The graph of $\phi_{\text{Morse}}$ (black) and the identified $\widehat{\phi}$ (red) with error $e_\phi=12.73\%$.  (b) The errors $e^*(t)$ (red) and $\widetilde{e}(t)$ (black) for the identified potential in (a). (c) The graph of $\phi_{\text{Topaz}}$ (black) and the identified $\widehat{\phi}$ (red) with error $e_\phi=7.59\%$. (d) The errors $e^*(t)$ (red) and $\widetilde{e}(t)$ (black) for the identified potential in (c).  The given data contain  $1\%$ Gaussian noise. We apply SDD and the adaptive support scheme with $\gamma = 10$.
	}
	\label{exp2}
\end{figure}

\subsection{Two-dimensional potential identification}
We next experiment on two-dimensional potentials without the adaptive support scheme, i.e., $\gamma=0$. Our computational domain is $[-1,1]^2$ with $\Delta x_1=\Delta x_2=1/15$. The clean data are generated by solving (\ref{eq.aggregation}) with $\Delta t=0.02, T=4$ and the initial condition:
\begin{align*}
	u_0(x_1,x_2)=\exp\left(-\frac{x_1^2+(x_2+0.3)^2}{0.2^2}\right)+ \exp\left(-\frac{x_1^2+(x_2-0.3)^2}{0.2^2}\right).
\end{align*}
Then $1\%$ Gaussian noise is added to generate noisy data. 

Figure~\ref{fig.RINO.2d.ar} shows the identification result of  the attraction-repulsion potential
\begin{align}
	\phi^*=10\left((x_1^2+x_2^2)^{0.55}/1.1-(x_1^2+x_2^2)^{0.5}\right)\exp\left(-(x_1^2+x_2^2)^{0.5}/0.1\right).
	\label{eq.numerical.2D.ar}
\end{align}
In this experiment, we set $\alpha=2\times10^{-4},\beta=2\times10^{-7}$ and $\lambda=2$. The exact potential $\phi^*$ and the identified potential $\widehat{\phi}$ are shown in Figure \ref{fig.RINO.2d.ar} (a) and (b), respectively. We compare their cross sections along $x_1=0$ and $x_2=0$ in (c) and (d), respectively. We observe that $\widehat{\phi}$ recovers the attraction-repulsion property and well approximates  $\phi^*$. The errors $e^*(t)$ and $\widetilde{e} (t)$ are shown in Figure \ref{fig.RINO.2d.ar} (e). As the solution evolves from $t=0$ to 4, $e^*(t)$ is always below $2.5\%$.
\begin{figure}[t!]
	\begin{tabular}{ccc}
		(a)&(b)&\\
		\includegraphics[width=0.32\textwidth]{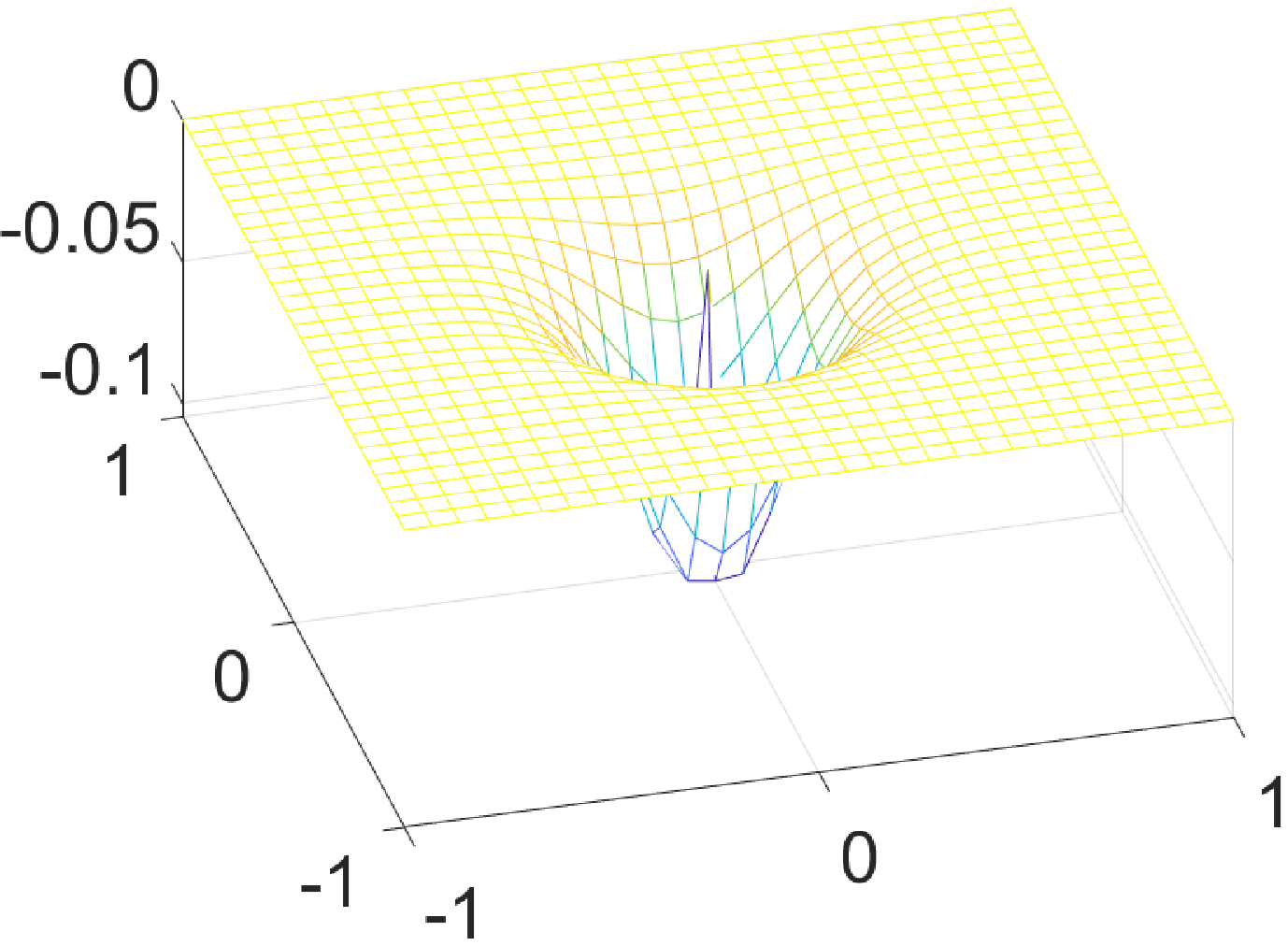}&
		\includegraphics[width=0.3\textwidth]{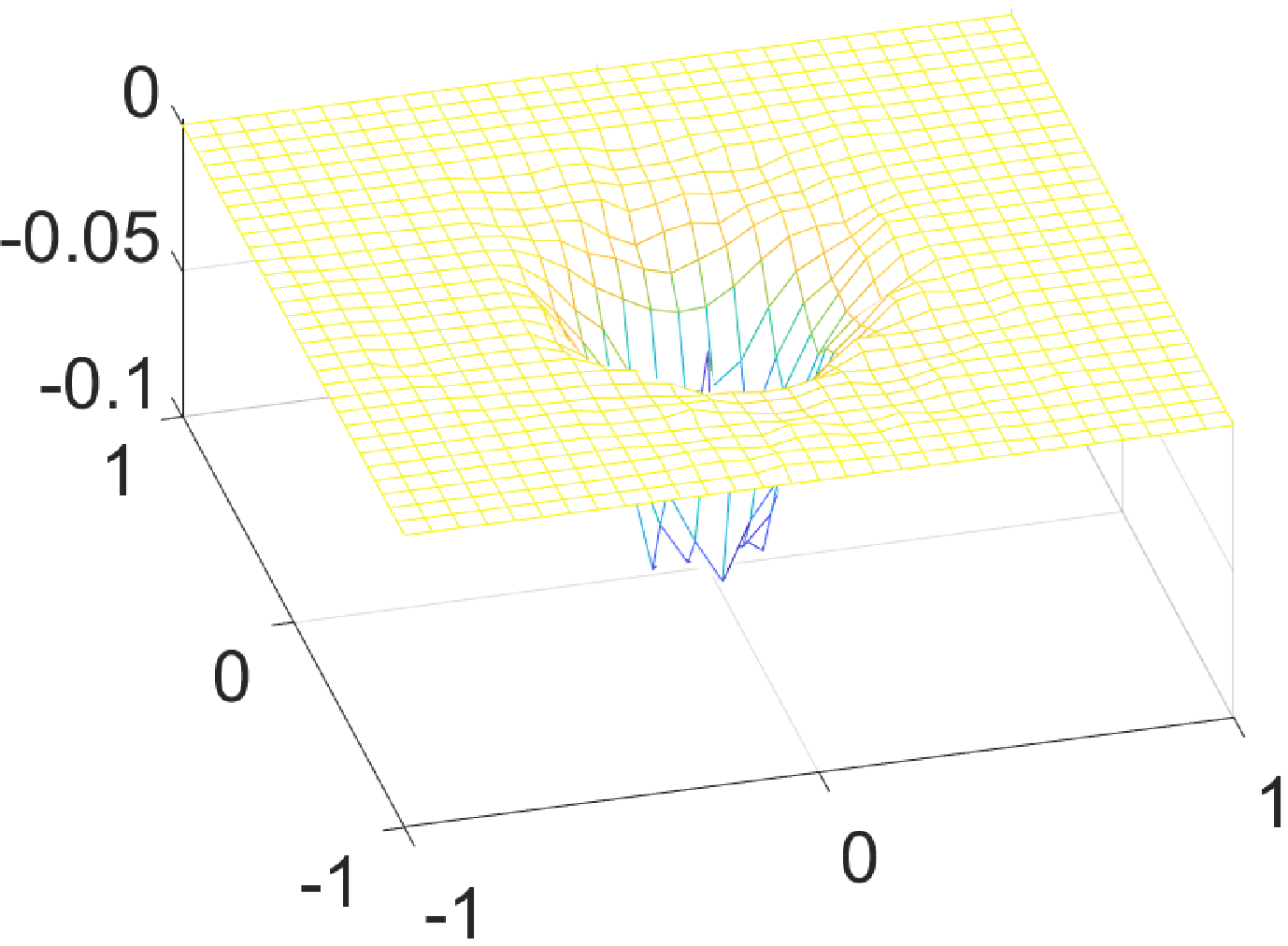} &\\
		(c)&(d)&(e)\\
		\includegraphics[width=0.32\textwidth]{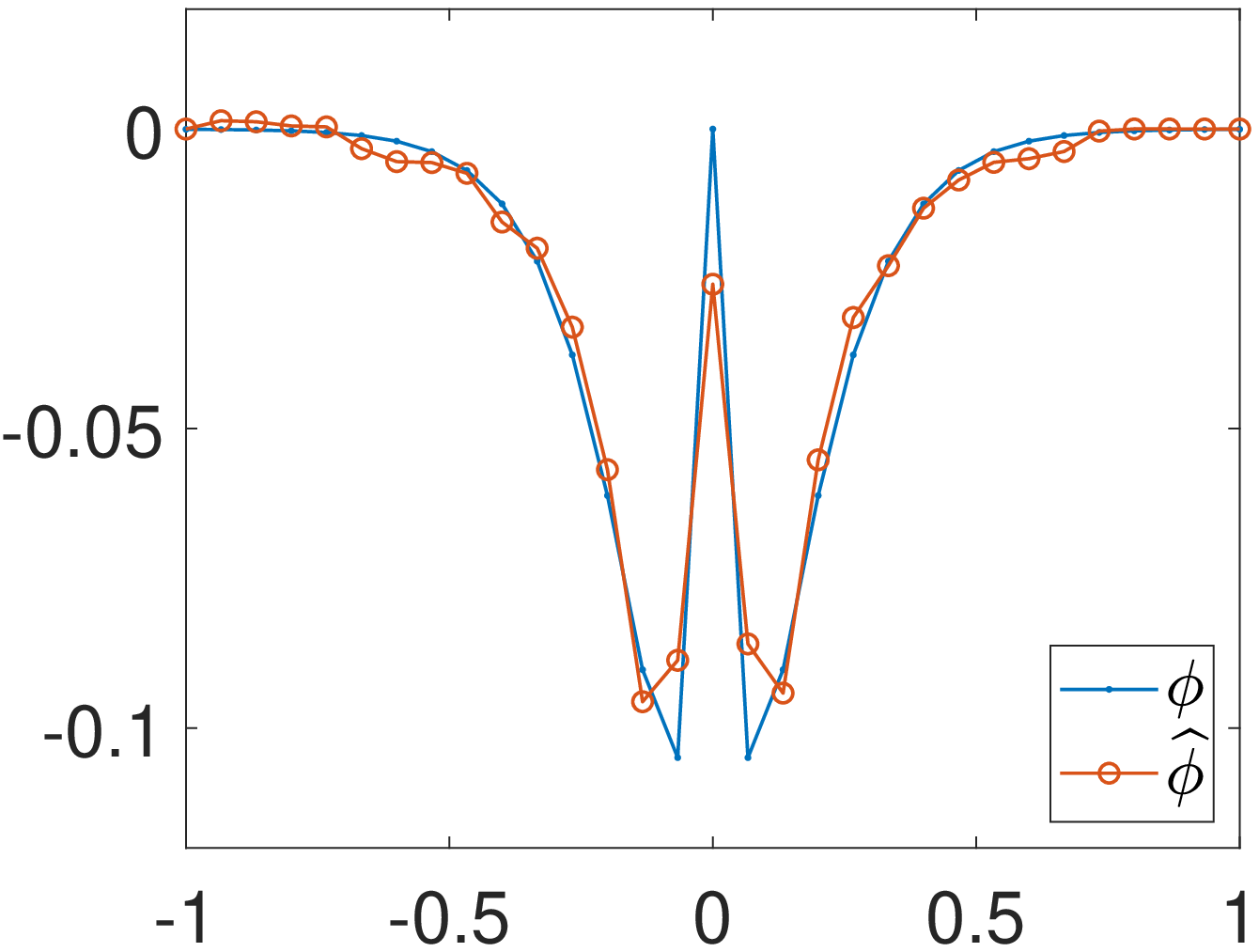}&
		\includegraphics[width=0.32\textwidth]{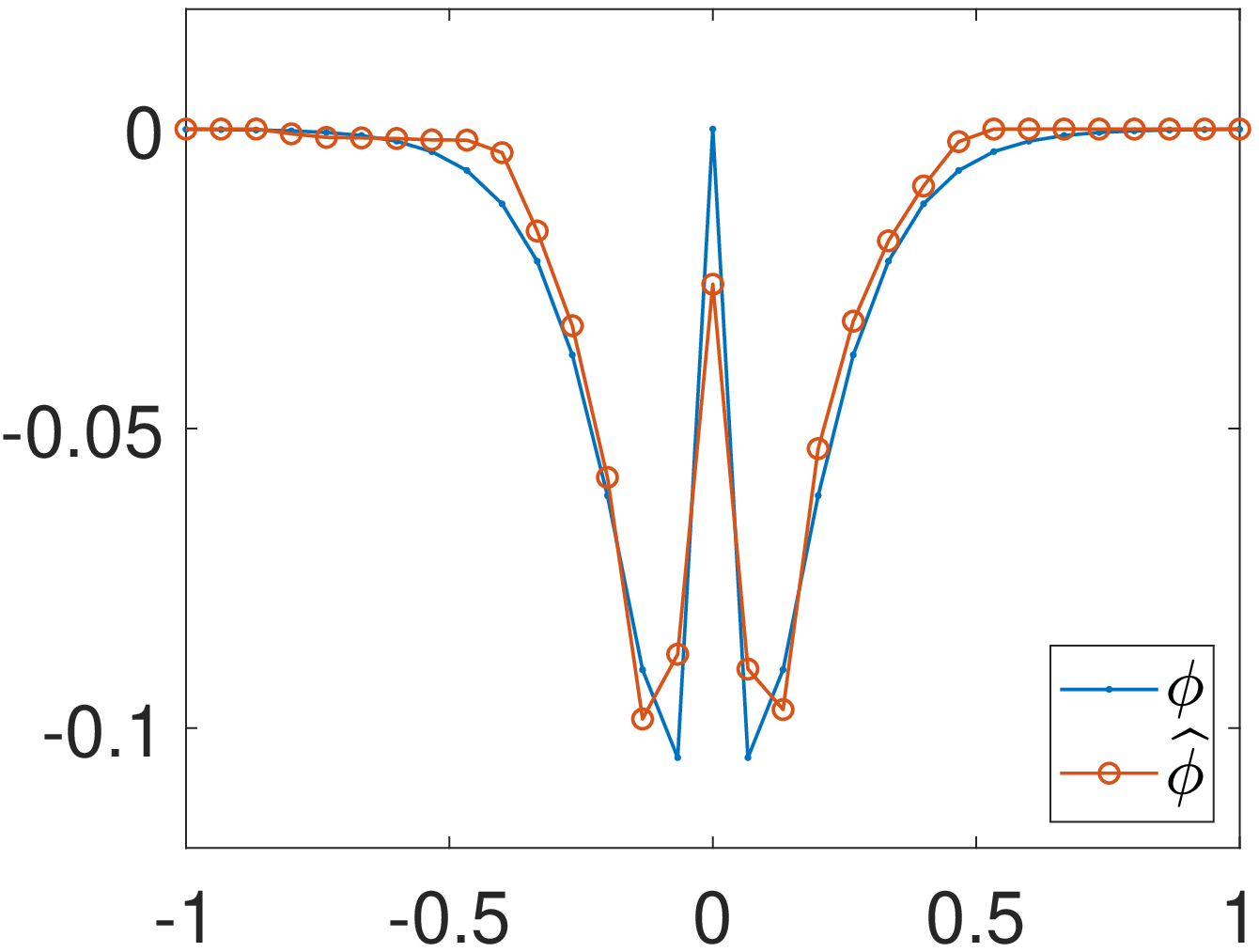} &
		\includegraphics[width=0.32\textwidth]{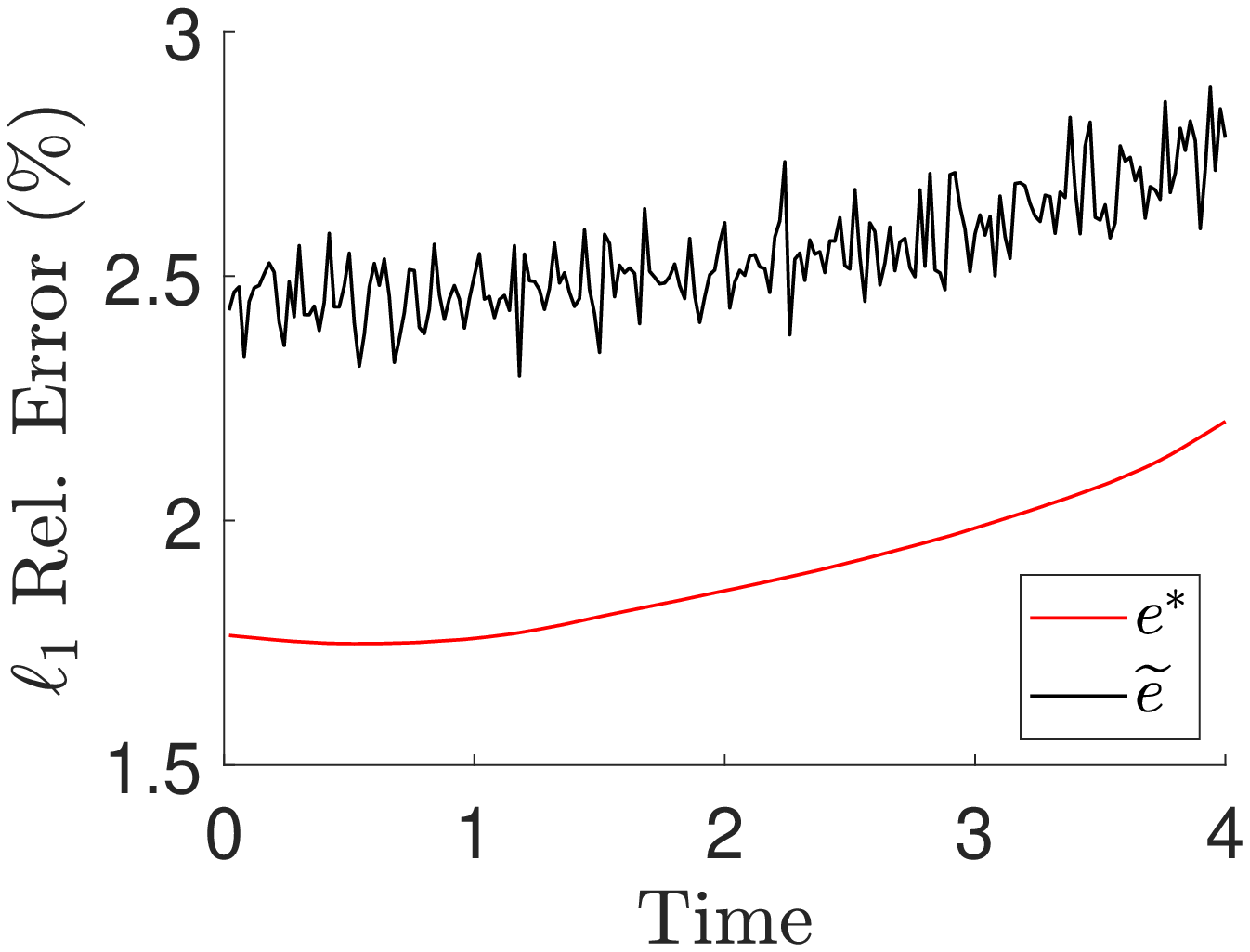}
	\end{tabular}
	\caption{2D example of the attraction-repulsion potential $\phi^*$ in (\ref{eq.numerical.2D.ar}). (a) The true potential $\phi^*$. (b) The identified potential $\widehat{\phi}$ with error  $e_\phi=18.44\%$. (c) Comparison of the cross sections along $x_1=0$. (d) Comparison of the cross sections along $x_2=0$. (e) The errors $e^*(t)$ (red) and $\tilde{e}(t)$ (black) as functions of $t$. The given data contain $1\%$ noise and we set $\alpha=2\times10^{-4},\beta=2\times10^{-7},\lambda=2$ and $\gamma=0$.}\label{fig.RINO.2d.ar}
\end{figure}

Figure \ref{fig.RINO.2d.aniso} shows the identification result of the following anisotropic potential
\begin{align}
	\phi^*=\frac{1}{5}\exp\left(-\frac{x_1^2+3x_2^2}{0.04}\right).
	\label{eq.numerical.2D.aniso}
\end{align}
In the experiment, we set $\alpha=2\times10^{-4},\beta=2\times10^{-7}$ and $\lambda=2$. The exact potential $\phi^*$ and the identified $\widehat{\phi}$ are shown in Figure \ref{fig.RINO.2d.aniso} (a) and (b), respectively. We compare their cross sections along $x_1=0$ and $x_2=0$ in (c) and (d), respectively. We observe that $\widehat{\phi}$ recovers the anisotropic property of $\phi^*$. Figure \ref{fig.RINO.2d.aniso} (e) shows the error $e^*(t)$ which is very small (less than $1.5\%$ for $t\in[0,4]$).
\begin{figure}[t!]
	\begin{tabular}{ccc}
		(a)&(b)&\\
		\includegraphics[width=0.32\textwidth]{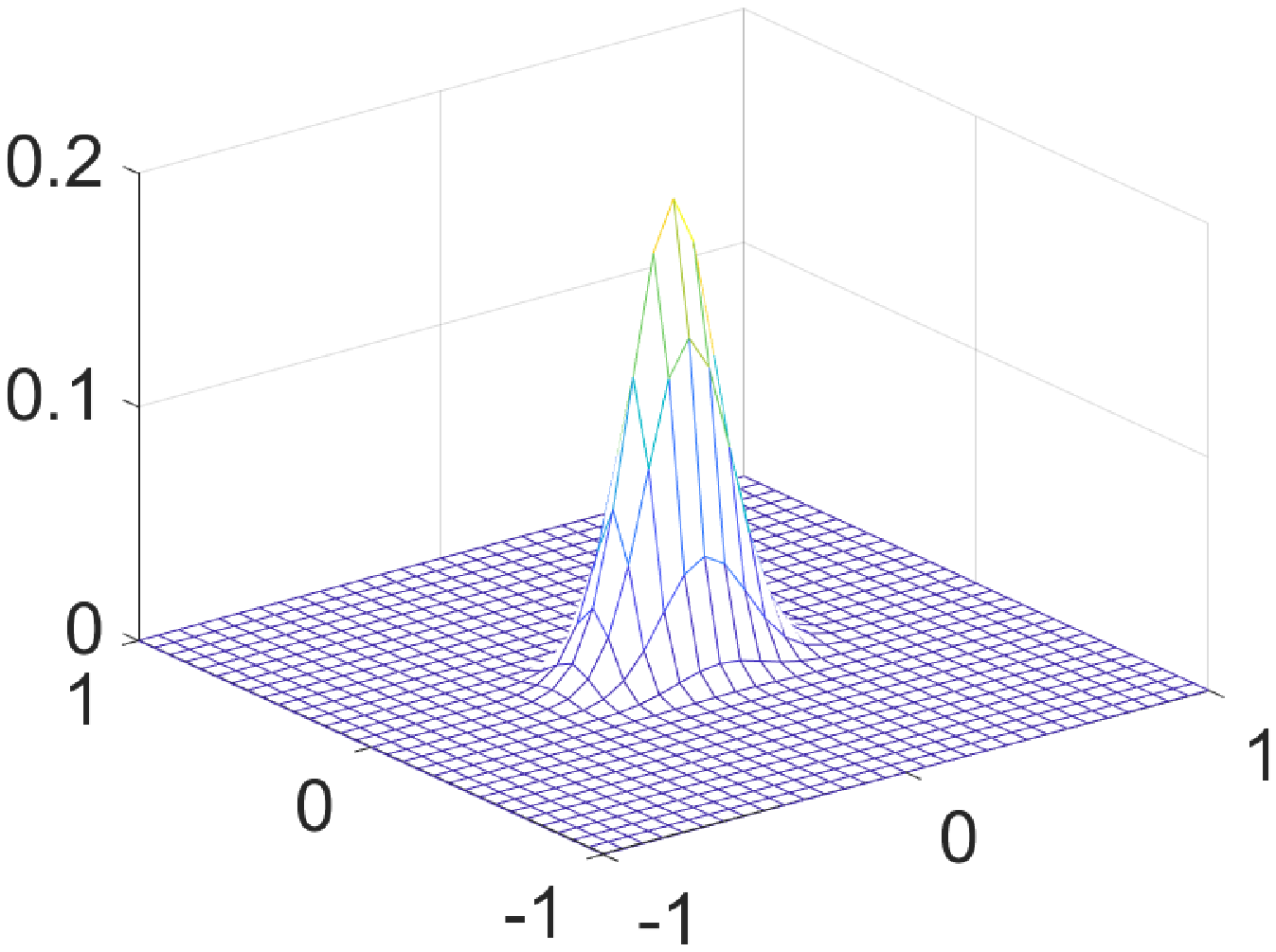}&
		\includegraphics[width=0.32\textwidth]{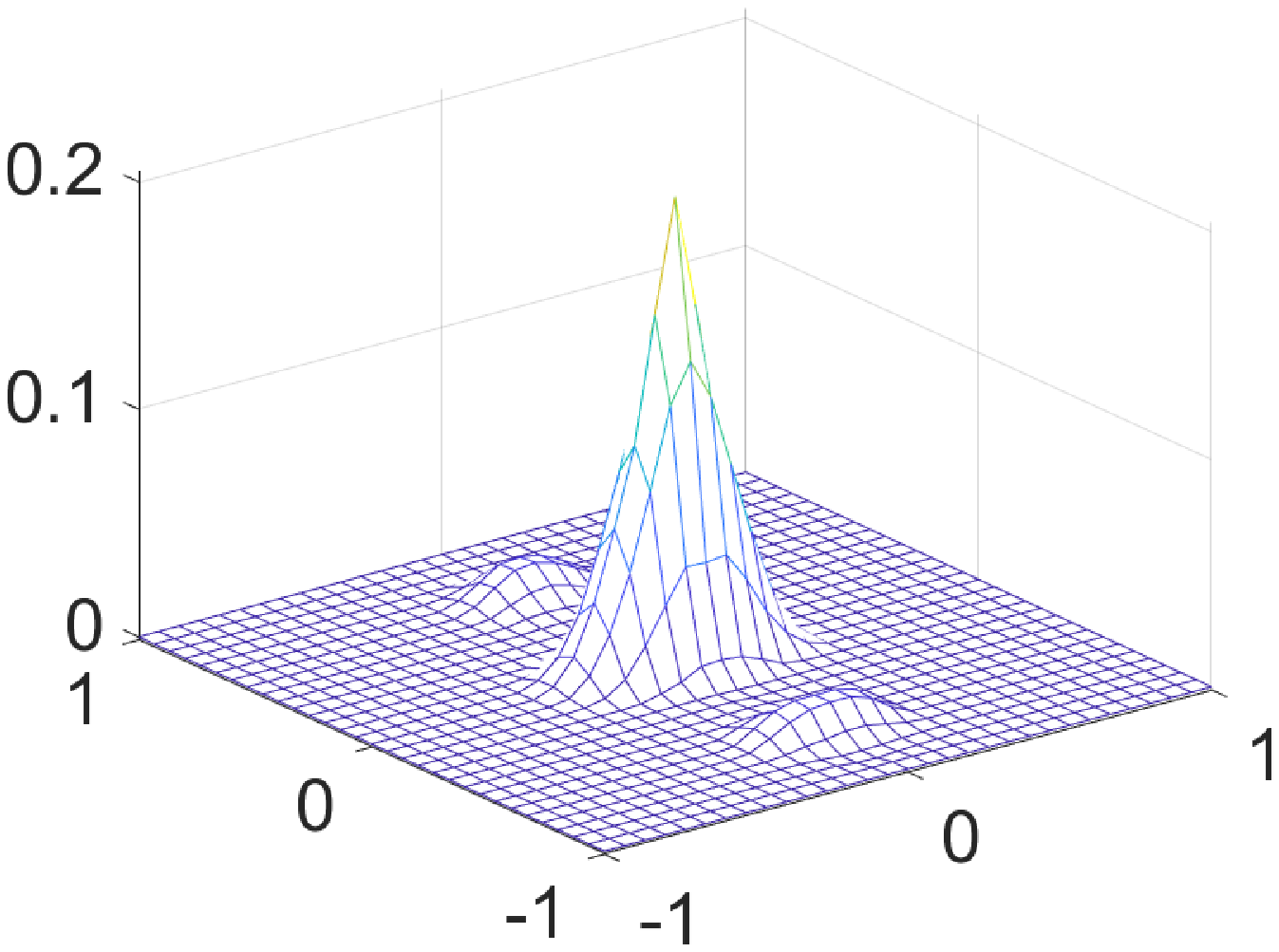} &\\
		(c)&(d)&(e)\\
		\includegraphics[width=0.32\textwidth]{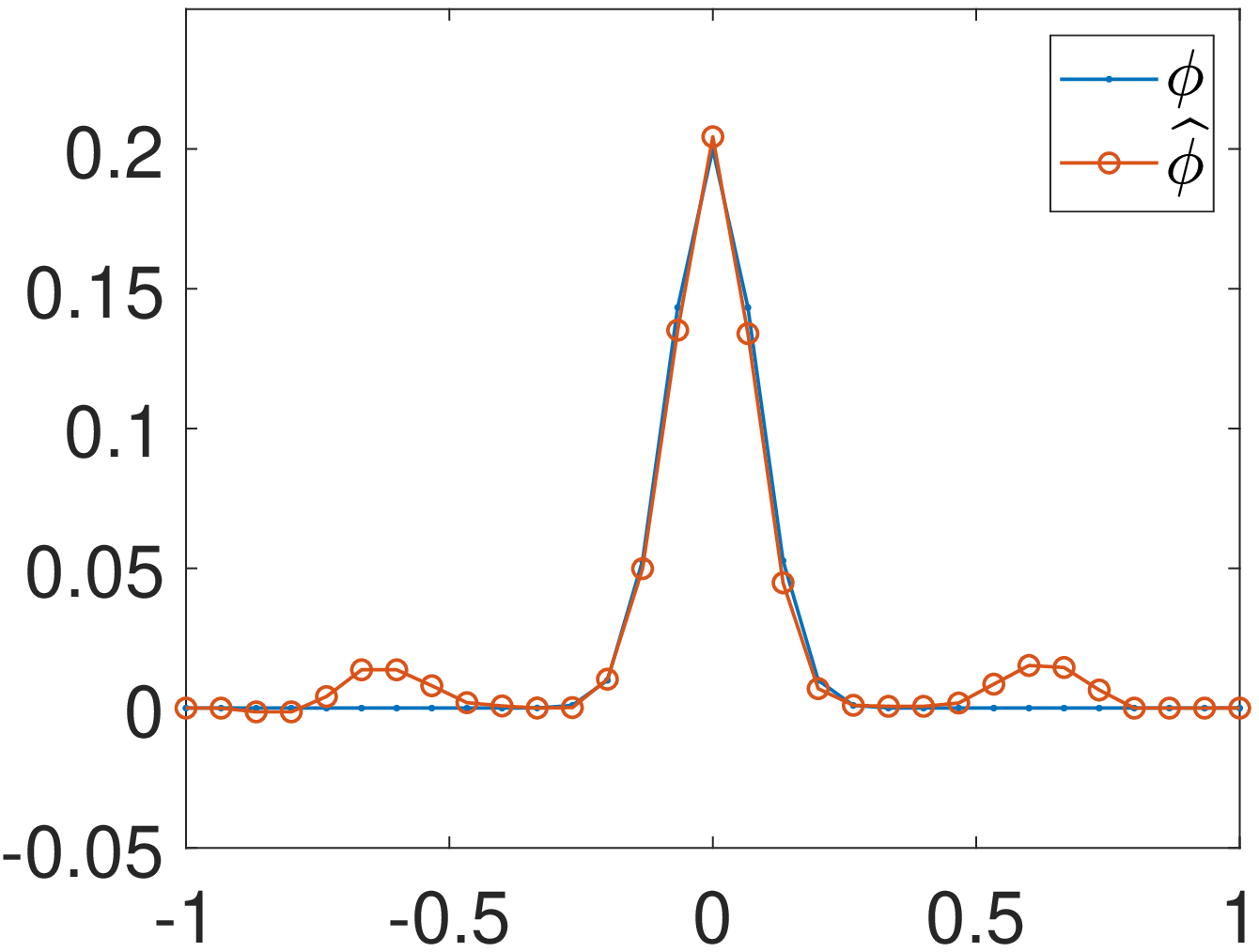}&
		\includegraphics[width=0.32\textwidth]{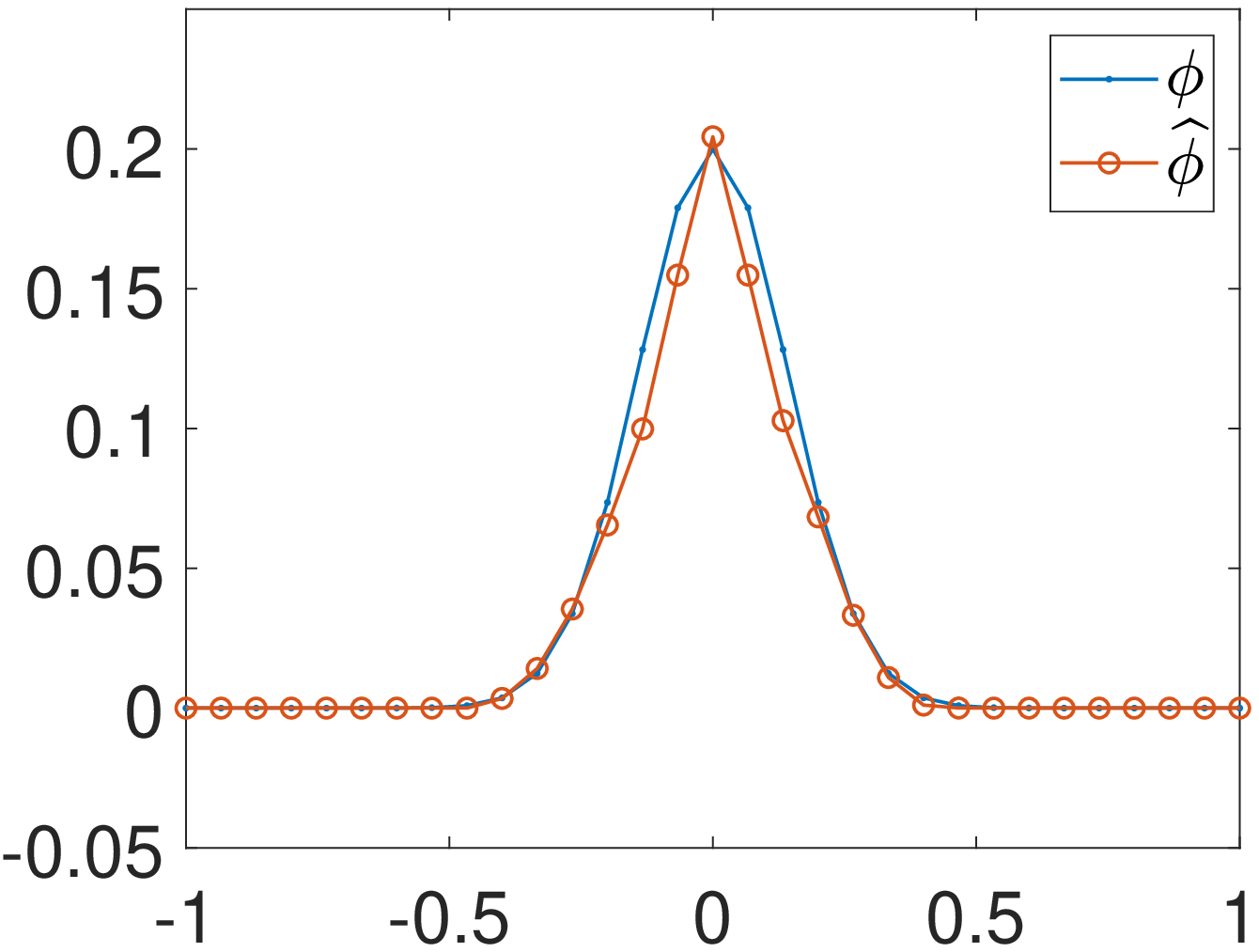} &
		\includegraphics[width=0.32\textwidth]{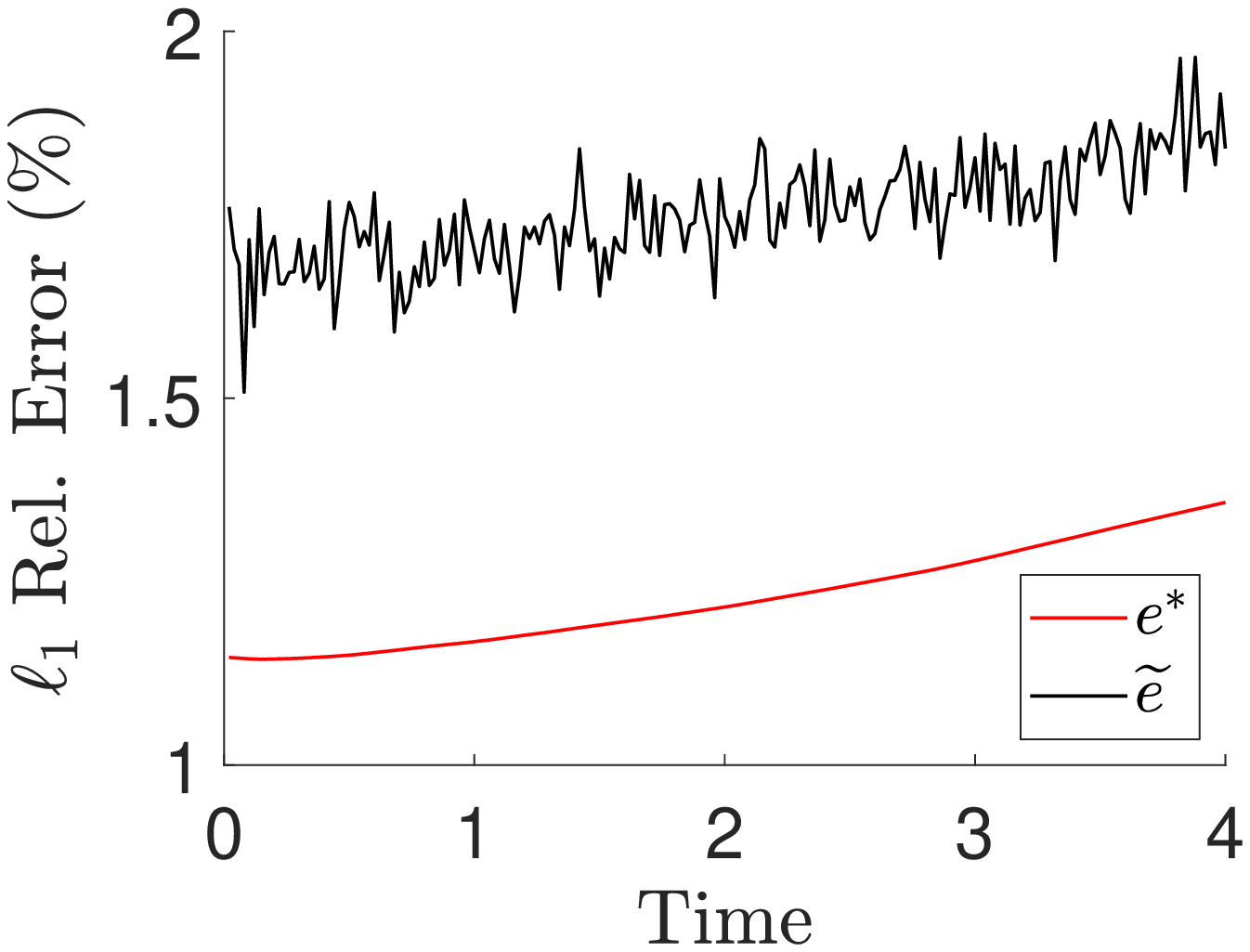}
	\end{tabular}
	\caption{2D example of the the anisotropic potential $\phi^*$ in (\ref{eq.numerical.2D.aniso}). (a) The true potential $\phi^*$. (b) The identified potential $\widehat{\phi}$ with error $e_\phi=22.87\%$. (c) Comparison of the cross sections along $x_1=0$. (d) Comparison of the cross sections along $x_2=0$. (e) The errors $e^*(t)$ (red) and $\tilde{e}(t)$ (black) as functions of $t$. The given data contain $1\%$ noise and we set $\alpha=2\times10^{-4},\beta=2\times10^{-7},\lambda=2$ and $\gamma=0$.
	}\label{fig.RINO.2d.aniso}
\end{figure}

\subsection{Symmetric potential  example}
We next demonstrate the effects of imposing the symmetry constraint on potentials as described in Section~\ref{sec_symm}. We consider the potential in (\ref{eq_RA}), with which the clean data is computed by solving  (\ref{eq.aggregation}) with $\Delta x=0.01, \Delta t=0.01$, and $T=3$. The noisy data are generated by adding $5\%$ Gaussian noise to the clean data.  We set $\alpha=1\times10^{-3}, \beta=1\times10^{-6},\gamma=20$ in Algorithm \ref{alg2} with the symmetry constraint, and use $\alpha=1\times10^{-3},\beta=5\times10^{-6},\gamma=10$ in Algorithm \ref{alg2} without the symmetry constraint. Figure~\ref{fig_SRINO} (a) compares the potential identified with (red) and without (blue) the symmetry constraint. The identified potential with the symmetry constraint approximates the exact potential better than that without the constraint. Such a constraint provides additional regularization which averages the noise on the negative axis and the positive axis. The error $e^*(t)$ is shown in Figure~\ref{fig_SRINO} (b).

\begin{figure}[h!]
	\begin{center}
		\begin{tabular}{cc}
			(a)&(b)\\
			\includegraphics[width=0.4\textwidth]{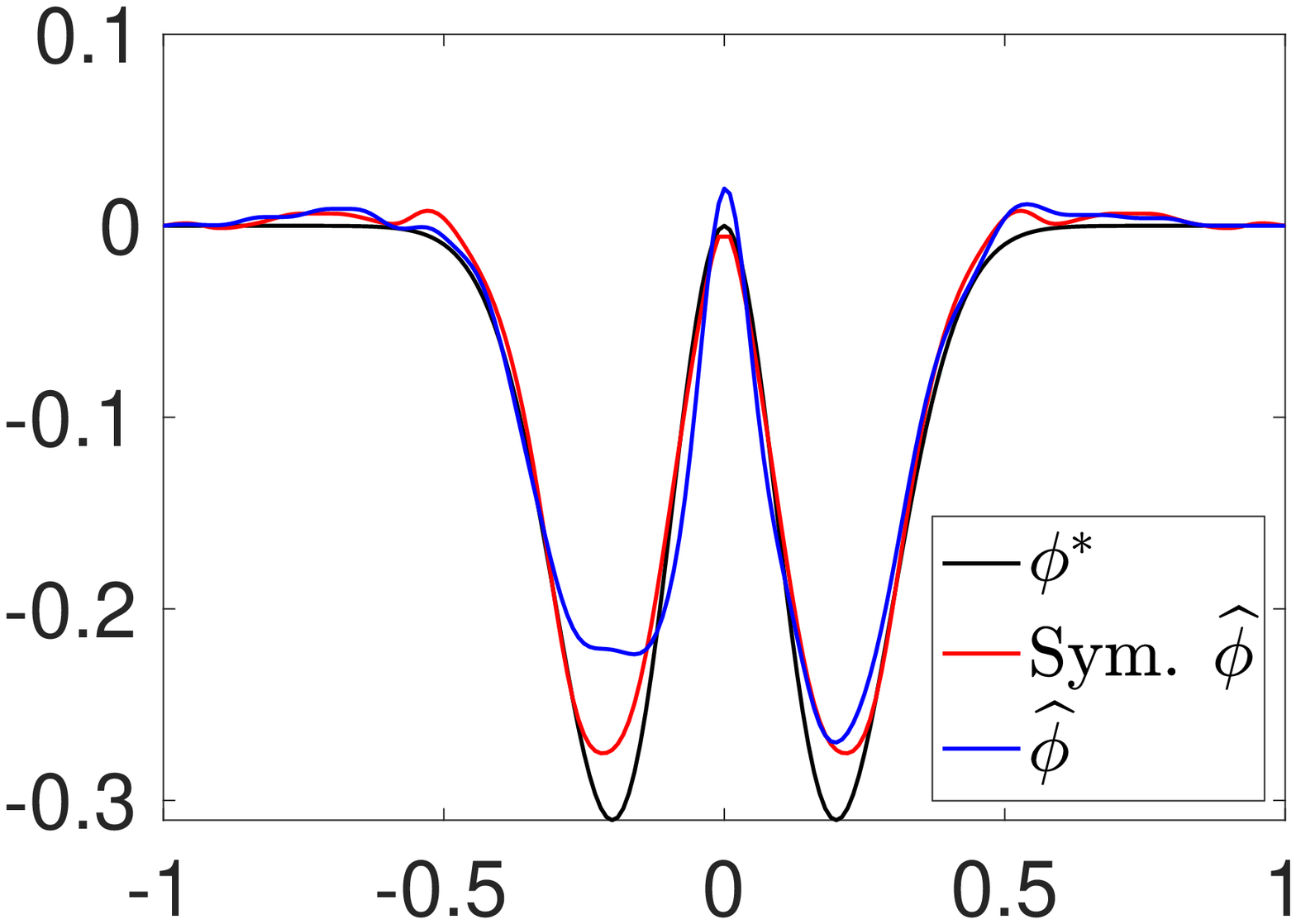}&
			\includegraphics[width=0.4\textwidth]{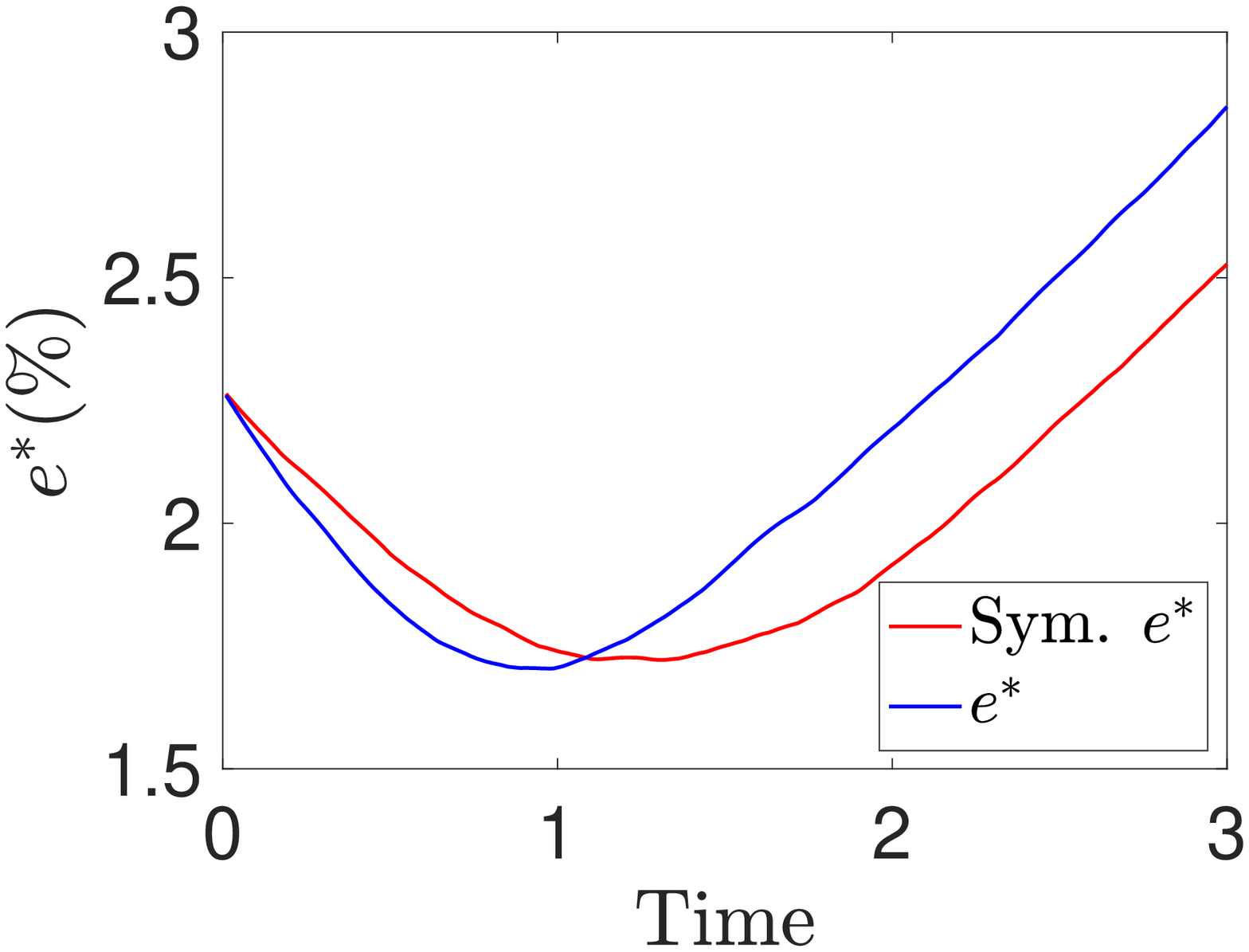}
		\end{tabular}
	\end{center}
	\caption{Effect of imposing the symmetry constraint for the identification of the repulsive-attractive potential in (\ref{eq_RA}). (a) The identified potential with (red) and without (blue) the symmetry constraint, compared to the true potential (black).  (b) The errors $e^*(t)$ as a function of $t$. The given data contain $5\%$ Gaussian noise. We set $\alpha=1\times10^{-3}, \beta=1\times10^{-6}$,  $\gamma=20$ for Algorithm \ref{alg2} with the symmetry constraint.  Without the symmetry constraint, we use $\alpha=1\times10^{-3},\beta=5\times10^{-6}$ and $\gamma=10$. 
		The symmetry constraint improves the identification result in comparison with Figure~\ref{exp1}. }\label{fig_SRINO}
\end{figure}

\subsection{Comparison of different regularization}
Regularization is important in stabilizing the potential identification from noisy data.
As discussed in Section~\ref{sec.model}, our choice of $|\nabla\phi|$ and $|\nabla^2\phi|^2$ is motivated by their physical meanings.
In this section, we justify our choice by numerical experiments.

We consider the following eight choices of regularizations, where $\alpha$ and $\beta$ represent the weight parameters as in (\ref{eq.min.reg}):
\begin{enumerate}
	\item $\alpha|\nabla\phi|$: The well-known TV regularization~\cite{chan1998total}, which tends to produce  piecewise constant recovery.
	\item $\frac{\alpha}{2}|\nabla\phi|^2$: The most classical squared $L^1$-norm regularization, which promotes smoothness in recovery.
	\item $\beta|\nabla^2\phi|$: The $L^1$-norm of the second order derivative, which has been explored in   \cite{chan2007image,bergounioux2010second} to reduce the staircase effect resulted from the TV regularization.
	\item $\frac{\beta}{2}|\nabla^2\phi|^2$: A second-order regularizer considered in nonlinear diffusion filters~\cite{didas2009properties}.
	\item $\alpha|\nabla\phi|+\beta|\nabla^2\phi|$: This combination contains the $L^1$-norms of the first and second order derivatives, which has been studied in~\cite{papafitsoros2014combined} for image deblurring and inpainting.
	\item $\alpha|\nabla\phi|+\beta|\nabla^2\phi|^2$: This is our proposed regularization in (\ref{eq.min.reg}).
	\item $\frac{\alpha}{2}|\nabla\phi|^2+\beta|\nabla^2\phi|$: This is a mixed-type regularizer which is not common in the literature. We include it for a comparison.
	\item $\frac{\alpha}{2}|\nabla\phi|^2+\frac{\beta}{2}|\nabla^2\phi|^2$: This combination contains the squared $L^1$-norms of the first and second order derivatives which strongly promotes smoothness.
\end{enumerate}

As for the choice of weight parameters, we test all weight parameters in the following lists:
\begin{align*}
	\alpha &\in \{10^{-6},5\times10^{-6},10^{-5},\dots,5\times 10^{-2}\}\\
	\beta &\in \{10^{-7},5\times10^{-7},10^{-6},\dots,5\times 10^{-3}\}
\end{align*}
and choose the one which minimizes the averaged $e^*$ error over time, i.e., $\frac{1}{N}\sum_{n=1}^Ne^*(t^n)$. We use the split Bregman algorithm whenever the $L^1$-type regularizer is used.  For all of the other cases,  such as the $L^2$ type,  the regularizers are smooth functions of $\phi$ that the proposed functional is minimized by solving a linear system.   All experiments are conducted with $\gamma=0$ and without the symmetry constraint. 

We consider the repulsive-attractive potential in \eqref{eq_RA} with four different choices of parameters, as shown in Figure \ref{fig_compare.potential} (a)-(d). In Figure~\ref{fig_compare.potential} (a) and (b), the potentials have singularities at the origin, corresponding to the condition that no collision occurs, which is commonly assumed in flock modeling~\cite{carrillo2014derivation}.

Figure \ref{fig_compare} (a) - (d) shows the identification results for the four potentials in Figure \ref{fig_compare.potential} (a) - (d) respectively. For each potential, we report the error $e^*(t)$ and its averaged value over time, when the eight choices of regularizations are used.
For potentials with singularities, the results in (a) and (b) show that the TV regularization is helpful in identifying such potentials.
When TV regularization is used, the  errors are reduced if TV is combined with $|\nabla^2\phi|$ or $|\nabla^2\phi|^2$.
%
This set of experiments shows that, the optimal choice of regularization depends on the underlying potential.    We pick the regularizer  $\alpha|\nabla\phi|+\beta|\nabla^2\phi|^2$, since it gives good results in general.
\begin{figure}[tb!]
	\centering
	\begin{tabular}{c@{\hspace{2pt}}c@{\hspace{2pt}}c@{\hspace{2pt}}c}
		(a)&(b)&(c)&(d)\\
		\includegraphics[width=0.24\textwidth]{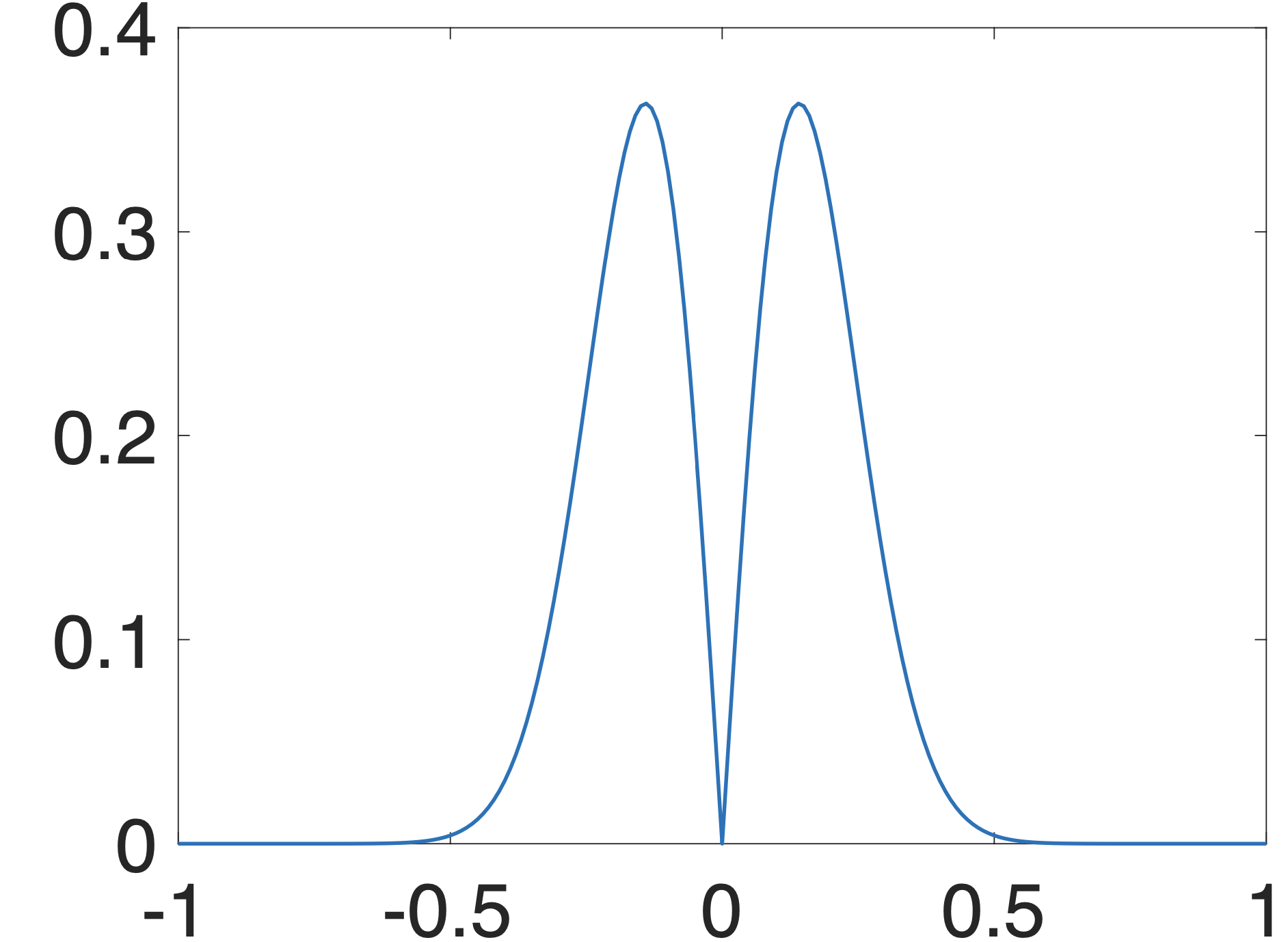}&
		\includegraphics[width=0.24\textwidth]{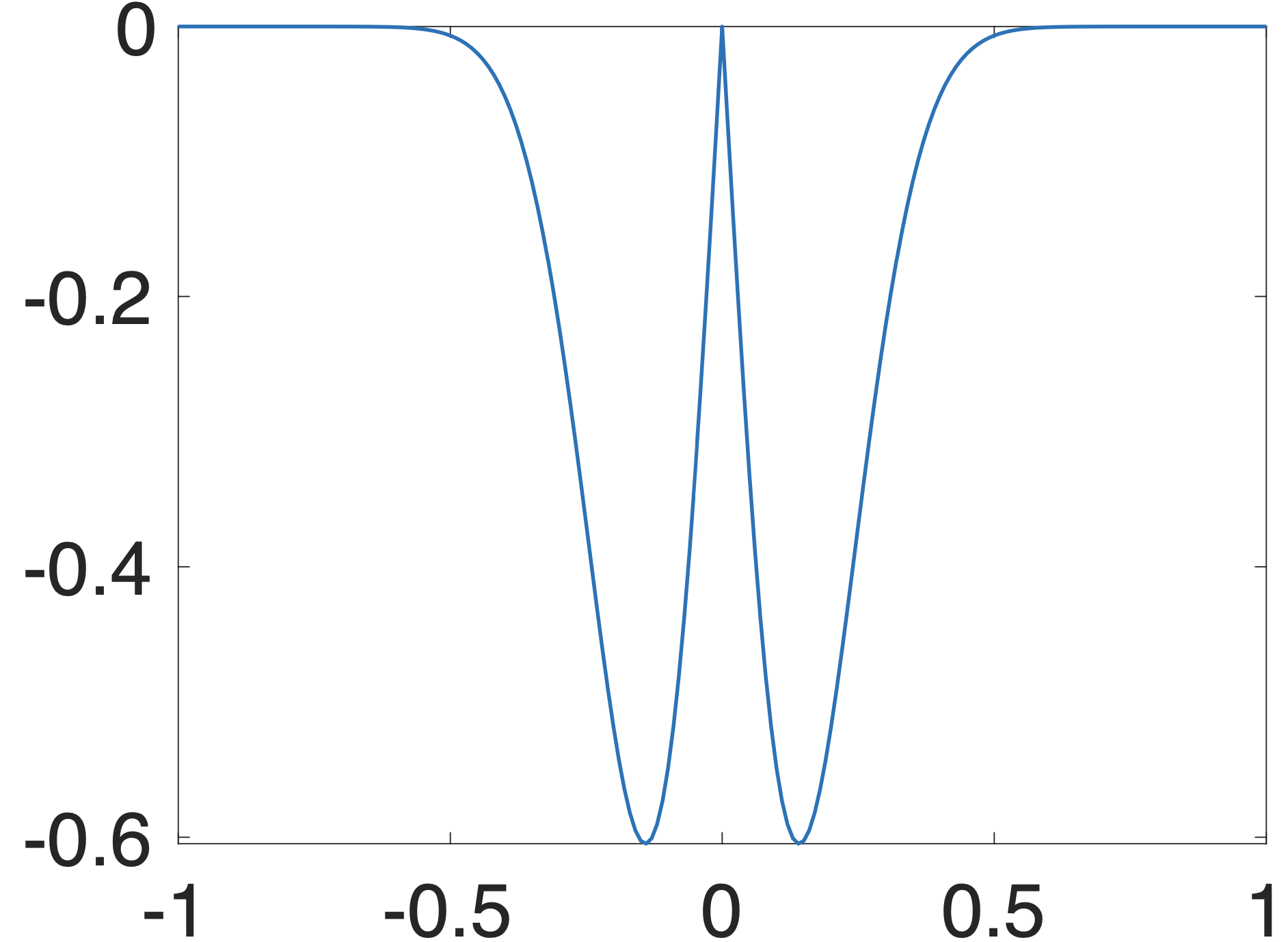}&
		\includegraphics[width=0.24\textwidth]{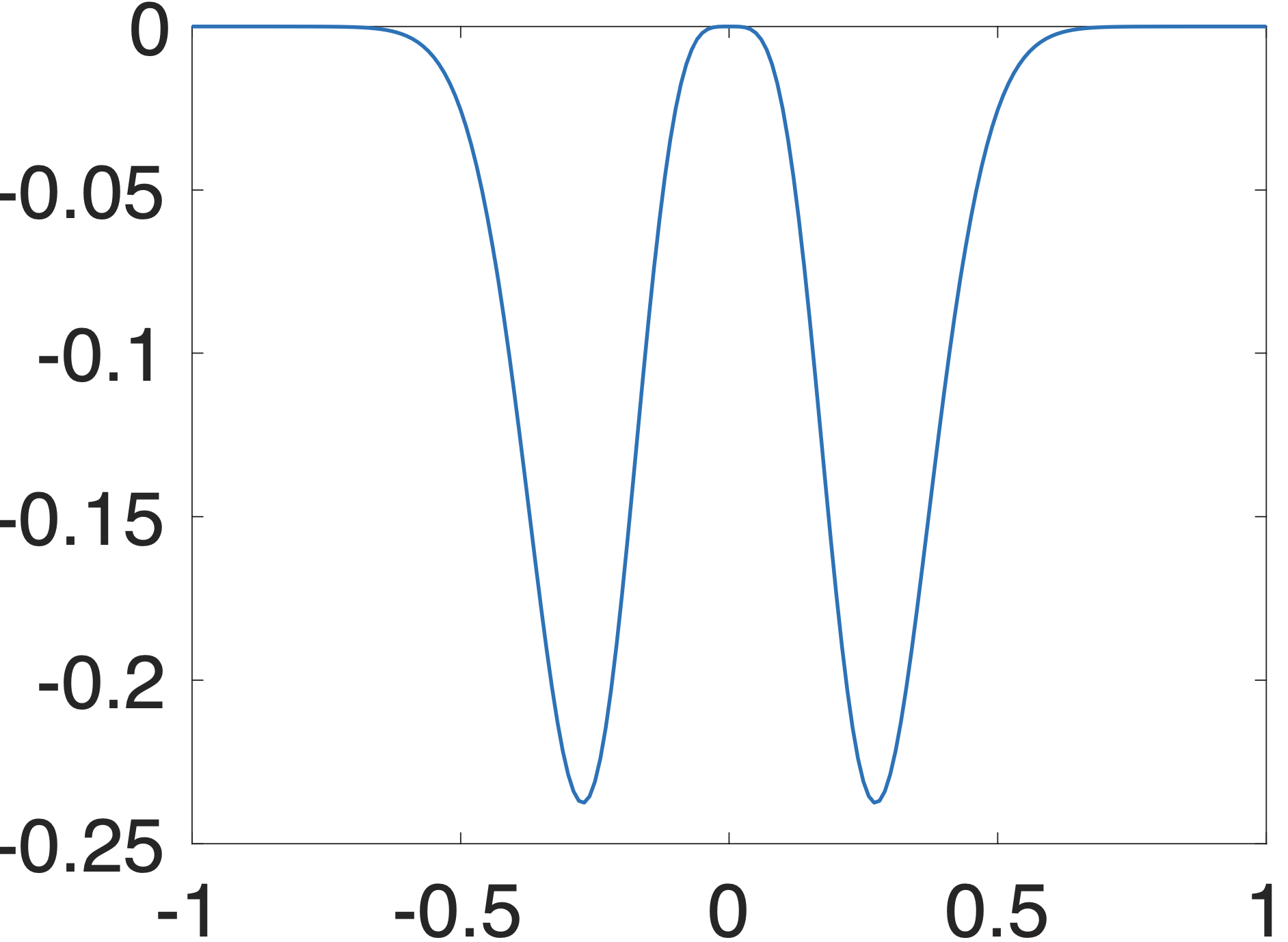}&
		\includegraphics[width=0.24\textwidth]{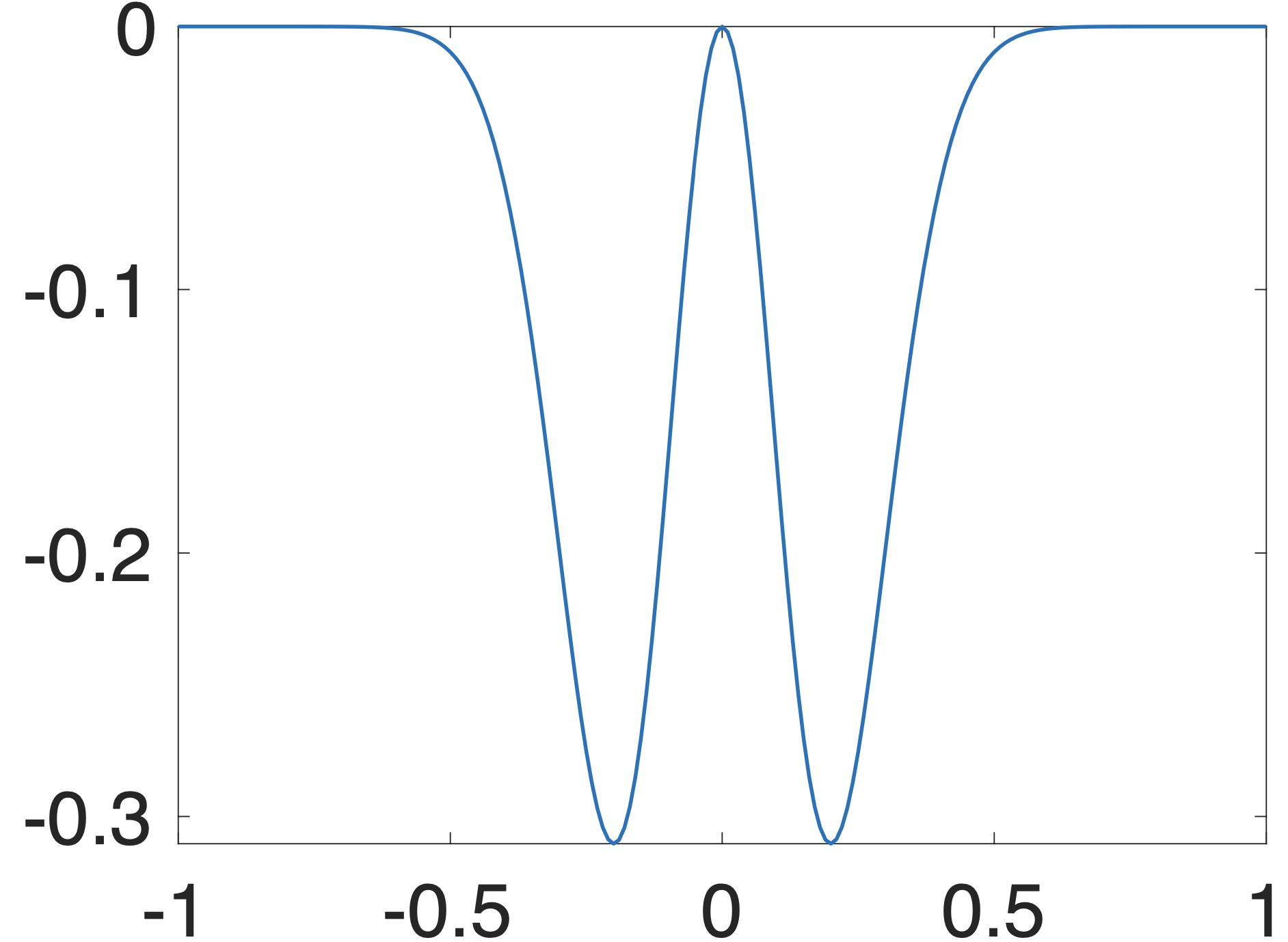}
	\end{tabular}
	\caption{Repulsive-attractive potentials in \eqref{eq_RA} with different parameters: (a) $\phi_{RA}(x;1,5,1.5)$, (b) $\phi_{RA}(x;5,1,2.5)$, (c) $\phi_{RA}(x;5,4,500)$, and (d) $\phi_{RA}(x;5,2,15)$.
	}\label{fig_compare.potential}
\end{figure}

\begin{figure}[t!]
	\centering
	\begin{tabular}{cc}
		(a)&(b)\\
		Averaged $e^* (\%)$ of Figure \ref{fig_compare.potential}(a)&Averaged $e^* (\%)$ of Figure \ref{fig_compare.potential}(b)\\
		\begin{tabular}{c|c| c|c|c}
			\hline
			&$|\nabla\phi|$&$|\nabla^2\phi|$&$|\nabla\phi|^2$& $|\nabla^2\phi|^2$\\\hline
			$|\nabla\phi|$&$2.42$ &$\mathbf{1.22}$&--- &$  2.00$\\\hline
			$|\nabla^2\phi|$& --- & $1.29$& $1.29$&---\\\hline
			$|\nabla\phi|^2$&--- & ---&$3.02$&$2.47$\\\hline
			$|\nabla^2\phi|^2$&--- & ---&--- &$2.47$\\\hline
		\end{tabular}&
		\begin{tabular}{c|c|c|c|c}
			\hline
			&$|\nabla\phi|$&$|\nabla^2\phi|$&$|\nabla\phi|^2$& $|\nabla^2\phi|^2$\\\hline
			$|\nabla\phi|$&$1.9418$ &$ 2.0677$&--- &$\mathbf{1.8723}$\\\hline
			$|\nabla^2\phi|$& --- & $2.0267$& $2.0265$&---\\\hline
			$|\nabla\phi|^2$&--- & ---&$2.1611$&$2.1593$\\\hline
			$|\nabla^2\phi|^2$&--- & ---&--- &$2.1593$\\\hline
		\end{tabular}\\
		\includegraphics[width=0.44\textwidth]{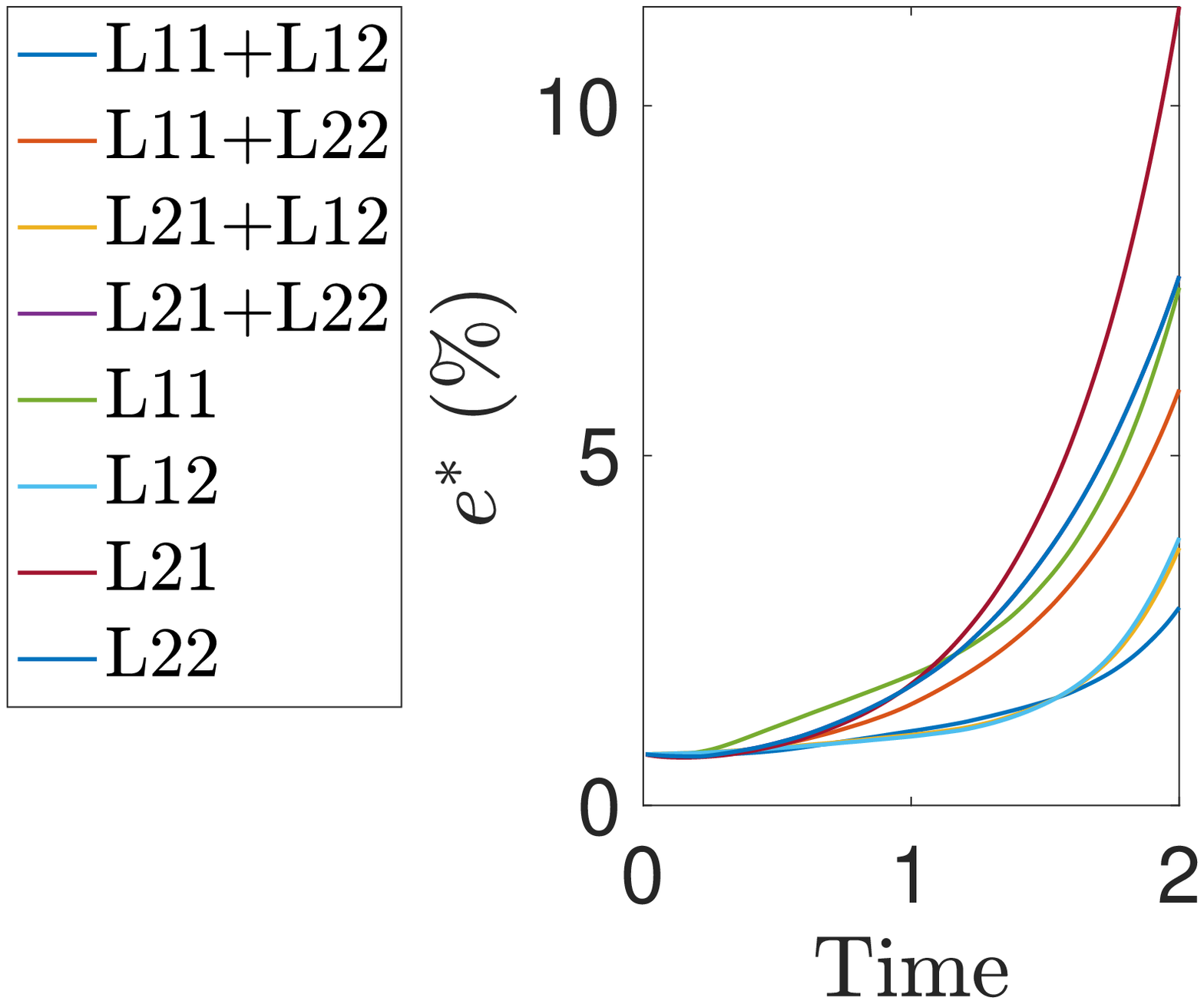}&	\includegraphics[width=0.44\textwidth]{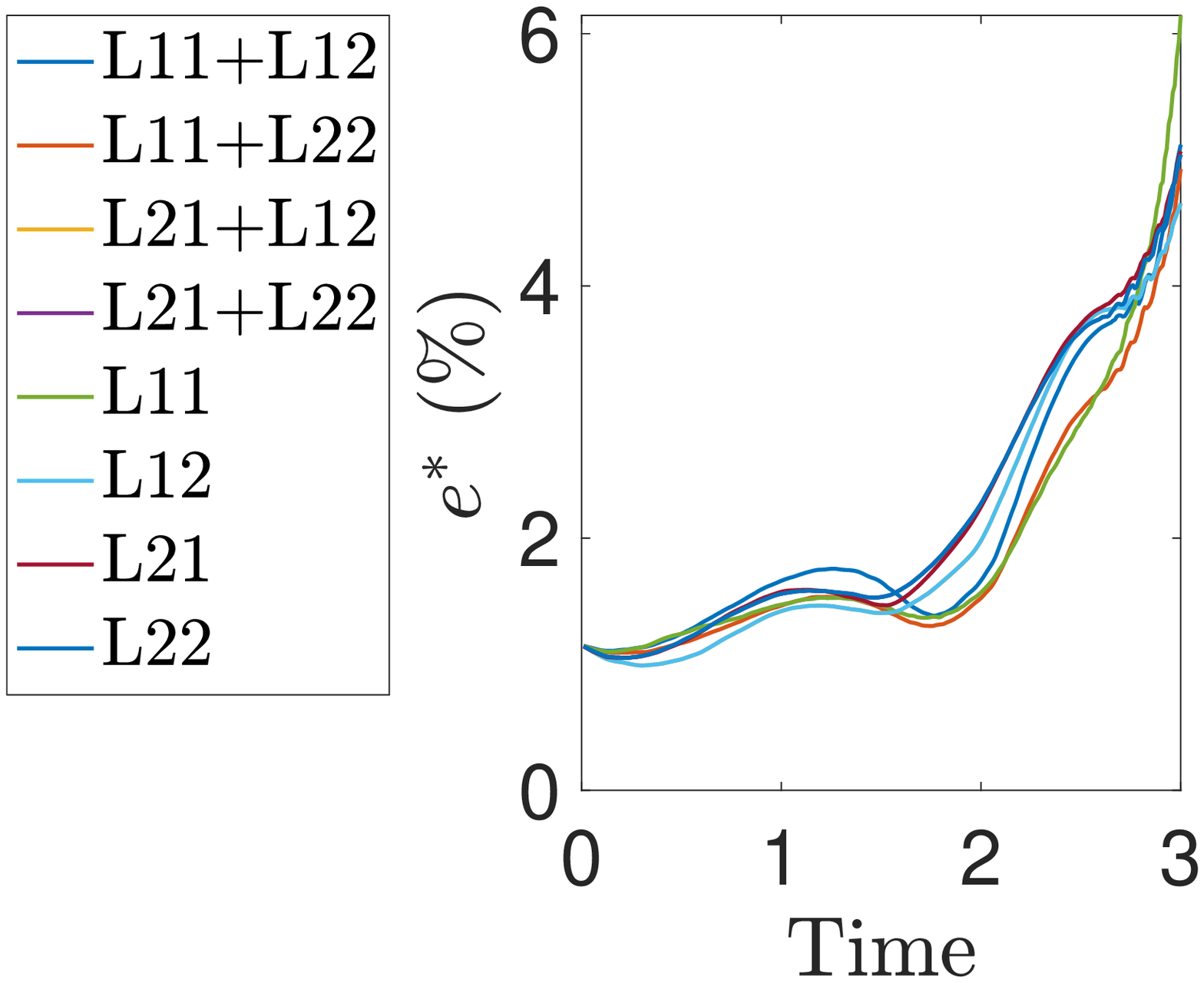}\\
		(c)&(d)\\
		Averaged $e^* (\%)$ of Figure \ref{fig_compare.potential}(c)&Averaged $e^* (\%)$ of Figure \ref{fig_compare.potential}(d)\\
		\begin{tabular}{c|c|c|c|c}
			\hline
			&$|\nabla\phi|$&$|\nabla^2\phi|$&$|\nabla\phi|^2$& $|\nabla^2\phi|^2$\\\hline
			$|\nabla\phi|$&$1.31$ &$1.52$&--- &$1.30$\\\hline
			$|\nabla^2\phi|$& --- & $1.33$& $1.33$&---\\\hline
			$|\nabla\phi|^2$&--- & ---&$\mathbf{1.28}$&$1.29$\\\hline
			$|\nabla^2\phi|^2$&--- & ---&--- &$1.31$\\\hline
		\end{tabular}&
		\begin{tabular}{c|c|c|c|c}
			\hline
			&$|\nabla\phi|$&$|\nabla^2\phi|$&$|\nabla\phi|^2$& $|\nabla^2\phi|^2$\\\hline
			$|\nabla\phi|$&$0.8198$ &$ 1.4997$&--- &$0.8246$\\\hline
			$|\nabla^2\phi|$& --- & $1.0997$& $1.0998$&---\\\hline
			$|\nabla\phi|^2$&--- & ---&$\mathbf{0.7329}$&$0.7464$\\\hline
			$|\nabla^2\phi|^2$&--- & ---&--- &$0.7463$\\\hline
		\end{tabular}\\
		\includegraphics[width=0.44\textwidth]{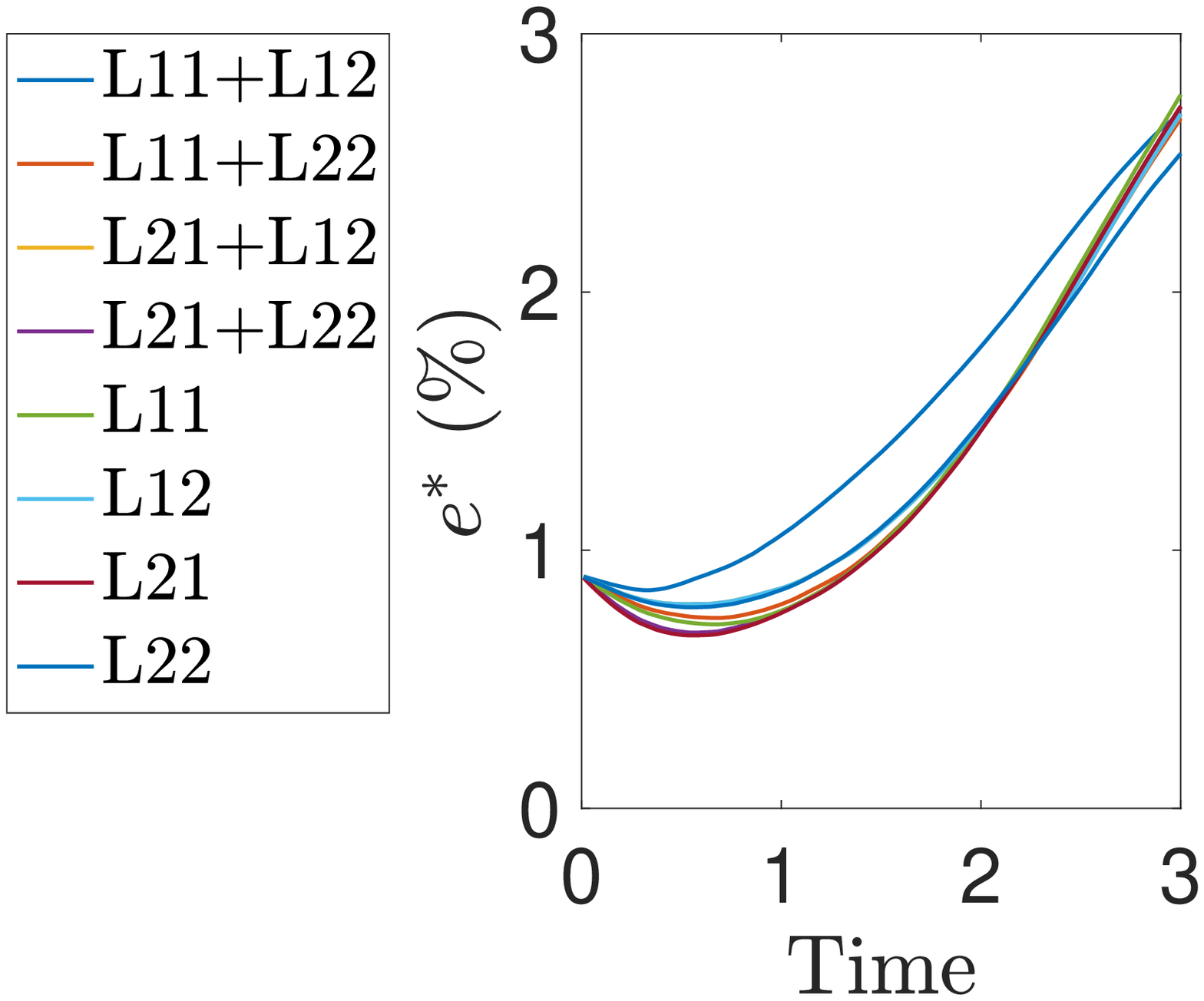}&	\includegraphics[width=0.44\textwidth]{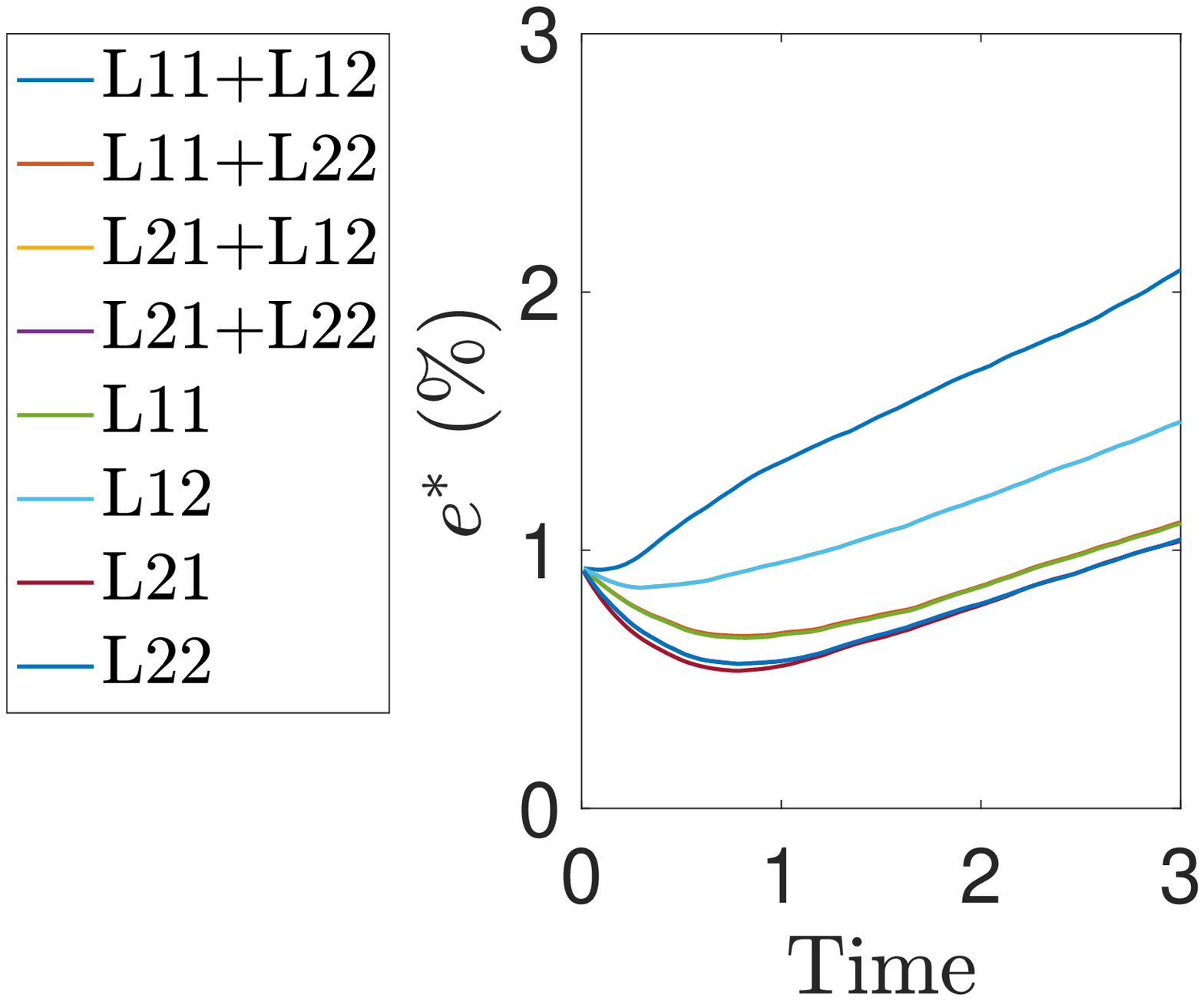}
	\end{tabular}
	\caption{Comparison of different regularizations. (a)-(d): The error $e^*(t)$ and its average over time for the identification of the potentials in Figure \ref{fig_compare.potential} (a)-(d), respectively. In the table, each entry shows the averaged $e^*(t)$ error with the combination of the column and row regularizers.
		The minimum error is bold.  In the legend of the figures, L{pq} denotes  the regularizer $|\nabla^q\phi|^p$ for integers $p,q>0$.
	}\label{fig_compare}
\end{figure}

\section{Identification of time-varying potentials}\label{sec_time}
As commonly observed in nature, many rules of interactions (or the potentials in (\ref{eq.aggregation})) change over time. For instance, for \textit{Temnothorax} ants, when the group density reaches at $0.8417$ ants per centimeter squared, the population behavior switches from tandem running to transporting~\cite{pratt2005behavioral}.  In this section, we propose a splitting-and-merge method to identify a time-varying potential.

\subsection{A splitting-and-merge method}
In our proposed method, we evenly divide the time interval $[0,T]$ into $Q$ subintervals: $\{[t^{(q)},t^{(q+1)}]\}_{q=1}^Q$. In each subinterval, we identify a time-independent potential $\phi^{(q)}$  as an approximation of the time-varying potential. Specifically, $\phi^{(q)}$ in the $q$-th subinterval  is identified by solving (\ref{eq.min.reg}):
\begin{align*}
	\phi^{(q)}=\argmin_{\phi} \left[\frac{1}{2}\int_{t^{(q)}}^{t^{(q+1)}} \int_{\Omega} (u_t-L_u\phi)^2 d\bx dt+  \alpha \int_{\Omega} |\nabla\phi| d\bx + \frac{\beta}{2} \int_{\Omega} |\nabla^2\phi|^2 d\bx\right].
\end{align*}
A time-varying potential is then constructed by gluing $\{\phi^{(q)}\}_{q=1}^Q$ together by a kernel function
\begin{align}
	\phi(t,\bx) = \sum_{q=1}^Q K_{h}\left(t-\frac{t^{(q)}+t^{(q+1)}}{2}\right)\phi^{(q)}(\bx),
	\label{eq.timevarying}
\end{align}
where $K_{h}(t)=h^{-1}K(t/h)$ for some  kernel function $K: \mathds{R}\rightarrow \mathds{R}^+$ with a compact support and $h>0$ is a bandwidth parameter. In the case when $K_{h}$ is a hat function, (\ref{eq.timevarying}) becomes a linear interpolation of the $\phi^{(q)}$'s. Here we take the Epanechnikov kernel~\cite{duong2015spherically}: $K(t) = \frac{3}{4C}(1-t^2)\mathds{1}_{\{|t|\leq 1\}}$, where $C$ is a constant such that $\int_\mathbb{R}K(t)\,dt=1$.

We can also partition the time interval into subintervals with overlaps to better utilize the data.
In this case, $\phi^{(q)}$ is identified from the data in the time interval  $[(q-1-\rho)T/Q, (q+\rho)T/Q]$ for $q=2,3,\dots, Q-1$, where $0<\rho<1$ represents an overlapping ratio. When $q=1$ (resp. $q=Q$), $\phi^{(q)}$ is identified from the data in the time interval $[0,(1+\rho)T/Q]$ (resp. $[(Q-1-\rho)T/Q,T]$). We then construct the time-varying potential using (\ref{eq.timevarying}).

\subsection{Numerical experiments on time-varying potentials}



%
%
%
%
\begin{figure}[h!]
	\begin{center}
		\begin{tabular}{c@{\hspace{2pt}}c}
			(a)&(b)\\
			\includegraphics[width=0.4\textwidth]{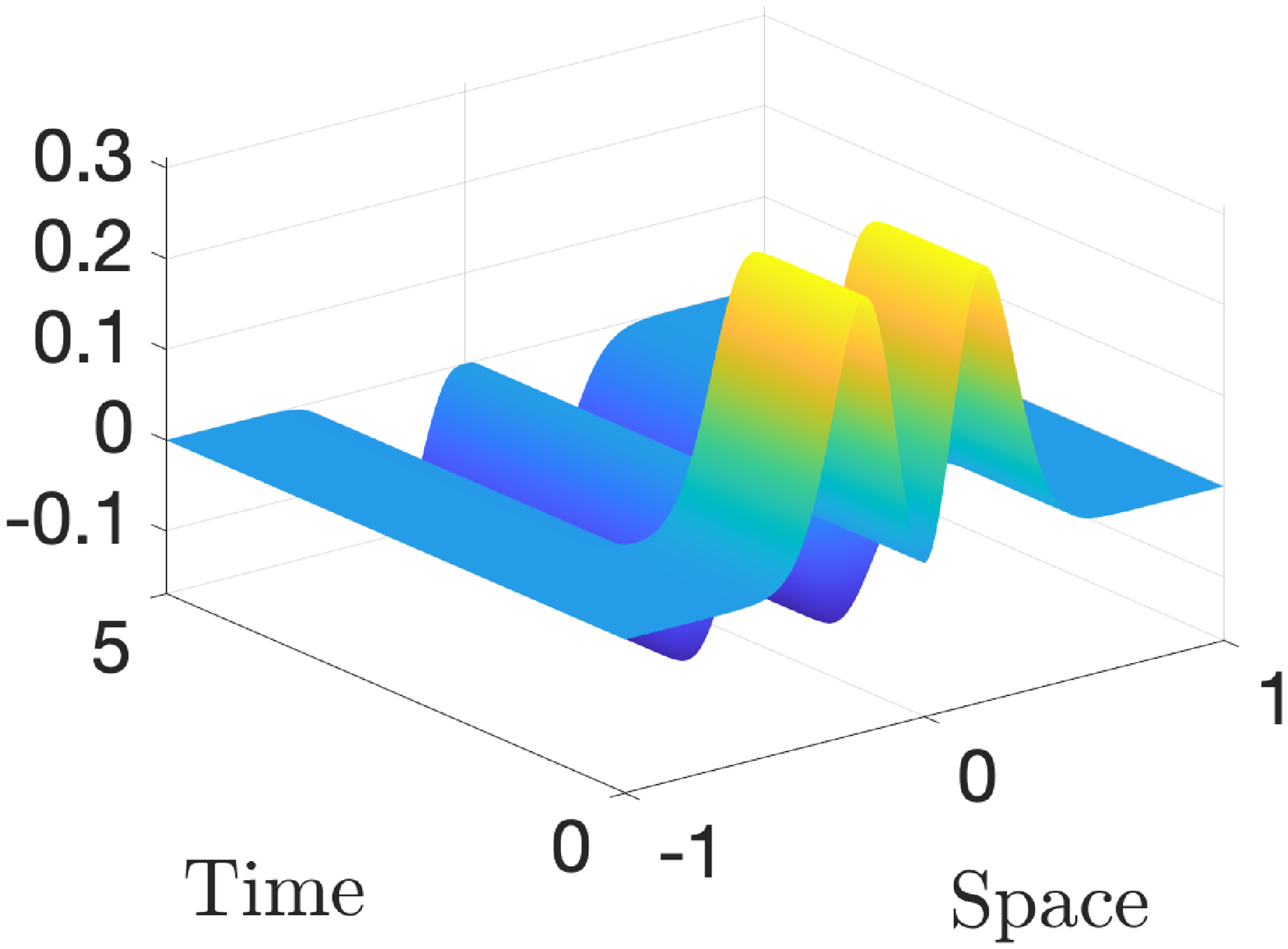}\hspace{0.5cm}	&
			\includegraphics[width=0.4\textwidth]{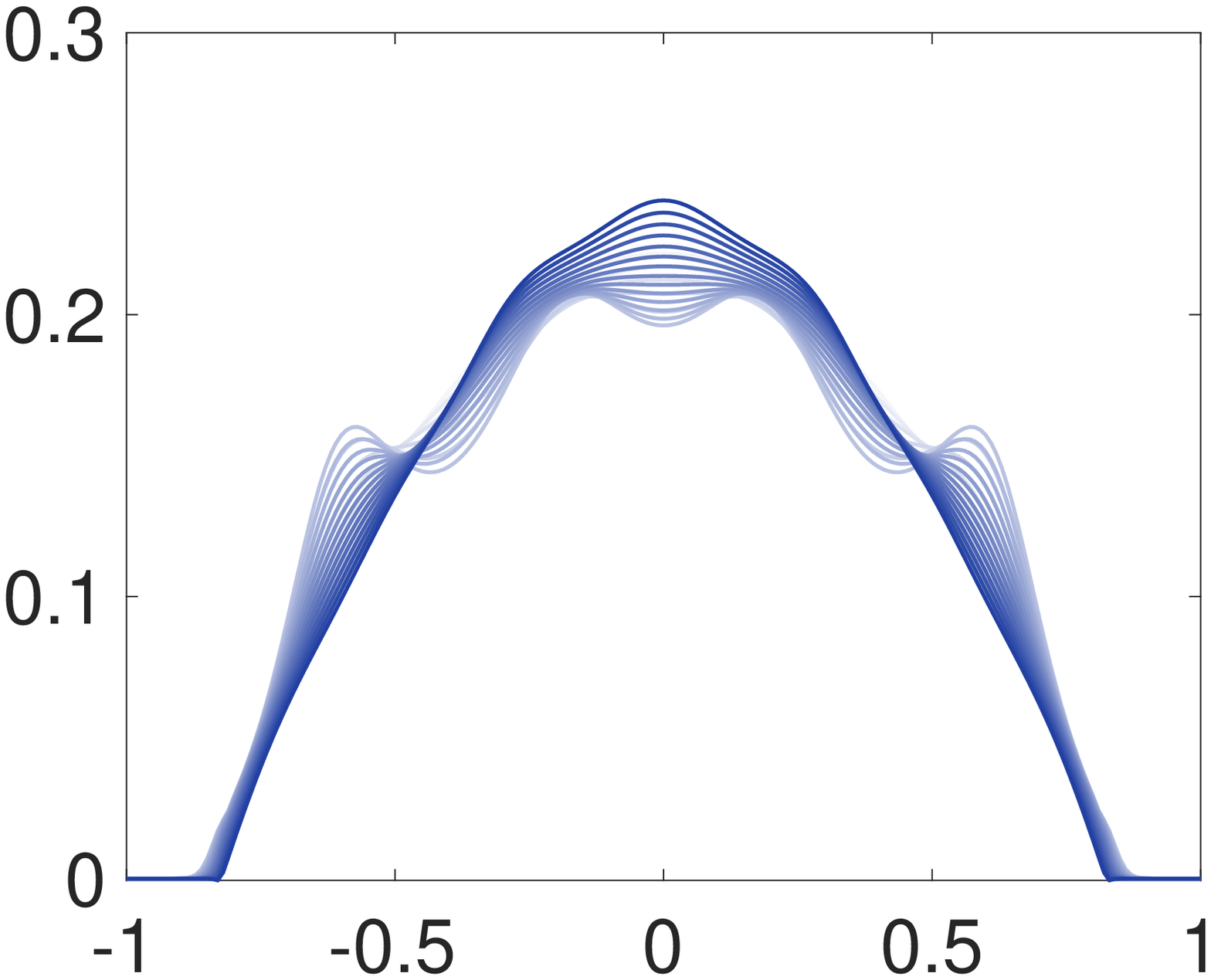}\\
			(c)&(d)\\
			\includegraphics[width=0.4\textwidth]{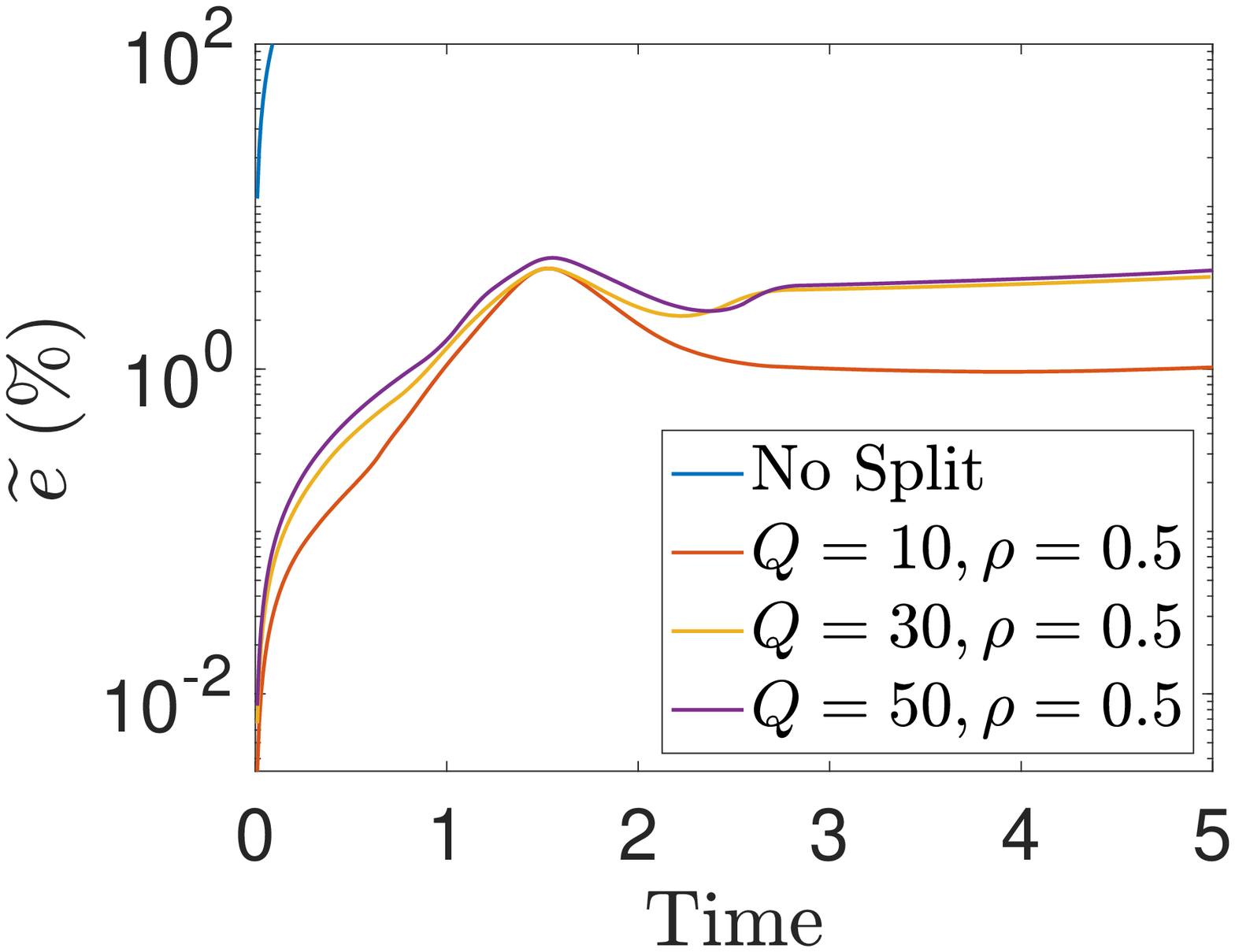}\hspace{0.5cm} &
			\includegraphics[width=0.4\textwidth]{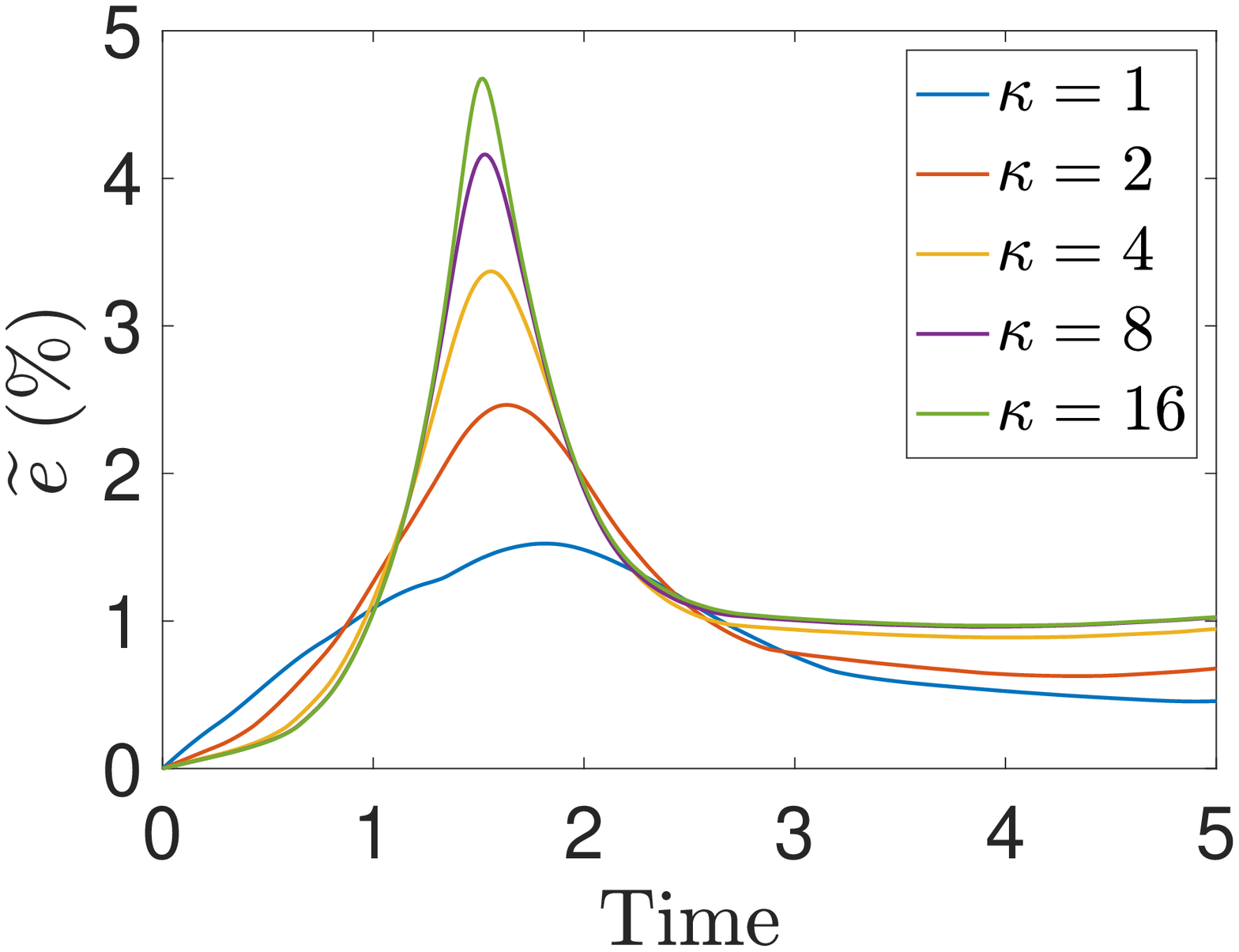}
		\end{tabular}
	\end{center}
	\caption{Identification of the time-varying potential (\ref{eq_TRINO_potential}). (a) The time-varying potential $\phi^*(t,x)$ with $t_B=1.5, \kappa=8$. (b) Data generated by solving (\ref{eq.aggregation}) with the potential in (a) and $\Delta t=0.01, \Delta x=0.01, T=5$.
		(c) The error $\widetilde{e}(t)$ of the identified time-varying potential with $Q=1,10,30,50$. We fix $\kappa=8, \rho=0.5$. (d) The error $\widetilde{e}(t)$ of the identified time-varying potential with various $\kappa$.
		In this experiment, we set $Q=10,\rho=0.5$, and the regularization parameters are chosen as $\alpha=1\times10^{-4}, \beta=1\times10^{-7}$, $\gamma =10$.
	}\label{fig_TRINO1}
\end{figure}

We consider the following time-varying potential
\begin{align}
	\phi^*(t,x) = g(t)\phi^*_1(x)+(1-g(t))\phi^*_2(x)\;,\;\text{with}\;g(t)=0.5+0.5\tanh(\kappa (t-t_B))\label{eq_TRINO_potential}
\end{align}
where $\kappa>0$ is a constant. In (\ref{eq_TRINO_potential}), $\phi^*$ is a weighted average potential from two static potentials $\phi^*_1$ and $\phi^*_2$, where $\kappa$ denotes the rate of transition, and $t_B$ represents the critical transition time. The larger $\kappa$ is, the faster $\phi^*$ transits from $\phi_2^*$ to $\phi_1^*$ around $t_B$. We use $t_B=1.5$,  $\phi_1^*(x)=\phi_{RA}(x;8,3,55)$, and $\phi_2^*(x)=\phi_{RA}(x;2,5,15)$ where $\phi_{RA}$ is defined in (\ref{eq_RA}). The graph of $\phi^*$ with $\kappa=8$ is shown in Figure \ref{fig_TRINO1}(a). Our data are generated by solving \eqref{eq.aggregation} with $\Delta t=\Delta x = 0.01$ and $T=5$, which is shown in Figure~\ref{fig_TRINO1}(b). When identifying time-varying potentials, we set $\rho=0.5$ and $h=0.19$.

In Figure~\ref{fig_TRINO1}(c), we fix $\kappa=8$ and
compare the error $\widetilde{e}(t)$ of the identification results when $Q=1,10,30$ and $50$. Our proposed method with $Q=10$ provides the best result. When $Q$ is large ($Q=30$ and 50), few data are available to identify a potential in each subinterval, leading to large errors. When $Q$ is small ($Q=1$ when there is no splitting), a time-independent potential is identified to approximate a time-varying potential in a large time interval, leading to large errors. There is a tradeoff between the number of data in each subinterval and the subinterval length. A good choice of $Q$ gives rise to the best result.

We then fix $Q=10,\rho=0.5$ and use our proposed method to identify the time-varying potentials with different $\kappa$. The error $\widetilde{e}(t)$ is shown in Figure~\ref{fig_TRINO1}(d). Note that $\phi^*$ changes most rapidly at $t_B=1.5$. As a result, the maximal errors occur around $t=1.5$. As $\kappa$ increases, $\phi^*$ transits faster from $\phi^*_1$ to  $\phi_2^*$. In each subinterval away from $t_B$, the potential $\phi^*$  is dominated by either $\phi^*_1$ or  $\phi_2^*$. Thus we have smaller identification errors in the subintervals away from $t_B$. Meanwhile, a larger $\kappa$ gives a sharper transition of $\phi^*$ around $t_B$. The identification errors in the subintervals around $t_B$ are larger. This is justified by the peaks in Figure~\ref{fig_TRINO1}(d). Moreover, a larger $\kappa$ yields a shorter duration of transitioning, which leads to a narrower peak in the error $\widetilde{e}(t)$.

\section{Potential identification from agent-based data}\label{sec.agentbased}
In previous sections, we consider data (density functions) that are solutions of (\ref{eq.aggregation}). In practice, the density function may not be directly observed. For example, the agent-based data records the agents' locations at different times. Our goal is to identify the potential from the agents' locations over a period of time.

\subsection{Conversion from agent-based data to density function}\label{sec_convert}

We first estimate the density function from the agents' locations and then apply our proposed algorithm to estimate a potential.
Let $\bx_v(t^n)\in\mathbb{R}^d$ be the location of the $v$-th agent at time $t^n$ for $v=1,2,\dots, V$. For any $\bx\in\mathbb{R}^d$, we compute the density function $u(t,\bx)$ at $t^n$ as
\begin{align}
	u(t^n,\bx) =\frac{1}{V} \sum_{v=1}^V K^\bx_H(\bx-\bx_v(t^n))\;,n=0,1,\dots,N-1\;,\label{eq_Kx}
\end{align}
where $K_H^\bx(\bx) = H^{-1/2}K^\bx(H^{-1/2}\bx)$ for some kernel function $K^\bx:\mathds{R}^d\rightarrow \mathds{R}^+$ with a compact support, and $H\in\mathbb{R}^{d\times d}$ is a positive definite matrix. For any $t\in[0,T]$, we compute the data-induced density function
\begin{align}
	u(t,\bx) =\frac{1}{N} \sum_{n=0}^{N-1} K^t_{h}(t-t^n)u(t^n,\bx)\;,\label{eq_data_u}
\end{align}
where $K_{h}^t(t)=Ch^{-1}K^t(h^{-1}t)\mathds{1}_{{\{-r<t\leq 0\}}}$ for some bandwidth $h>0$, some kernel function $K^t:\mathds{R}\rightarrow \mathds{R}^+$, and a thresholding parameter $r>0$. Here $C$ is a constant such that $\int_\mathbb{R}K_{h}^t(t)\,dt =1$. The indicator function $\mathds{1}_{\{0<t\leq r\}}$ ensures that $u(t,\bx)$ is computed from the data in the time interval ${(t^n-r,t^n]}$. It also imposes the assumption that each agent has a short memory and cannot foresee the future. We then sample $u(t,\bx)$ on  a regular grid in space and time to obtain a data set for potential identification.

In this paper, we use the spherical Epanechnikov kernels~\cite{duong2015spherically} for the estimation in~\eqref{eq_Kx} and  for~\eqref{eq_data_u}, i.e.,
\begin{align}
	K^\bx(\bx)=\frac{d+2}{2V_d}(1-|\bx|^2)\mathds{1}_{\{|\bx|^2\leq 1\}}\;,\; K^t(t)=\frac{3}{4C'}(1-t^2)\mathds{1}_{\{|t|\leq 1\}}\label{eq_kernel},
\end{align}
with the thresholding parameter $r=1$. Here
$V_d$ denotes the volume of a unit ball in $\mathbb{R}^d$ and $C'$ is a normalization parameter such that $\int_{\mathbb{R}}K^t(t)\,dt=1$.


We consider a  realistic \textit{noise model} for the agent-based data: The agents' locations are noisy due to the lack of measurement precision such that the measured positions are
\begin{align}
	\widetilde{\bx}_v(t^n) = \bx_v(t^n)+\bm{\varepsilon}_{v}^n\;,
	\label{eq.agent.noise}
\end{align}
where $\bm{\varepsilon}_{v}^n\in\mathds{R}^d$ represents noise. The noise contaminates the estimated density in (\ref{eq_Kx}) and (\ref{eq_data_u}).
This is different from the additive noise model (\ref{eq.data.noise}),
which introduces additive noise to the density function. The noise models in (\ref{eq.agent.noise}) and (\ref{eq.data.noise})  differ in two aspects: First, (\ref{eq.data.noise}) has the same level of noise across time and space, whereas (\ref{eq.agent.noise}) has larger noise in the regions with a higher density. Second,  (\ref{eq.data.noise}) does not consider possible correlations among the density values in a neighborhood, while (\ref{eq.agent.noise}) incorporates these correlations via in the kernel density estimation. In the following numerical experiments, we will illustrate these behaviors and  show that our proposed method works successfully on the noise model  (\ref{eq.agent.noise}). Furthermore, even when the agent data are not simulated from an aggregation equation with certain potential, our proposed method can identify a potential which generates the dynamics as a good approximation of the given data.
\subsection{Numerical experiments}
In the first example, we generate agents' locations as samples from a probability distribution simulated from (\ref{eq.aggregation}) with the potential $\phi_{\text{RA}}(x;5,2,12)$. We first solve (\ref{eq.aggregation}) with the initial condition \eqref{Yao_initial} and $\Delta x= \Delta t = 0.01$ to to obtain the density function $u$. Then the agent-based locations are randomly sampled from the probability distribution whose density function is proportional to $u$. Figure~\ref{fig_DRINO}~(a) shows $100$ samples at each time level perturbed by a Gaussian noise with variance $\sigma=0.01$. From this data set, we compute a density function by the kernel method  and then identify the potential by Algorithm \ref{alg2}, as described in Section~\ref{sec_convert}. For the kernel density estimation, we take the window parameter $H=4\Delta x$ for space and $h=2\Delta t$ for time.  The identified potentials from data with various number of samples are shown in Figure~\ref{fig_DRINO} (b), where darker curves are identified potentials from more samples. The identified potential converges to the underlying potential as we increase the number of samples. This is because more samples give a more accurate approximation of the density function from the agents' locations.  Figure~\ref{fig_DRINO}(c) shows the averaged error  $\widetilde{e}(t)$ over time as a function of the sample size   when $\sigma = 0.01, 0.1,0.5$ respectively. As the noise standard deviation  $\sigma$ increases, the averaged error $\widetilde{e}(t)$ increases. 
\begin{figure}[t!]
	\begin{center}
		\begin{tabular}{c@{\hspace{2pt}}c@{\hspace{2pt}}c}
			(a)&(b)&(c)\\
			\includegraphics[width=0.33\textwidth]{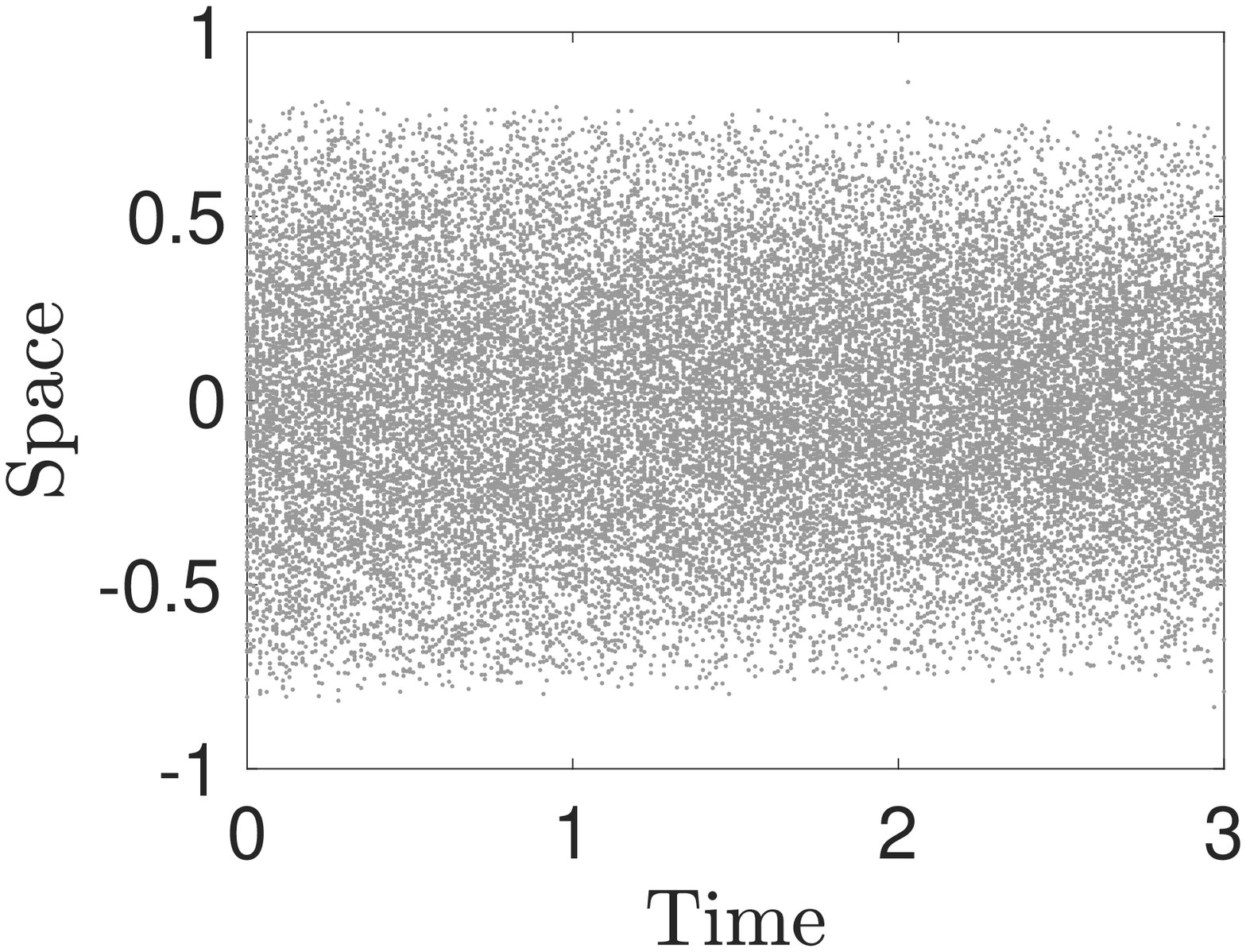}&
			\includegraphics[width=0.33\textwidth]{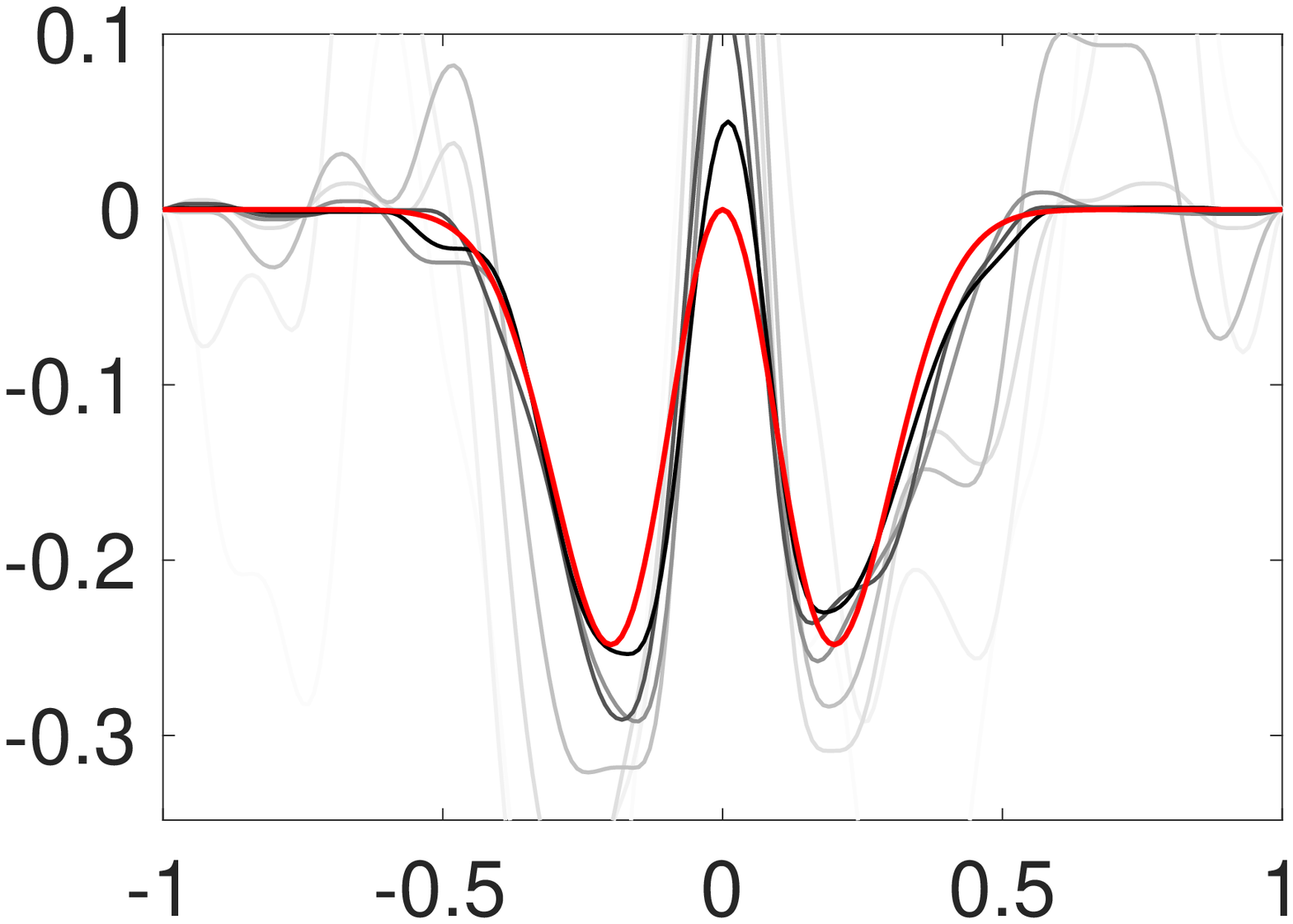}&
			\includegraphics[width=0.33\textwidth]{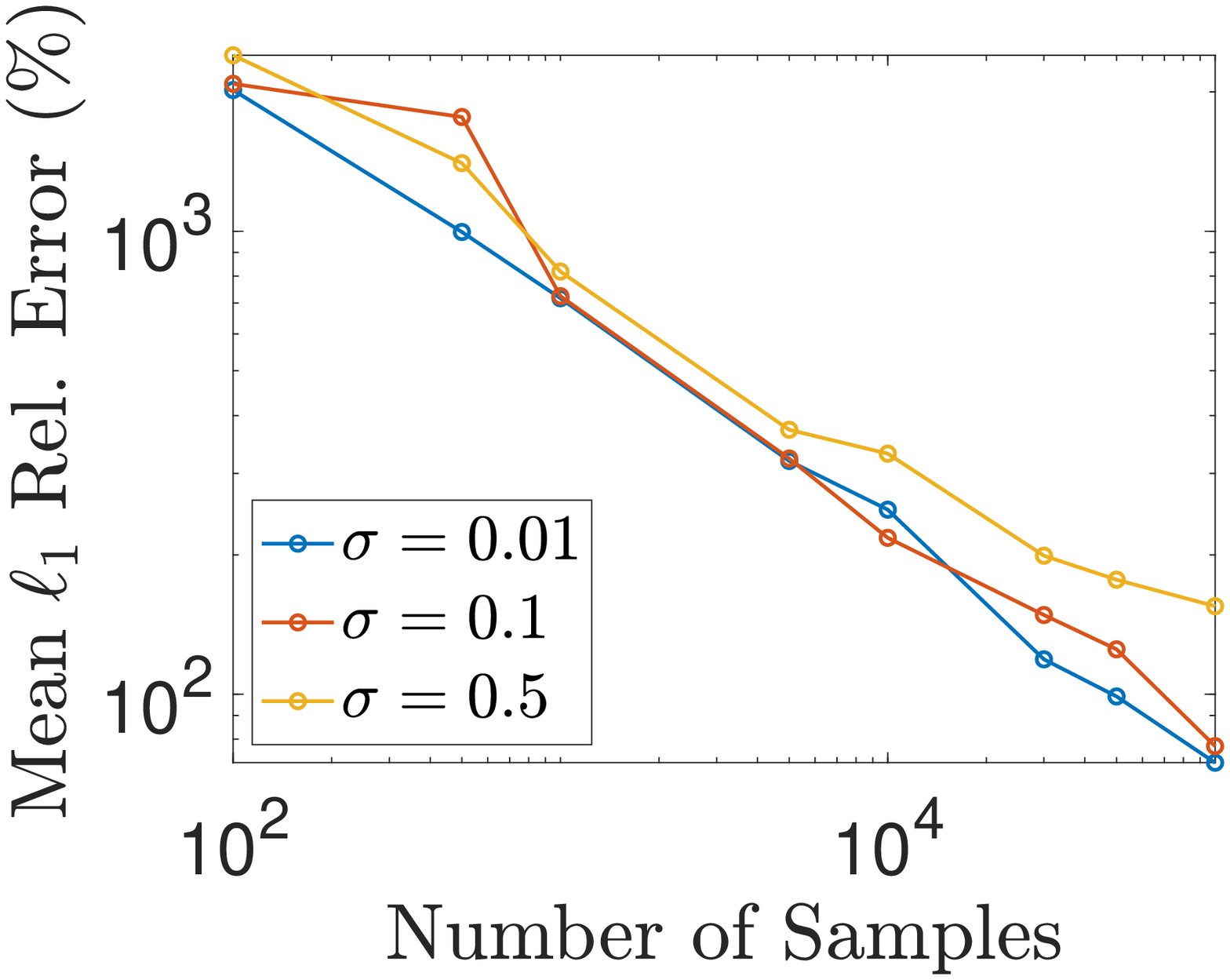}
		\end{tabular}
	\end{center}
	\caption{Agent based data with an underlying potential.
		(a) Agent-based data sampled from the solution of \eqref{eq.aggregation} with the potential~$\phi_{\text{RA}}(x;5,2,12)$ using $\Delta x= \Delta t=0.01$.  Each time level contains 100 agents. The positions are purterbed by Gaussian noise with standard deviation $\sigma = 0.01$.
		(b) The underlying potential (red) and identified potentials (black) from data with the number of agents ranging over $10^2,5\times 10^2,10^3,5\times10^3,10^4, 3\times 10^4,5\times 10^4$ and $10^6$. Darker curves are results using larger numbers of agents. (c) With noise variance $\sigma=0.01, 0.1, 0.5$, the averaged error $\widetilde{e}(t)$ of the identified potentials versus the number of agents. In both (b) and (c), we fix $\alpha=1\times10^{-3},\beta=1\times10^{-7},\gamma=10$.}\label{fig_DRINO}
\end{figure}

We next consider agent-based data, which are not generated by solving the aggregation equation. Specifically, we generate the data from the Reynold's boids model \cite{reynolds1987flocks} which follows a set of interaction rules for the agents. 
Our proposed method is then used to identify a potential that approximates the dynamics of the Reynold's boids model. In the data generating process, 500 agents are used to simulate an repulsive dynamic on the domain $[-1,1]^2$ for 200 steps with $\Delta t=0.01$. At every time step, the locations of all agents are recorded. Figure \ref{fig.RINO.2d.agent.repul}(a)-(d) show the distribution of these agents at $t=0, 1, 1.5$ and $2$. The density function is then computed by the kernel methods discussed in Section~\ref{sec_convert}. We set $H$ as a diagonal matrix with diagonal entries $0.15$ and there is no smoothing in the temporal direction, i.e., only (\ref{eq_Kx}) is used. The density is sampled on the grid with $\Delta x=0.1,\Delta t=0.01$. The cross-sections of the density function along $x_1=0$ are shown in Figure \ref{fig.RINO.2d.agent.repul}(e). Since the kernel method has a smoothing effect, we do not apply SDD in this example. The parameters are set as $\alpha=1\times10^{-7}, \beta=1\times10^{-9}$  and $\gamma=0$. Our proposed method identifies the potential shown in Figure \ref{fig.RINO.2d.agent.repul}(f), which corresponds to the repulsive dynamic.
We next verify if the data are well approximated by the dynamics of the aggregation equation with the identified potential.
After solving (\ref{eq.aggregation}) with the identified potential, we show the cross-sections of the solution along $x_1=0$ in Figure \ref{fig.RINO.2d.agent.repul}(g). We observe that, the simulated solution approximates the given data well and recovers the repulsive behavior. The error $\widetilde{e}(t)$ is shown in Figure~\ref{fig.RINO.2d.agent.repul}(h). 

In the next experiment, we identify a time-varying potential from agent-based data containing two different dynamics: the agents first expand then concentrate. The data are generated in the same manner as the previous experiment with $\Delta t=0.01$ and $T=4$. Our spatial computational domain is $[-1,1]^2$ with a $21\times 21$ grid. In our experiment, $\alpha=1\times10^{-7}, \beta=1\times10^{-9},\gamma=0$ and $\rho=0$ are used. We apply our algorithm with $Q=2,4,8$. After the potentials on each subinterval are identified, we construct the time-varying potential by linearly interpolating them. Our results are shown in Figure \ref{fig.RINO.2d.agent2}. Figure \ref{fig.RINO.2d.agent2}~(a)-(d) show the distribution of the agent-based data at $t=0,1.5,2.5$, and $4$.
In Figure \ref{fig.RINO.2d.agent2}~(e)-(f), we show the cross-section of the identified potential with $Q=2$ and $Q=4$ along $x_2=0$. The two identified potentials look similar.  As time marches, the potential transits from a repulsive one to an attractive one, corresponding to the given data's two dynamic phases. The error $\widetilde{e}(t)$ with different $Q$'s is shown in Figure \ref{fig.RINO.2d.agent2}(f). The error $\widetilde{e}(t)$ for all $Q$'s are close to each other. For this data set, the result is not sensitive to the value of $Q$. Similar to the observations in Section \ref{sec_time}, the error $\tilde{e}(t)$ achieves its maximum around the time of transition between the two dynamics.


\begin{figure}[t!]
	\begin{center}
		\begin{tabular}{c@{\hspace{2pt}}c@{\hspace{2pt}}c@{\hspace{2pt}}c}
			(a)&(b)&(c)&(d)\\
			\includegraphics[trim={2cm 0 3.2cm 0},clip,width=0.24\textwidth]{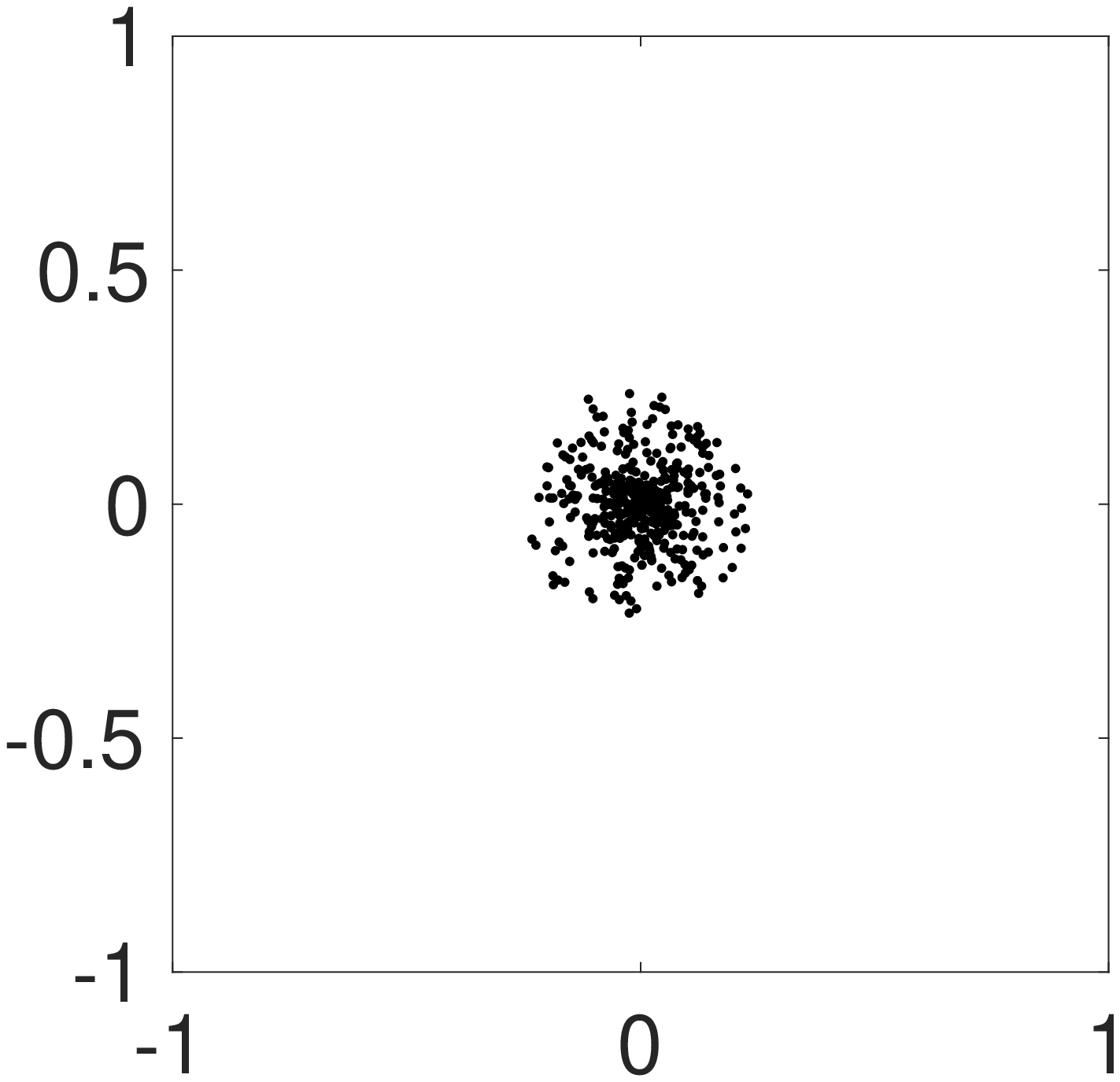}&
			\includegraphics[trim={2cm 0 3.2cm 0},clip,width=0.24\textwidth]{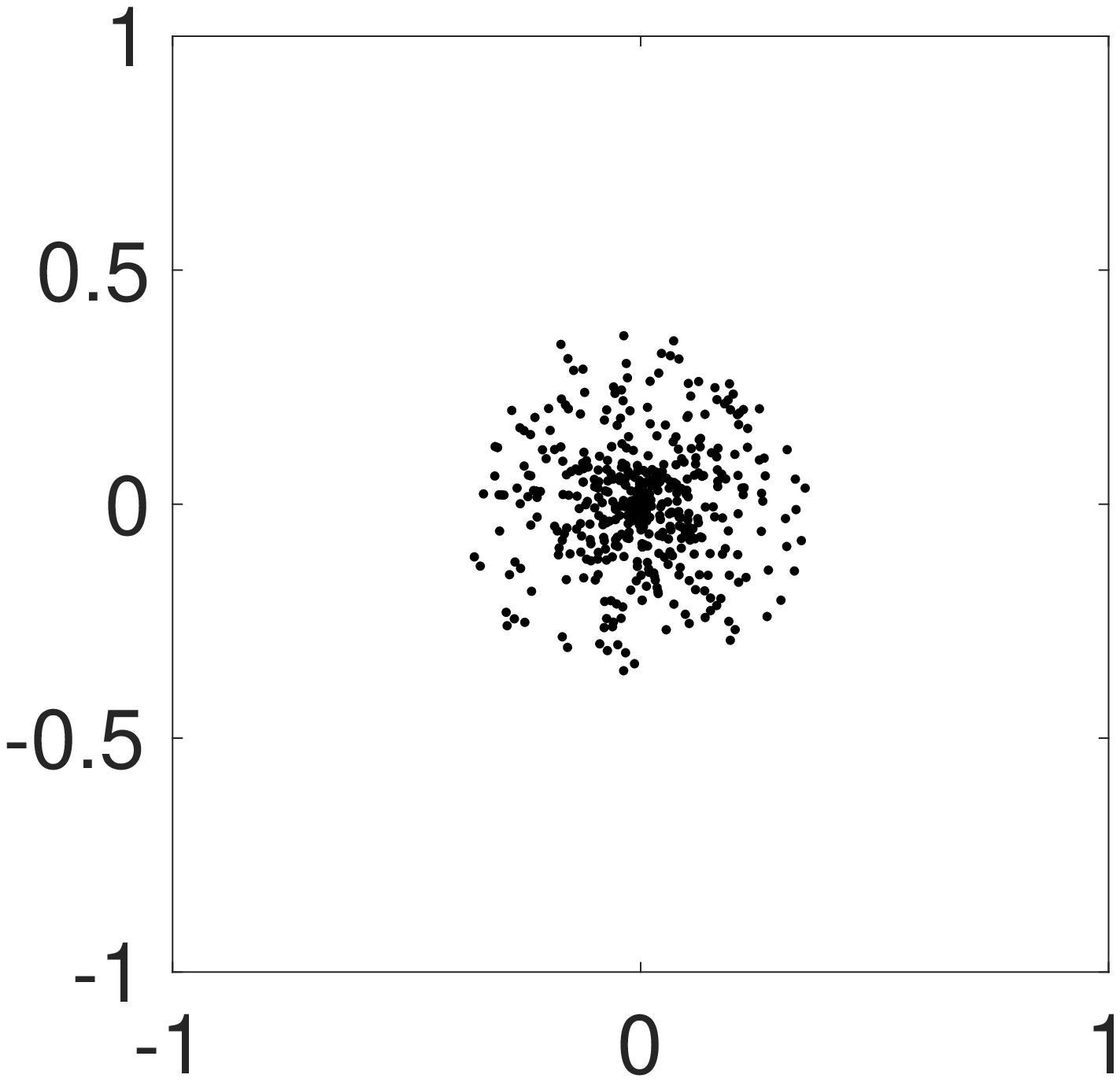}&
			\includegraphics[trim={2cm 0 3.2cm 0},clip,width=0.24\textwidth]{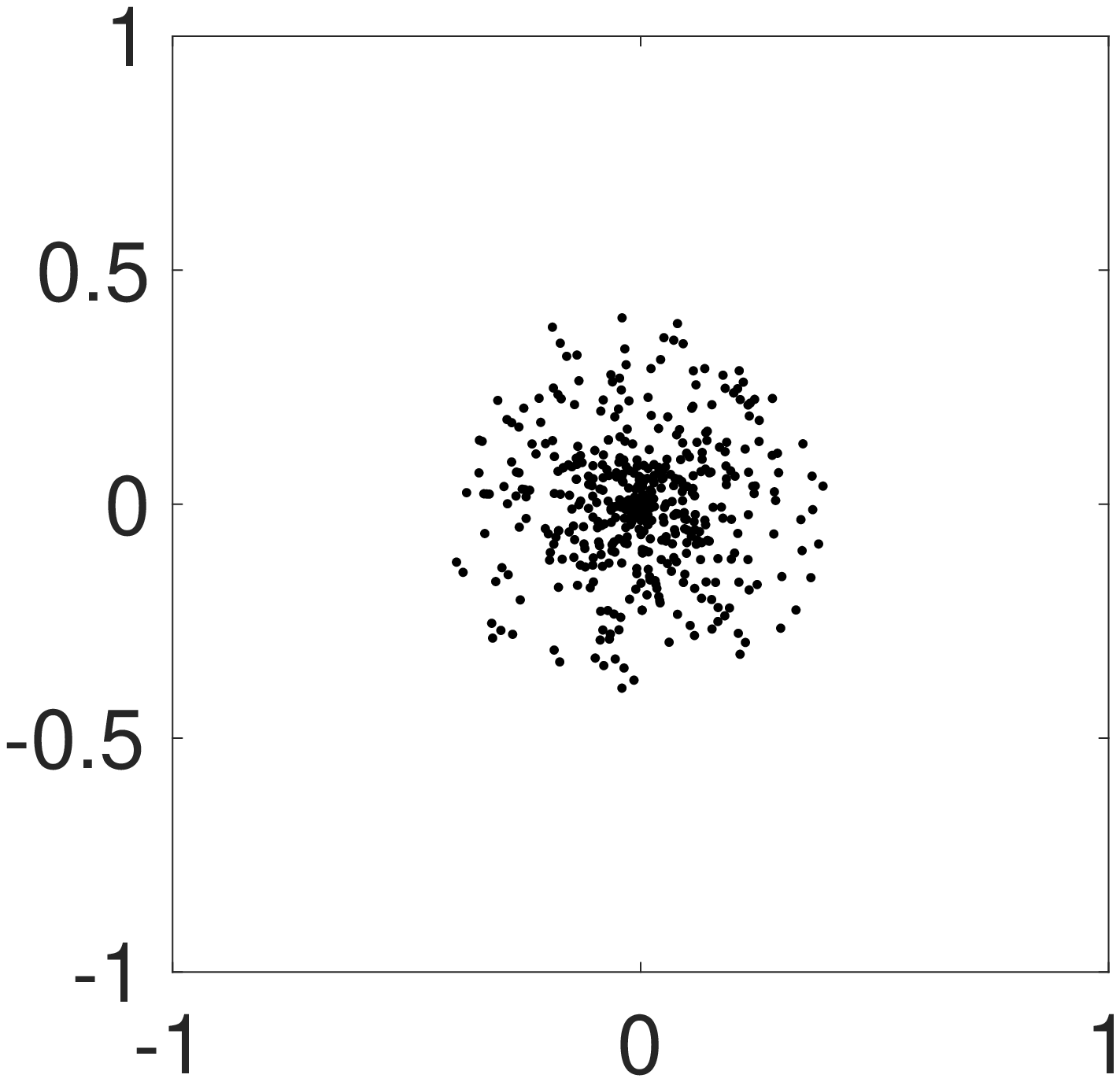}&
			\includegraphics[trim={2cm 0 3.2cm 0},clip,width=0.24\textwidth]{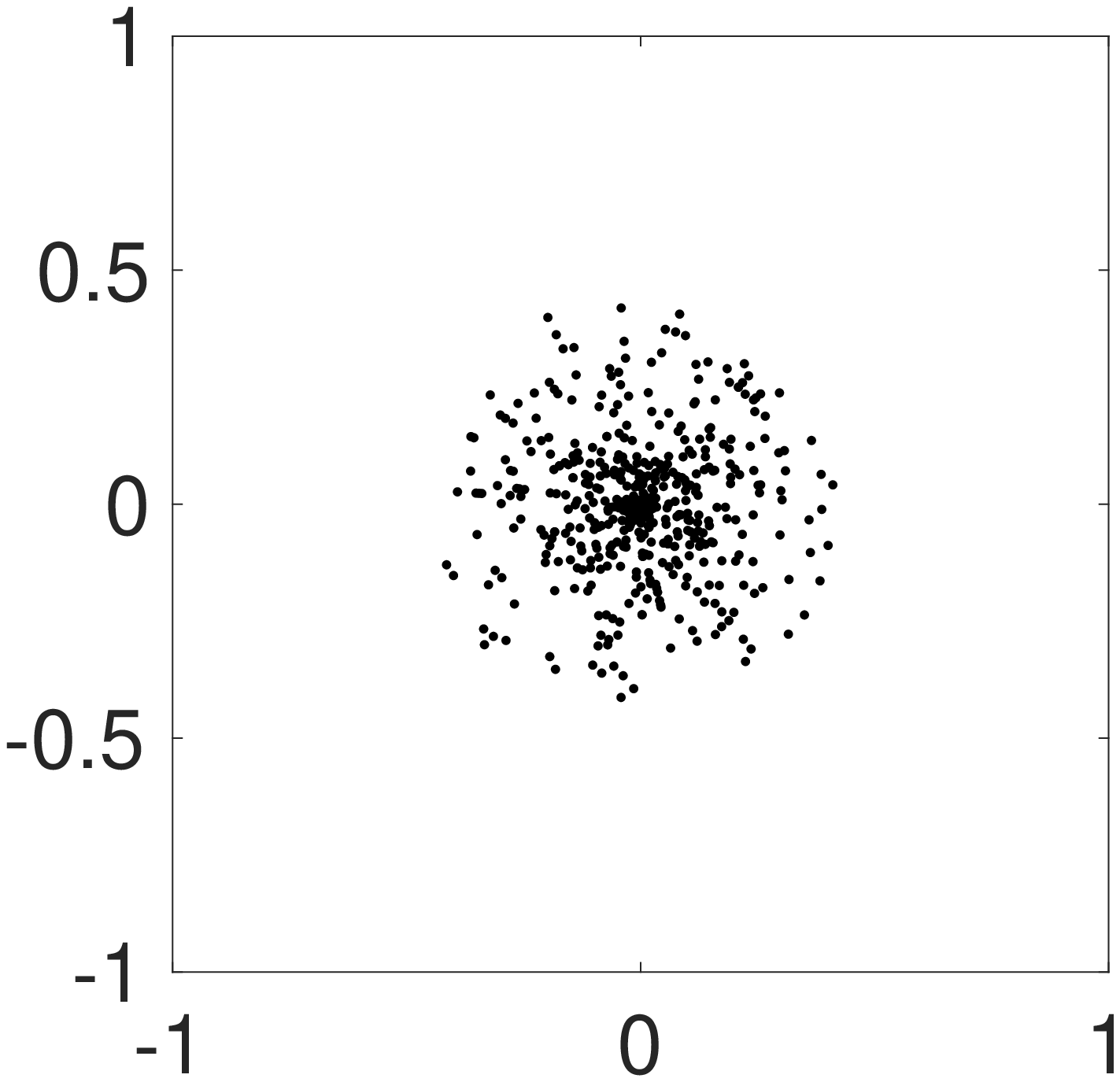}\\
			\multicolumn{2}{c}{(e)}&	\multicolumn{2}{c}{(f)}\\
			\multicolumn{2}{c}{	\includegraphics[width=0.37\textwidth]{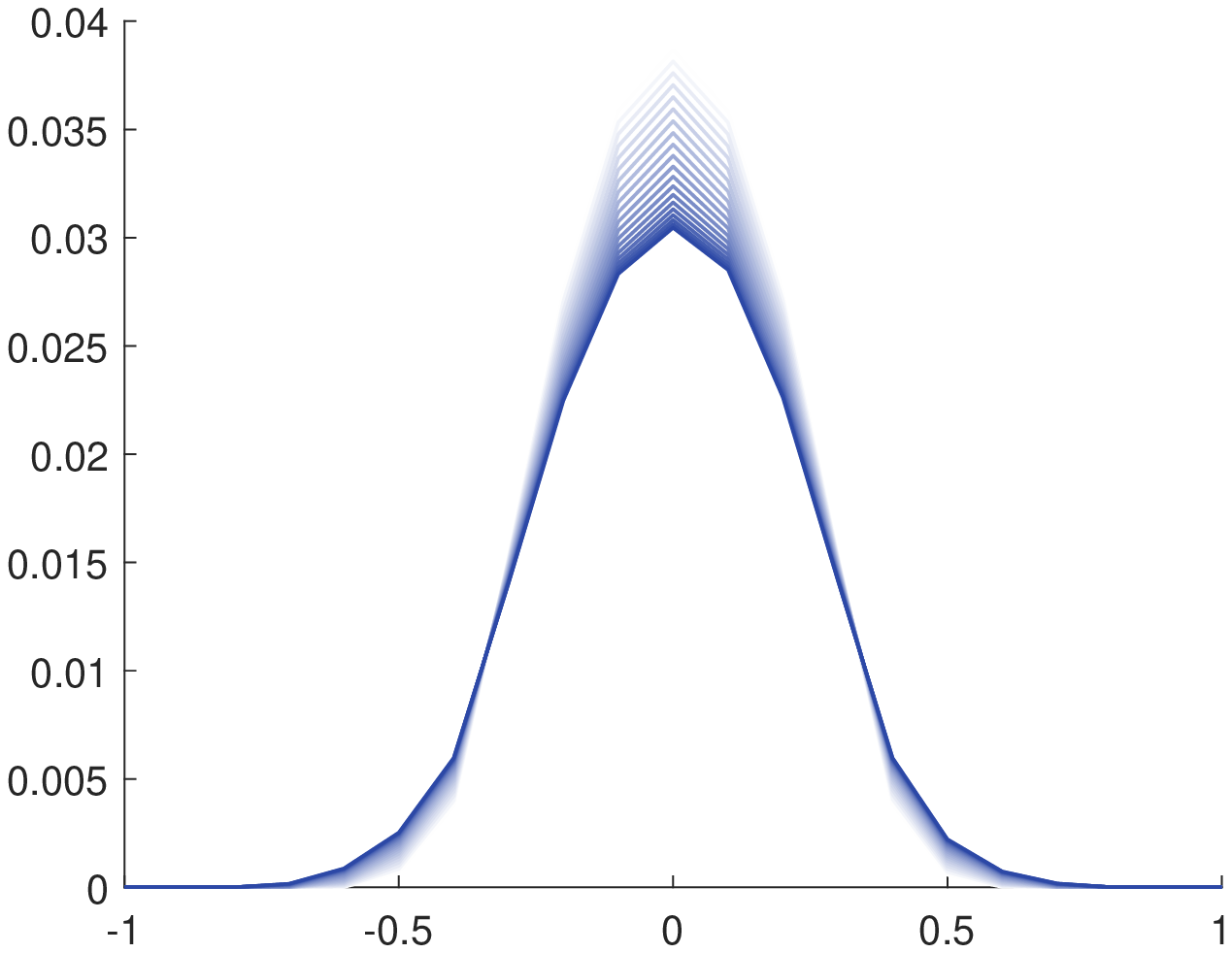}}&	\multicolumn{2}{c}{\includegraphics[width=0.37\textwidth]{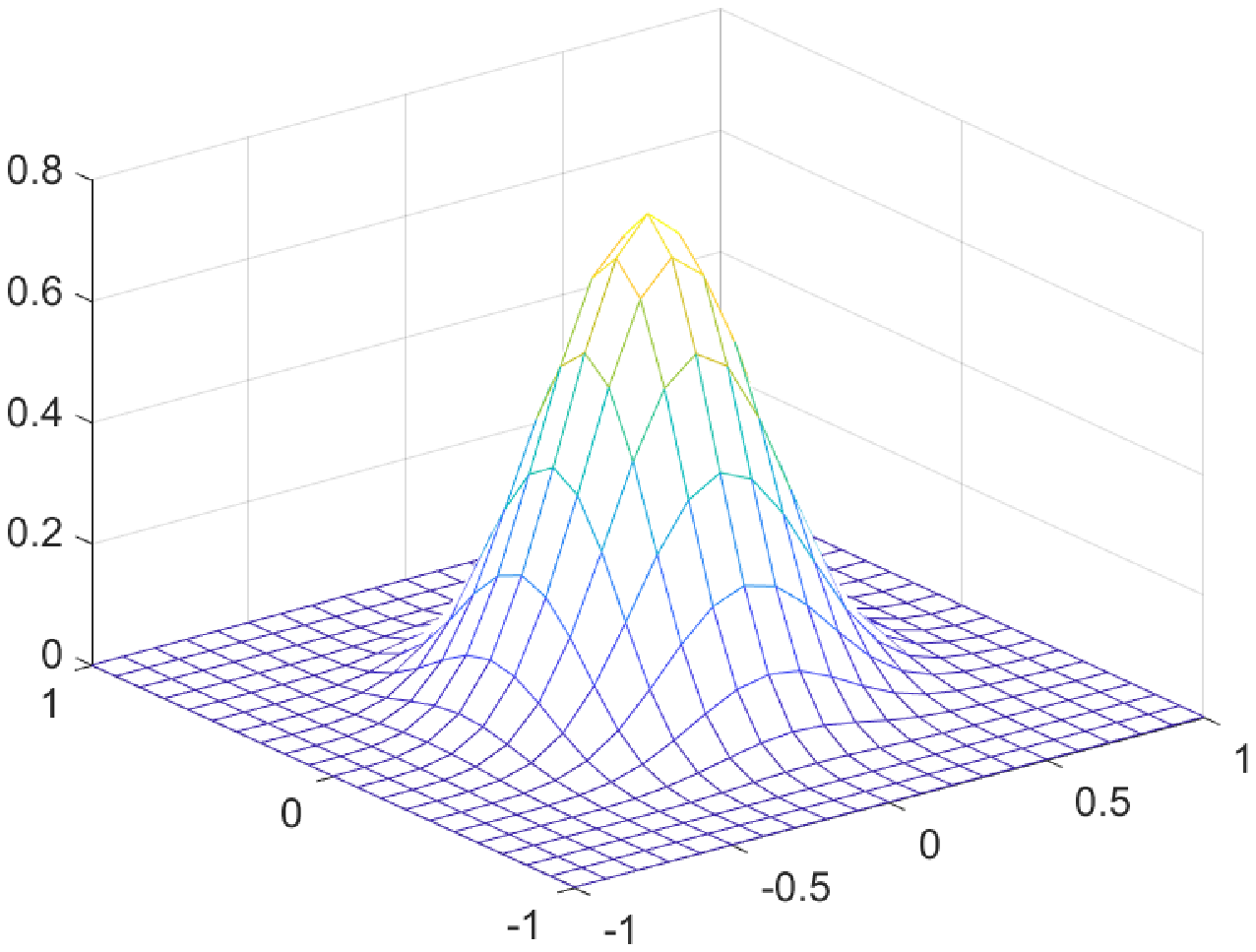}}\\
			\multicolumn{2}{c}{(g)}&	\multicolumn{2}{c}{(h)}\\
			\multicolumn{2}{c}{	\includegraphics[width=0.37\textwidth]{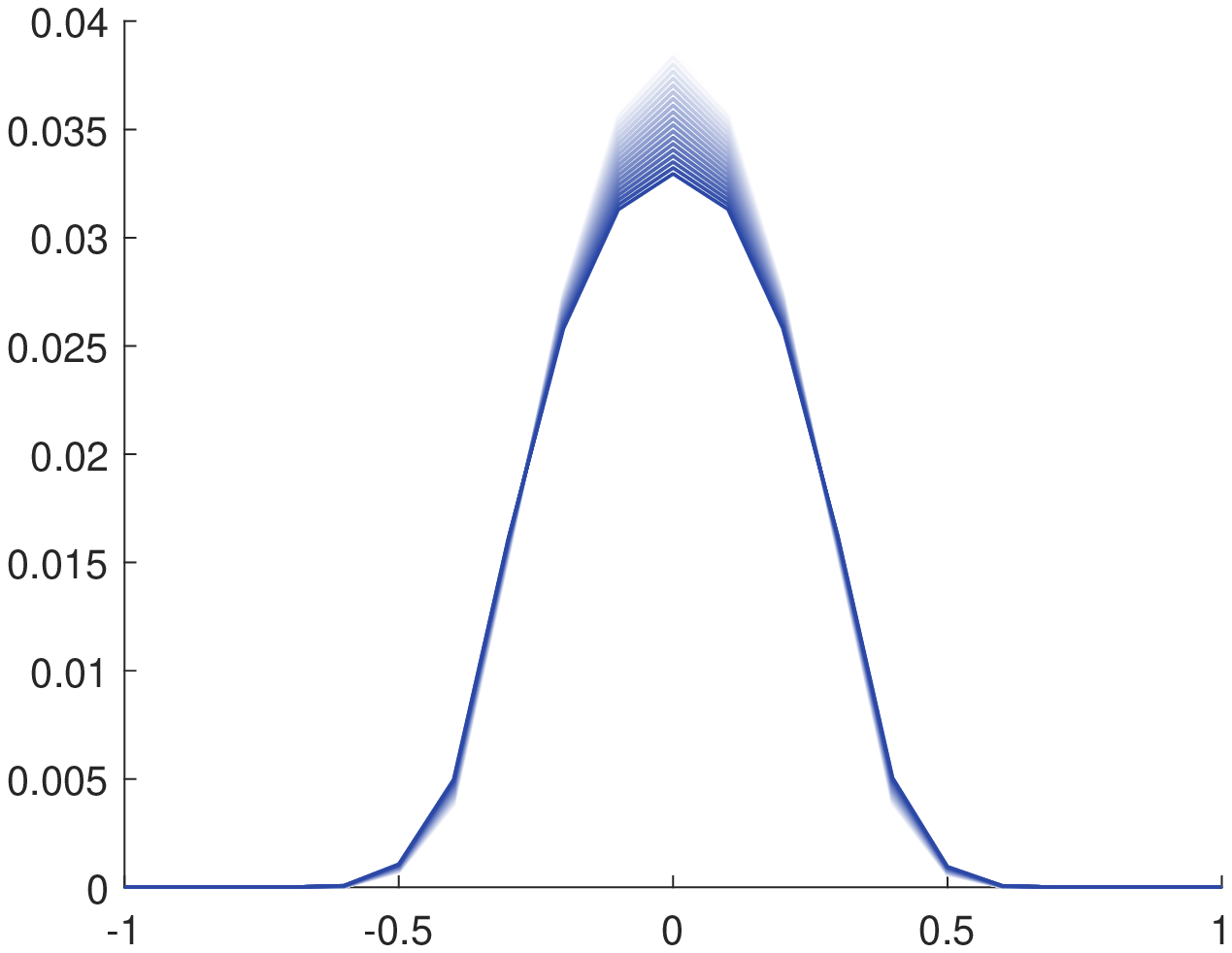}}&	\multicolumn{2}{c}{\includegraphics[width=0.37\textwidth]{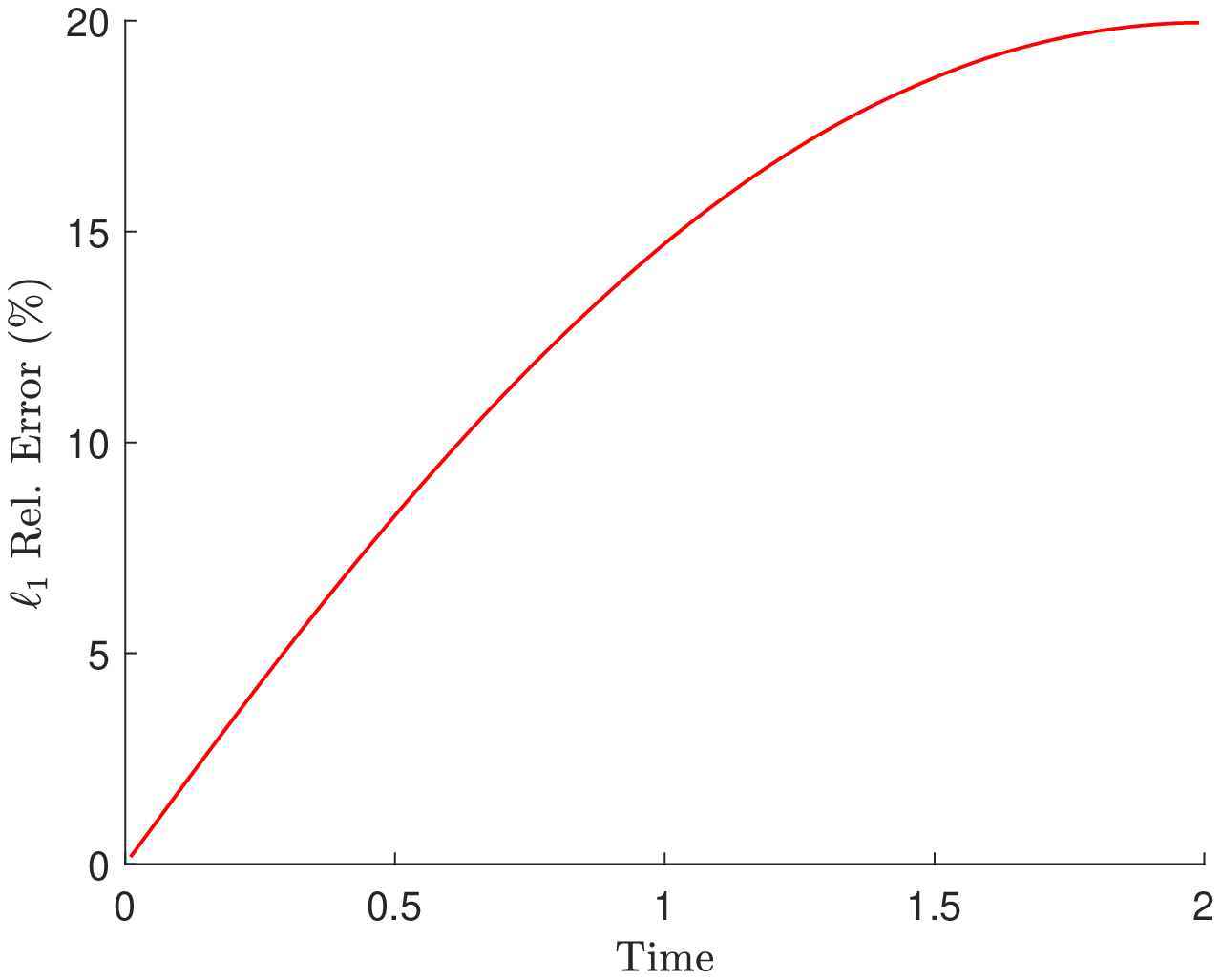}}
		\end{tabular}
	\end{center}
	\caption{Reynold's boids model (repulsion). (a)-(d): Simulated agents' locations (black dots) at (a) $t=0$, (b) $t=1$, (c) $t=1.5$, and (d) $t=2$. (e) Cross sections of the estimated density function along $x_1=0$. (f) The identified potential $\widehat{\phi}$. (g) Cross sections of the solution to (\ref{eq.aggregation}) with the identified potential $\widehat{\phi}$ along $x_1=0$. (h) The error $\widetilde{e}(t)$ as a function of $t$. The regularization parameters are set as $\alpha=1\times10^{-7}, \beta=1\times10^{-9},\gamma=0$. }\label{fig.RINO.2d.agent.repul}
\end{figure}

\begin{figure}[t]
	\begin{center}
		\begin{tabular}{c@{\hspace{2pt}}c@{\hspace{2pt}}c@{\hspace{2pt}}c}
			(a)&(b)&(c)&(d)\\
			\includegraphics[trim={2cm 0 3.2cm 0},clip,width=0.24\textwidth]{Figures/Reynold501.eps}&
			\includegraphics[trim={2cm 0 3.2cm 0},clip,width=0.24\textwidth]{Figures/Reynold600.eps}&
			\includegraphics[trim={2cm 0 3.2cm 0},clip,width=0.24\textwidth]{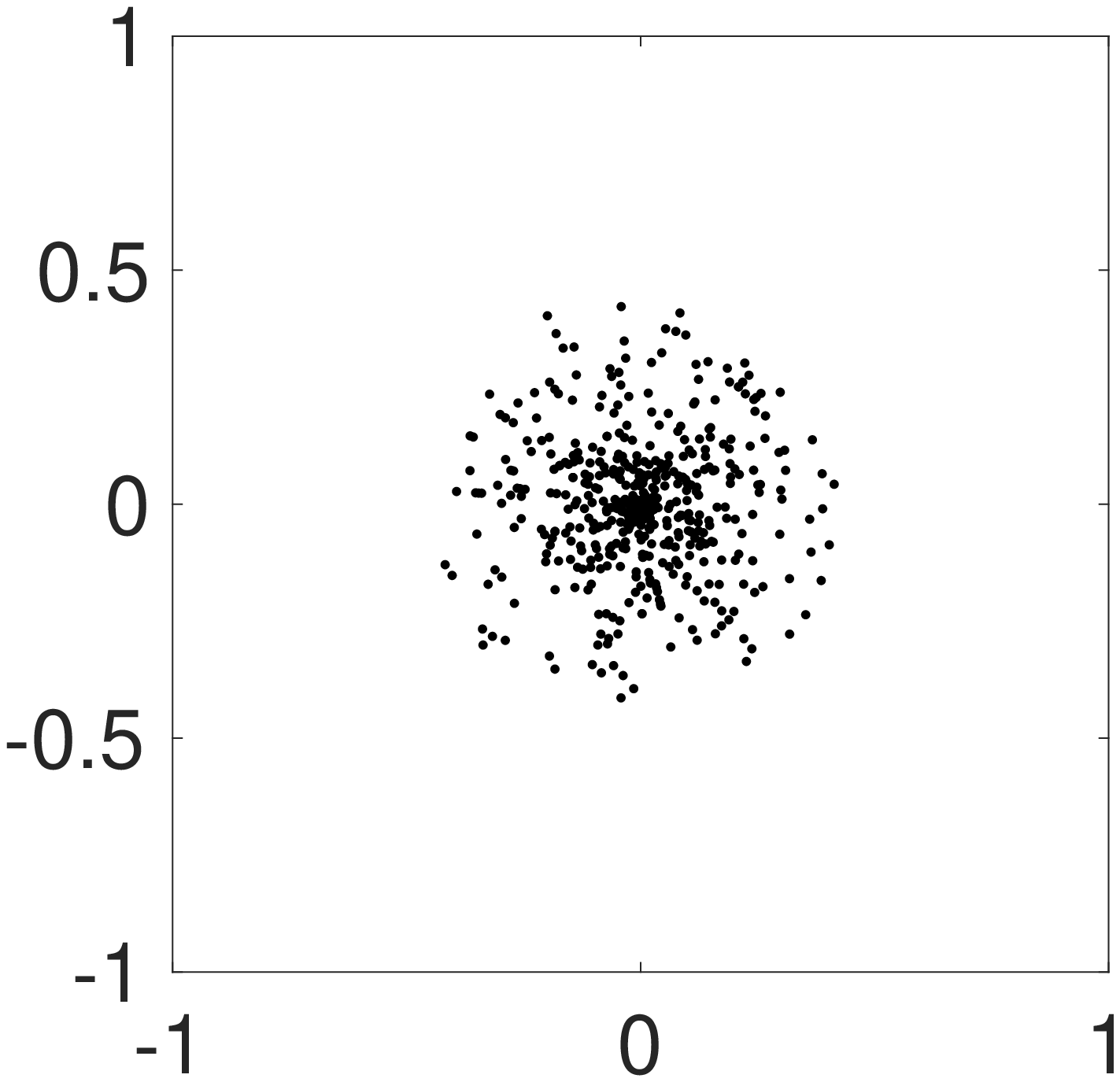}&
			\includegraphics[trim={2cm 0 3.2cm 0},clip,width=0.24\textwidth]{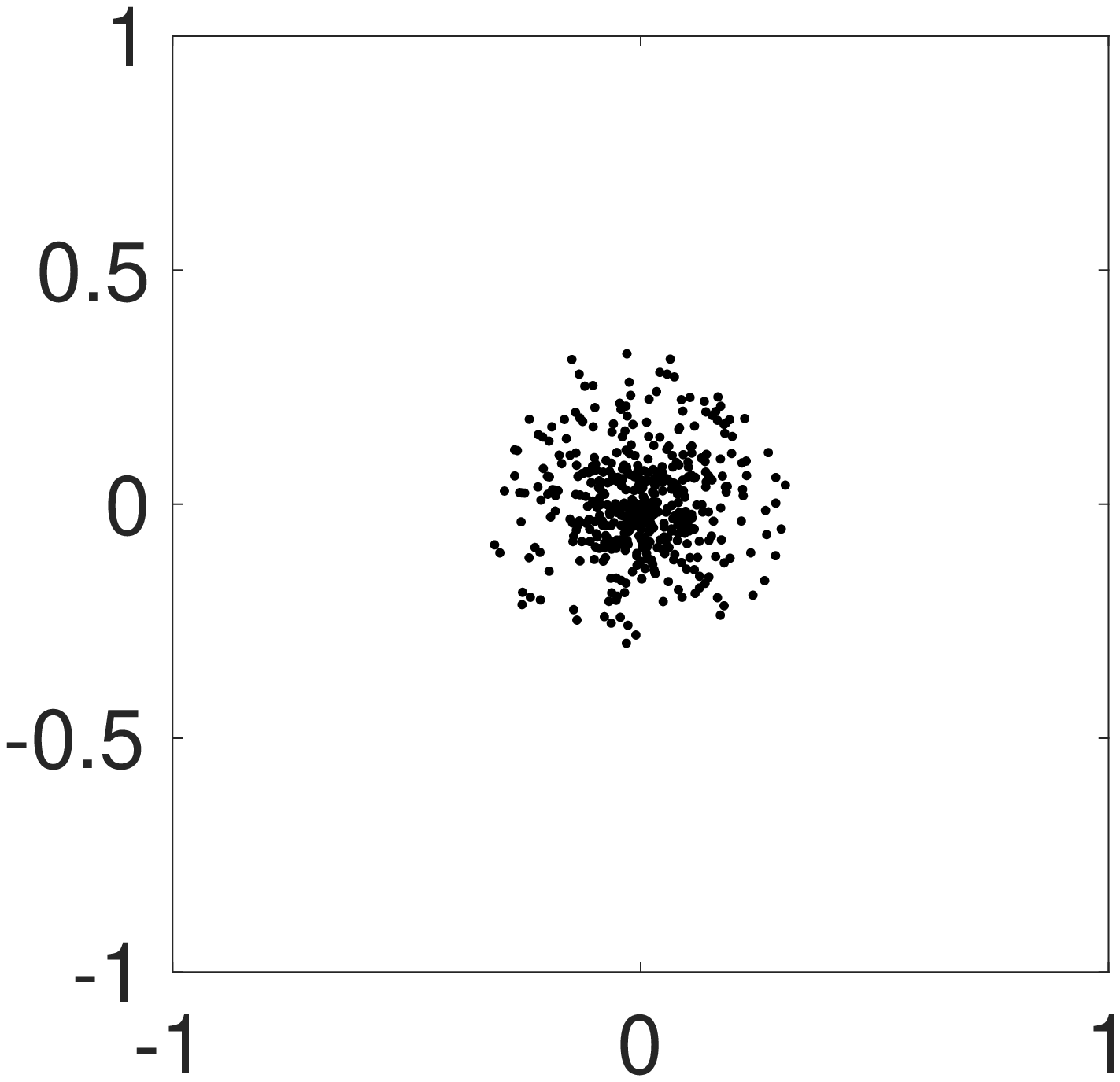}
		\end{tabular}
		\begin{tabular}{ccc}
			(e)&(f)&(g)\\
			\includegraphics[width=0.32\textwidth]{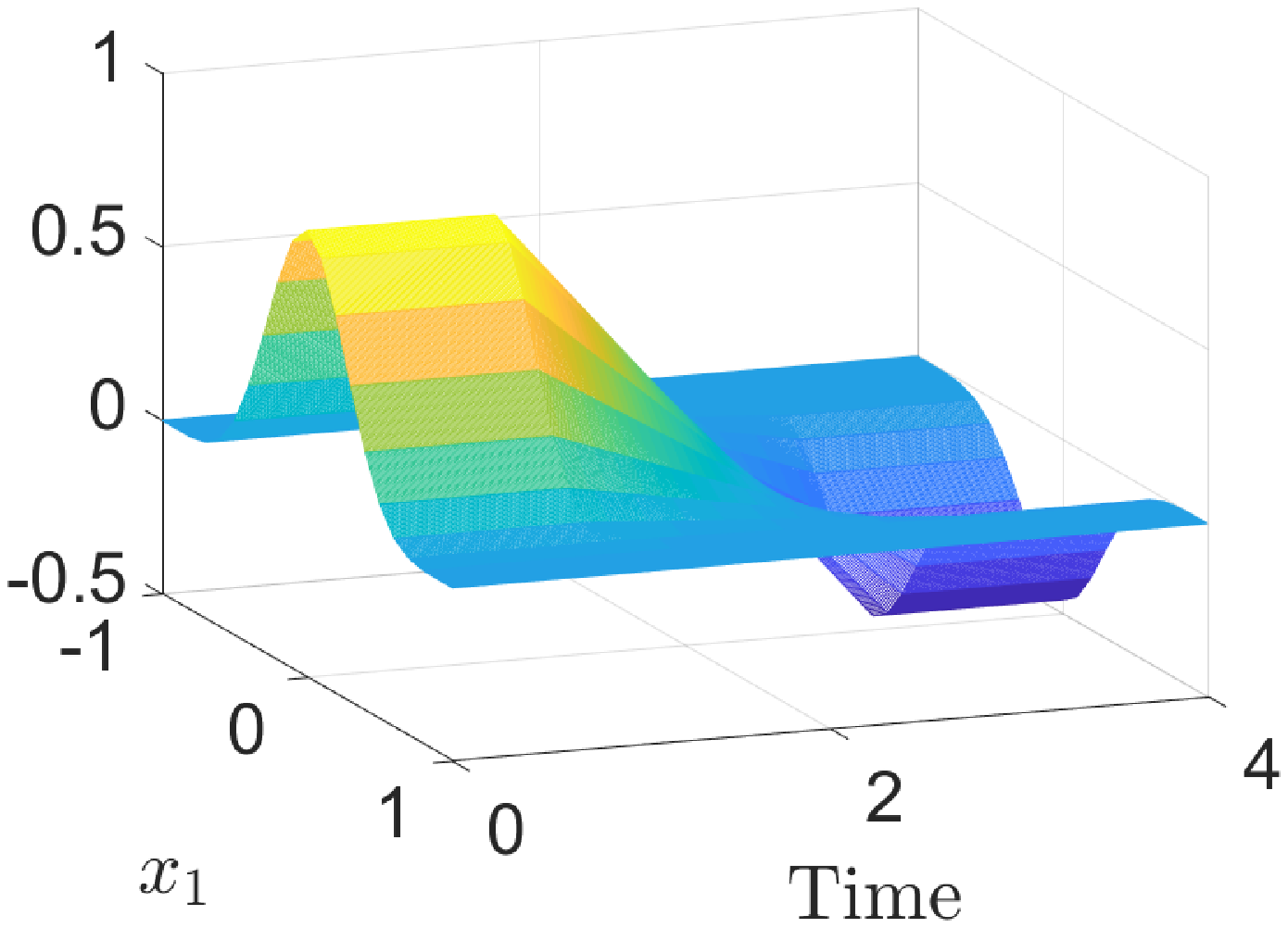}&	
			\includegraphics[width=0.32\textwidth]{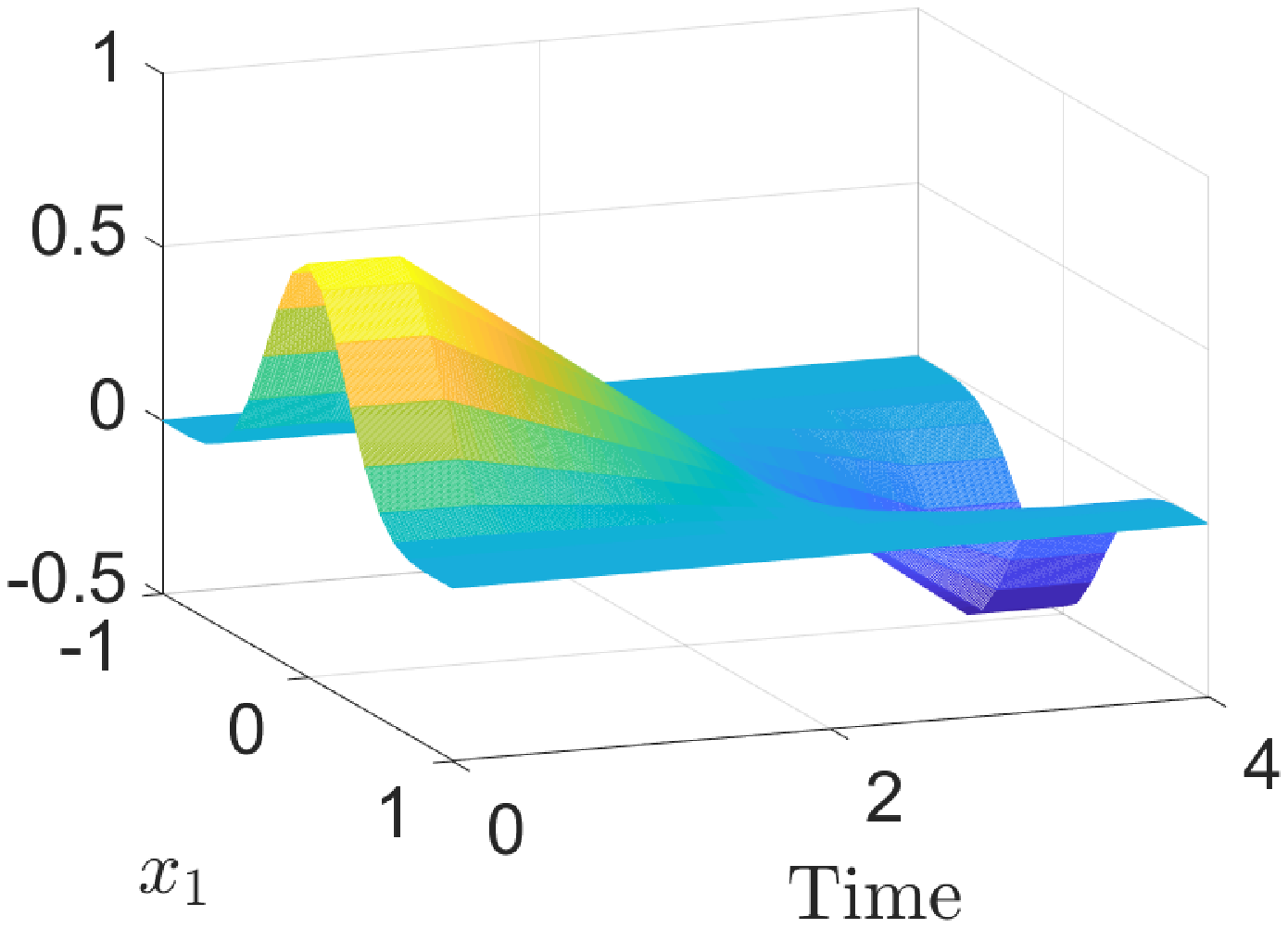}&	
			\includegraphics[width=0.32\textwidth]{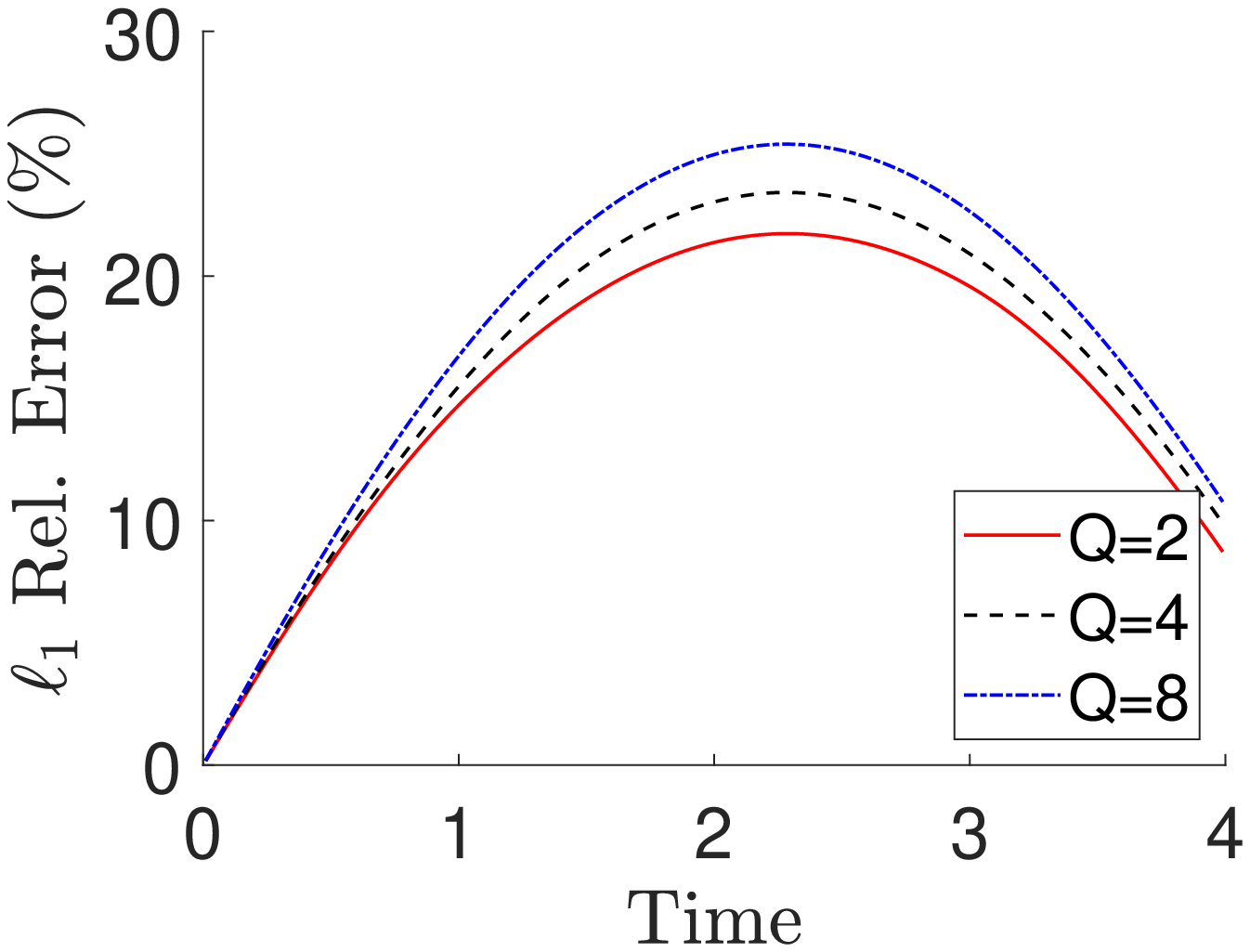}
		\end{tabular}
	\end{center}
	\caption{Reynold's boids model (two dynamics). (a)-(d): Simulated agents' locations (black dots) at (a) $t=0$, (b) $t=1$, (c) $t=2.5$, and (d) $t=4$. (e)-(f): Cross sections of the identified potentials $\widehat{\phi}$  along $x_2=0$ with  $Q=2$ (e) and  $Q=4$ (f) respectively. (g) The error $\widetilde{e}(t)$as functions of $t$  with $Q=2,4,8$ . The parameters are set as $\alpha=1\times10^{-7}, \beta=1\times10^{-9},\gamma=0, \rho=0$. }\label{fig.RINO.2d.agent2}
\end{figure}

\section{Conclusion and future work}\label{sec.conclusion}
This paper proposes a numerical method to identify potentials in aggregation equations from a noisy data set.
We propose to minimize a functional regularized by the total variation and the squared Laplacian of the potential. A splitting Bregman method is then used to efficiently find the proposed functional minimizer.   To improve the robustness of the proposed method, we designed an adaptive support scheme and a technique of imposing the symmetry constraint for symmetric potentials. We also propose a splitting-and-merge strategy to identify time-varying potentials, and a method to identify potentials from agent-based data. Systematic 
experiments demonstrate that our method can identify a good approximation of the underlying potential from a noisy data set. Even when the agent data are not simulated from an aggregation PDE model, our proposed method can identify a potential that generates the dynamics as a good approximation of the given data.

This paper focuses on the recovery of spatially dependent potentials.  
When the potential is time-dependent only, one can replace regularity penalties in space with those along the time direction. Identifying time and spatially-dependent potentials is more challenging, especially with noisy data. 
The method proposed in Section \ref{sec_time} is a simple extension of our method towards resolving time and spatially dependent potentials. We leave the design of a more robust method as our future work. 

The problem studied in this paper assumes that the discretized PDE value is known on every grid point (spatial and temporal domain). 
Suppose the data is only given in a few non-consecutive time frames or even only at the final time frame. In that case, the problem becomes more challenging since one cannot compute the temporal partial derivative of the solution easily.
One possible direction is to formulate it as a PDE constraint optimization problem and solve it by the adjoint state method. We leave it as our future work.

\section*{Acknowledgment}
The authors would like to sincerely thank Prof. Yao Yao in the School of Mathematics at Georgia Institute of Technology for invaluable discussions on aggregation models. 
Sung-Ha Kang is supported in part by Simons Foundation grant 282311 and 584960. Wenjing Liao is supported in part by NSF grant NSF-DMS 1818751 and NSF-DMS 2012652. Hao Liu is supported in part by HKBU 162784 and 179356. Yingjie Liu is supported in part by NSF grants DMS-1522585 and DMS-CDS\&E-MSS-1622453.

\appendix
\section*{Appendix}
\section{The space $P$ is complete and reflexive}
\begin{Prop}\label{prop_Banach}
	The space $P$ defined in (\ref{eq_space})
	equipped with the norm $\|\phi\|_P = \|\phi\|_{H_0^1(\Omega)}+\lim_{\varepsilon\to 0^+}\|\nabla\cdot\nabla\phi\|_{L^2(\Omega\setminus\bar{B}_\varepsilon)}$ is a reflexive Banach space.
\end{Prop}
\begin{proof}Take a Cauchy sequence $\{\phi_i\}\subseteq P$ and a decreasing sequence $\{\varepsilon_N\}_{N=1}^\infty$ converging to $0$. Since $H_0^1(\Omega)$ is complete, there exists $\phi$ such that $\phi_i\to\phi$ in $H^1_0(\Omega)$. Fix an arbitrary integer $N>0$, then for any $n\in\mathbb{N}$ with $n>N$, there exist integers $i_n,j_n$  and  a real number $\varepsilon_n<\varepsilon_N $ such that
	\begin{align*}
		\|\nabla\cdot\nabla\phi_{i_n}-\nabla\cdot\nabla\phi_{j_n}\|_{L^2(\Omega\setminus\bar{B}_{\varepsilon_N})}<\|\nabla\cdot\nabla\phi_{i_n}-\nabla\cdot\nabla\phi_{j_n}\|_{L^2(\Omega\setminus\bar{B}_{\varepsilon_n})}<2^{-n}\;.
	\end{align*}
	Since $L^2(\Omega\setminus\bar{B}_{\varepsilon_N})$ is complete, $\nabla\cdot\nabla\phi_i$ converges to some $\psi_N$ in $L^2(\Omega\setminus\bar{B}_{\varepsilon_N})$. Define $\psi = \sum_{N}\zeta_N\psi_N$ where $\{\zeta_N\}_{N=1}^\infty$ is a partition of unity of $\Omega\setminus\{0\}$ subordinate to the open cover $\{\Omega\setminus\bar{B}_{\varepsilon_N}\}$, and it is easy to see  that $\psi=\nabla\cdot\nabla\phi$ on any $\Omega\setminus\bar{B}_{\varepsilon_N}$, $N\geq 1$; hence $P$ is complete. Moreover, since $H_0^1(\Omega)$ is reflexive, as a closed subspace, $P$ is also reflexive.
\end{proof}

\section{Proof of Theorem~\ref{th_ex}}
Denote the energy in~\eqref{eq.min.reg} as $\mathcal{E}(\phi)$. Take a minimizing sequence $(\phi_m)\in P$ such that $\lim\limits_{m\rightarrow \infty} \mathcal{E}(\phi_m)=c:=\inf_{\phi\in P}\mathcal{E}(\phi)$. By~\cite{scherzer2009variational} It is easy to check that  $\mathcal{E}(\phi)$ is sequentially weakly lower semi-continuous. Hence, by the Eberlein-\v{S}mulian theorem~\cite{dunford1988linear}~(p.430), we can assume that $\phi_m$ weakly converges to $\phi^*$ for some $\phi_0\in P$. Since the lower level-set of $\mathcal{E}$ is weakly closed, $\mathcal{E}(\phi^*)<\infty$, thus  the minimizer of $\mathcal{E}$ exists in $P$. The uniqueness follows from the fact that $\mathcal{E}(\phi)$ is strictly convex in $\phi$.

\section{Derivation of the the boundary terms in (\ref{eq_boundaryterm})}\label{app_der}
To derive the first variation of~\eqref{eq.split.1} with symmetry, we take a  smooth  test function $\eta:[-L,L]\to\mathbb{R}$ such that $\eta(x)=\eta(-x)$ for $x\in [-L,L]$, which satisfies $\eta(L)=\eta_x(L)=0$. We next compute the perturbed energy  along $\eta$
\begin{align*}
	E(h)&=\frac{1}{2}\int_0^T\int_\Omega|u_t -L_u(\phi+h\eta)|^2\,d\bx\,dt+\beta\int_0^L ((\phi+h\eta)_{xx})^2\,d\bx+\lambda\int_0^L(\bb^k+(\phi+h\eta)_x-\psi^k)^2\,d\bx\\
	&=\frac{1}{2}\int_0^T\int_\Omega|u_t -L_u(\phi+h\eta)|^2\,d\bx\,dt+\beta\int_0^L \phi^2_{xx}\,d\bx+2\beta\int_0^L h\phi_{xx}\eta_{xx} \,d\bx+\beta\int_0^L h^2\eta^2_{xx}\\
	&+\lambda\int_0^L(\bb^k+\phi_x-\psi^k)^2\,d\bx+2\lambda\int_0^L(\bb^k+\phi_x-\psi^k)h\eta_x\,d\bx+\lambda\int_0^L h^2\eta_x^2\,d\bx
\end{align*}
Its first variation is
\begin{align*}
	\frac{d\,E(0)}{dh}&=\frac{1}{2}\frac{d}{dh}\Bigr|_{h=0}\int_0^T\int_\Omega|u_t -L_u(\phi+h\eta)|^2\,d\bx\,dt+2\beta\int_0^L \phi_{xx}\eta_{xx} \,d\bx+2\lambda\int_0^L(\bb^k+\phi_x-\psi^k)\eta_x\,d\bx\;.
\end{align*}
Since the first integral does not give the boundary term. We only focus on the last two integrals and apply integration by parts
\begin{align*}
	&2\beta\int_0^L \phi_{xx}\eta_{xx} \,d\bx+2\lambda\int_0^L(\bb^k+\phi_x-\psi^k)\eta_x\,d\bx\\
	&=2\beta\phi_{xx}\eta_x\Bigr|_{0}^L-2\beta\int_0^L\phi_{xxx}\eta_{x}\,d\bx+2\lambda (\bb^k+\phi_x-\psi^k)\eta\Bigr|_0^L-2\lambda\int_0^L(\bb_x^k+\phi_x-\psi_x^k)\eta\,d\bx\\
	&=2\beta\phi_{xx}\eta_x\Bigr|_{0}^L-2\beta\phi_{xxx}\eta\Bigr|_0^L+2\beta\int_0^L\phi_{xxxx}\eta\,d\bx+2\lambda (\bb^k+\phi_x-\psi^k)\eta\Bigr|_0^L-2\lambda\int_0^L(\bb_x^k+\phi_x-\psi_x^k)\eta\,d\bx.\\
\end{align*}
Hence we obtain the boundary terms as in~\eqref{eq_boundaryterm}.

\bibliographystyle{abbrv}
\bibliography{PDElit}

\begin{thebibliography}{10}

\bibitem{balague2013dimensionality}
D.~Balagu{\'e}, J.~Carrillo, T.~Laurent, and G.~Raoul.
\newblock Dimensionality of local minimizers of the interaction energy.
\newblock {\em Archive for Rational Mechanics and Analysis}, 209(3):1055--1088,
  2013.

\bibitem{bates2002spectral}
P.~W. Bates and F.~Chen.
\newblock Spectral analysis and multidimensional stability of traveling waves
  for nonlocal allen--cahn equation.
\newblock {\em Journal of mathematical analysis and applications},
  273(1):45--57, 2002.

\bibitem{behzadan2015multiplication}
A.~Behzadan and M.~Holst.
\newblock Multiplication in sobolev spaces, revisited.
\newblock {\em arXiv preprint arXiv:1512.07379}, 2015.

\bibitem{bergounioux2010second}
M.~Bergounioux and L.~Piffet.
\newblock A second-order model for image denoising.
\newblock {\em Set-Valued and Variational Analysis}, 18(3-4):277--306, 2010.

\bibitem{bertozzi2009blow}
A.~L. Bertozzi, J.~A. Carrillo, and T.~Laurent.
\newblock Blow-up in multidimensional aggregation equations with mildly
  singular interaction kernels.
\newblock {\em Nonlinearity}, 22(3):683, 2009.

\bibitem{bongini2017inferring}
M.~Bongini, M.~Fornasier, M.~Hansen, and M.~Maggioni.
\newblock Inferring interaction rules from observations of evolutive systems i:
  The variational approach.
\newblock {\em Mathematical Models and Methods in Applied Sciences},
  27(05):909--951, 2017.

\bibitem{bressloff2001traveling}
P.~C. Bressloff.
\newblock Traveling fronts and wave propagation failure in an inhomogeneous
  neural network.
\newblock {\em Physica D: Nonlinear Phenomena}, 155(1-2):83--100, 2001.

\bibitem{brunton2016discovering}
S.~L. Brunton, J.~L. Proctor, and J.~N. Kutz.
\newblock Discovering governing equations from data by sparse identification of
  nonlinear dynamical systems.
\newblock {\em Proceedings of the National Academy of Sciences},
  113(15):3932--3937, 2016.

\bibitem{caglioti2002homogeneous}
E.~Caglioti and C.~Villani.
\newblock Homogeneous cooling states are not always good approximations to
  granular flows.
\newblock {\em Archive for Rational Mechanics and Analysis}, 163(4):329--343,
  2002.

\bibitem{cai2013two}
X.~Cai, R.~Chan, and T.~Zeng.
\newblock A two-stage image segmentation method using a convex variant of the
  mumford--shah model and thresholding.
\newblock {\em SIAM Journal on Imaging Sciences}, 6(1):368--390, 2013.

\bibitem{carrillo2014derivation}
J.~A. Carrillo, Y.-P. Choi, and M.~Hauray.
\newblock The derivation of swarming models: mean-field limit and wasserstein
  distances.
\newblock In {\em Collective dynamics from bacteria to crowds}, pages 1--46.
  Springer, 2014.

\bibitem{carrillo2019aggregation}
J.~A. Carrillo, K.~Craig, and Y.~Yao.
\newblock Aggregation-diffusion equations: dynamics, asymptotics, and singular
  limits.
\newblock In {\em Active Particles, Volume 2}, pages 65--108. Springer, 2019.

\bibitem{carrillo2011global}
J.~A. Carrillo, M.~DiFrancesco, A.~Figalli, T.~Laurent, D.~Slep{\v{c}}ev,
  et~al.
\newblock Global-in-time weak measure solutions and finite-time aggregation for
  nonlocal interaction equations.
\newblock {\em Duke Mathematical Journal}, 156(2):229--271, 2011.

\bibitem{chan2007image}
T.~F. Chan, S.~Esedoglu, and F.~E. Park.
\newblock Image decomposition combining staircase reduction and texture
  extraction.
\newblock {\em Journal of Visual Communication and Image Representation},
  18(6):464--486, 2007.

\bibitem{chan1998total}
T.~F. Chan and C.-K. Wong.
\newblock Total variation blind deconvolution.
\newblock {\em IEEE transactions on Image Processing}, 7(3):370--375, 1998.

\bibitem{didas2009properties}
S.~Didas, J.~Weickert, and B.~Burgeth.
\newblock Properties of higher order nonlinear diffusion filtering.
\newblock {\em Journal of mathematical imaging and vision}, 35(3):208--226,
  2009.

\bibitem{dolak2005kinetic}
Y.~Dolak and C.~Schmeiser.
\newblock Kinetic models for chemotaxis: Hydrodynamic limits and
  spatio-temporal mechanisms.
\newblock {\em Journal of mathematical biology}, 51(6):595--615, 2005.

\bibitem{donoho1995noising}
D.~L. Donoho.
\newblock De-noising by soft-thresholding.
\newblock {\em IEEE transactions on information theory}, 41(3):613--627, 1995.

\bibitem{d2006self}
M.~R. D’Orsogna, Y.-L. Chuang, A.~L. Bertozzi, and L.~S. Chayes.
\newblock Self-propelled particles with soft-core interactions: patterns,
  stability, and collapse.
\newblock {\em Physical review letters}, 96(10):104302, 2006.

\bibitem{dunford1988linear}
N.~Dunford and J.~T. Schwartz.
\newblock {\em Linear operators, part 1: general theory}, volume~10.
\newblock John Wiley \& Sons, 1988.

\bibitem{duong2015spherically}
T.~Duong.
\newblock Spherically symmetric multivariate beta family kernels.
\newblock {\em Statistics \& Probability Letters}, 104:141--145, 2015.

\bibitem{fetecau2011swarm}
R.~C. Fetecau, Y.~Huang, and T.~Kolokolnikov.
\newblock Swarm dynamics and equilibria for a nonlocal aggregation model.
\newblock {\em Nonlinearity}, 24(10):2681, 2011.

\bibitem{glowinski2019finite}
R.~Glowinski, H.~Liu, S.~Leung, and J.~Qian.
\newblock A finite element/operator-splitting method for the numerical solution
  of the two dimensional elliptic monge--amp{\`e}re equation.
\newblock {\em Journal of Scientific Computing}, 79(1):1--47, 2019.

\bibitem{glowinski2017splitting}
R.~Glowinski, S.~J. Osher, and W.~Yin.
\newblock {\em Splitting methods in communication, imaging, science, and
  engineering}.
\newblock Springer, 2017.

\bibitem{goldstein2009split}
T.~Goldstein and S.~Osher.
\newblock The split bregman method for l1-regularized problems.
\newblock {\em SIAM journal on imaging sciences}, 2(2):323--343, 2009.

\bibitem{halliday2013fundamentals}
D.~Halliday, R.~Resnick, and J.~Walker.
\newblock {\em Fundamentals of Physics}.
\newblock John Wiley \& Sons, 2013.

\bibitem{he2020robust}
Y.~He, S.~H. Kang, W.~Liao, H.~Liu, and Y.~Liu.
\newblock Robust identification of differential equations by numberical
  techniques from a single set of noisy observation.
\newblock {\em arXiv preprint arXiv:2006.06557}, 2020.

\bibitem{he2020curvature}
Y.~He, S.~H. Kang, and H.~Liu.
\newblock Curvature regularized surface reconstruction from point clouds.
\newblock {\em SIAM Journal on Imaging Sciences}, 13(4):1834--1859, 2020.

\bibitem{holm2006formation}
D.~D. Holm and V.~Putkaradze.
\newblock Formation of clumps and patches in self-aggregation of finite-size
  particles.
\newblock {\em Physica D: Nonlinear Phenomena}, 220(2):183--196, 2006.

\bibitem{huang2019learning}
H.~Huang, J.-G. Liu, and J.~Lu.
\newblock Learning interacting particle systems: Diffusion parameter estimation
  for aggregation equations.
\newblock {\em Mathematical Models and Methods in Applied Sciences},
  29(01):1--29, 2019.

\bibitem{huang2010self}
Y.~Huang and A.~L. Bertozzi.
\newblock Self-similar blowup solutions to an aggregation equation in
  {$\mathbb{R}^n$}.
\newblock {\em SIAM Journal on Applied Mathematics}, 70(7):2582--2603, 2010.

\bibitem{kang2019ident}
S.~H. Kang, W.~Liao, and Y.~Liu.
\newblock {IDENT}: Identifying differential equations with numerical time
  evolution.
\newblock {\em Journal of Scientific Computing}, 2021.

\bibitem{keller1970initiation}
E.~F. Keller and L.~A. Segel.
\newblock Initiation of slime mold aggregation viewed as an instability.
\newblock {\em Journal of theoretical biology}, 26(3):399--415, 1970.

\bibitem{lancaster1981surfaces}
P.~Lancaster and K.~Salkauskas.
\newblock Surfaces generated by moving least squares methods.
\newblock {\em Mathematics of computation}, 37(155):141--158, 1981.

\bibitem{leung2006adjoint}
S.~Leung and J.~Qian.
\newblock An adjoint state method for three-dimensional transmission traveltime
  tomography using first-arrivals.
\newblock {\em Communications in Mathematical Sciences}, 4(1):249--266, 2006.

\bibitem{leung2021level}
S.~Leung, J.~Qian, and J.~Hu.
\newblock A level-set adjoint-state method for transmission traveltime
  tomography in irregular domains.
\newblock {\em SIAM Journal on Scientific Computing}, 43(3):A2352--A2380, 2021.

\bibitem{li2020three}
X.~Li, X.~Yang, and T.~Zeng.
\newblock A three-stage variational image segmentation framework incorporating
  intensity inhomogeneity information.
\newblock {\em SIAM Journal on Imaging Sciences}, 13(3):1692--1715, 2020.

\bibitem{liu2021operator}
H.~Liu, X.-C. Tai, and R.~Glowinski.
\newblock An operator-splitting method for the gaussian curvature
  regularization model with applications in surface smoothing and imaging.
\newblock {\em arXiv preprint arXiv:2108.01914}, 2021.

\bibitem{liu2020color}
H.~Liu, X.-C. Tai, R.~Kimmel, and R.~Glowinski.
\newblock A color elastica model for vector-valued image regularization.
\newblock {\em SIAM Journal on Imaging Sciences}, 14(2):717--748, 2021.

\bibitem{lu2021learning}
F.~Lu, M.~Maggioni, and S.~Tang.
\newblock Learning interaction kernels in heterogeneous systems of agents from
  multiple trajectories.
\newblock {\em Journal of Machine Learning Research}, 22(32):1--67, 2021.

\bibitem{lu2019nonparametric}
F.~Lu, M.~Zhong, S.~Tang, and M.~Maggioni.
\newblock Nonparametric inference of interaction laws in systems of agents from
  trajectory data.
\newblock {\em Proceedings of the National Academy of Sciences},
  116(29):14424--14433, 2019.

\bibitem{maso2009higher}
G.~D. Maso, I.~Fonseca, G.~Leoni, and M.~Morini.
\newblock A higher order model for image restoration: the one-dimensional case.
\newblock {\em SIAM Journal on Mathematical Analysis}, 40(6):2351--2391, 2009.

\bibitem{morale2005interacting}
D.~Morale, V.~Capasso, and K.~Oelschl{\"a}ger.
\newblock An interacting particle system modelling aggregation behavior: from
  individuals to populations.
\newblock {\em Journal of mathematical biology}, 50(1):49--66, 2005.

\bibitem{motsch2014heterophilious}
S.~Motsch and E.~Tadmor.
\newblock Heterophilious dynamics enhances consensus.
\newblock {\em SIAM review}, 56(4):577--621, 2014.

\bibitem{papafitsoros2014combined}
K.~Papafitsoros and C.-B. Sch{\"o}nlieb.
\newblock A combined first and second order variational approach for image
  reconstruction.
\newblock {\em Journal of mathematical imaging and vision}, 48(2):308--338,
  2014.

\bibitem{parrish1997animal}
J.~K. Parrish and W.~M. Hamner.
\newblock {\em Animal Groups in Three Dimensions: How Species Aggregate}.
\newblock Cambridge University Press, 1997.

\bibitem{pratt2005behavioral}
S.~C. Pratt.
\newblock Behavioral mechanisms of collective nest-site choice by the ant
  temnothorax curvispinosus.
\newblock {\em Insectes Sociaux}, 52(4):383--392, 2005.

\bibitem{reynolds1987flocks}
C.~W. Reynolds.
\newblock Flocks, herds and schools: A distributed behavioral model.
\newblock In {\em Proceedings of the 14th annual conference on Computer
  graphics and interactive techniques}, pages 25--34, 1987.

\bibitem{rudin1992nonlinear}
L.~I. Rudin, S.~Osher, and E.~Fatemi.
\newblock Nonlinear total variation based noise removal algorithms.
\newblock {\em Physica D: nonlinear phenomena}, 60(1-4):259--268, 1992.

\bibitem{rudy2017data}
S.~H. Rudy, S.~L. Brunton, J.~L. Proctor, and J.~N. Kutz.
\newblock Data-driven discovery of partial differential equations.
\newblock {\em Science Advances}, 3(4):e1602614, 2017.

\bibitem{schaeffer2017learning}
H.~Schaeffer.
\newblock Learning partial differential equations via data discovery and sparse
  optimization.
\newblock {\em Proceedings of the Royal Society A: Mathematical, Physical and
  Engineering Sciences}, 473(2197):20160446, 2017.

\bibitem{scherzer2009variational}
O.~Scherzer, M.~Grasmair, H.~Grossauer, M.~Haltmeier, and F.~Lenzen.
\newblock {\em Variational Methods in Imaging}, volume 167.
\newblock Springer Science \& Business Media, 2008.

\bibitem{sei1995convergent}
A.~Sei and W.~W. Symes.
\newblock Convergent finite-difference traveltime gradient for tomography.
\newblock In {\em SEG Technical Program Expanded Abstracts 1995}, pages
  1258--1261. Society of Exploration Geophysicists, 1995.

\bibitem{strong2003edge}
D.~Strong and T.~Chan.
\newblock Edge-preserving and scale-dependent properties of total variation
  regularization.
\newblock {\em Inverse problems}, 19(6):S165, 2003.

\bibitem{taillandier2009first}
C.~Taillandier, M.~Noble, H.~Chauris, and H.~Calandra.
\newblock First-arrival traveltime tomography based on the adjoint-state
  method.
\newblock {\em Geophysics}, 74(6):WCB1--WCB10, 2009.

\bibitem{topaz2004swarming}
C.~M. Topaz and A.~L. Bertozzi.
\newblock Swarming patterns in a two-dimensional kinematic model for biological
  groups.
\newblock {\em SIAM Journal on Applied Mathematics}, 65(1):152--174, 2004.

\bibitem{topaz2006nonlocal}
C.~M. Topaz, A.~L. Bertozzi, and M.~A. Lewis.
\newblock A nonlocal continuum model for biological aggregation.
\newblock {\em Bulletin of mathematical biology}, 68(7):1601, 2006.

\bibitem{you2020data1}
H.~You, Y.~Yu, S.~Silling, and M.~D'Elia.
\newblock Data-driven learning of nonlocal models: from high-fidelity
  simulations to constitutive laws.
\newblock {\em arXiv preprint arXiv:2012.04157}, 2020.

\bibitem{you2020data}
H.~You, Y.~Yu, N.~Trask, M.~Gulian, and M.~D'Elia.
\newblock Data-driven learning of robust nonlocal physics from high-fidelity
  synthetic data.
\newblock {\em arXiv preprint arXiv:2005.10076}, 2020.

\end{thebibliography}
	
\end{document}